\documentclass[11pt]{amsart}

\usepackage[text={14.06cm,22.84cm},centering]{geometry}

\usepackage{xcolor}
\usepackage{esint,amssymb}
\usepackage{graphicx}
\usepackage{mathtools}
\usepackage[colorlinks=true, pdfstartview=FitV, linkcolor=blue, citecolor=blue, urlcolor=blue,pagebackref=false]{hyperref}
\usepackage{microtype}

\usepackage{bm}
\usepackage{scalerel} 
\usepackage{mathrsfs}
\usepackage[font={footnotesize}]{caption}

\usepackage{enumitem}

\usepackage{comment}

\usepackage{scalerel,stackengine}
\stackMath
\newcommand\reallywidehat[1]{%
\savestack{\tmpbox}{\stretchto{%
  \scaleto{%
    \scalerel*[\widthof{\ensuremath{#1}}]{\kern-.6pt\bigwedge\kern-.6pt}%
    {\rule[-\textheight/2]{1ex}{\textheight}}
  }{\textheight}%
}{0.5ex}}%
\stackon[1pt]{#1}{\tmpbox}%
}
\definecolor{darkgreen}{rgb}{0,0.5,0}
\definecolor{darkblue}{rgb}{0,0,0.7}
\definecolor{darkred}{rgb}{0.9,0.1,0.1}
\definecolor{lightblue}{rgb}{0,0.51,1}

\usepackage[english]{babel}

\newtheorem{theorem}{Theorem}
\newtheorem{proposition}{Proposition}
\newtheorem{lemma}[proposition]{Lemma}

\theoremstyle{remark}
\newtheorem{remark}[proposition]{Remark}

\theoremstyle{definition}
\newtheorem{definition}[proposition]{Definition}

\numberwithin{equation}{section}
\numberwithin{proposition}{section}

\newcommand{\R}{\mathbb{R}}

\newcommand{\ep}{\varepsilon}

\newcommand{\Rey}{\texttt{Re}}

\newcommand{\Ro}{\texttt{Ro}}
\newcommand{\Ma}{\texttt{Ma}}
\newcommand{\Fr}{\texttt{Fr}}

\renewcommand{\subset}{\subseteq}

\renewcommand{\bar}{\overline}

\renewcommand{\div}{\divg}

\allowdisplaybreaks

\newcommand{\tsl}{\textsl}

\newcommand{\mbb}{\mathbb}

\newcommand{\mc}{\mathcal}

\newcommand{\veps}{\varepsilon}

\newcommand{\vphi}{\varphi}

\newcommand{\oline}{\overline}

\newcommand{\ra}{\rightarrow}

\newcommand{\g}{\gamma}

\newcommand{\z}{\zeta}

\newcommand{\de}{\delta}

\newcommand{\Q}{\mathbb{Q}}

\renewcommand{\div}{{\rm div}\,}

\newcommand{\curl}{{\rm curl}\,}

\newcommand{\Id}{{\rm Id}\,}

\def\d{\partial}

\begin{document}
	
\author[E. Bocchi]{Edoardo Bocchi}
\address[E. Bocchi]{Universidad de Sevilla, Departamento An\'alisis Matem\'atico \& Instituto de Matem\'aticas Universidad de Sevilla, Avenida Reina Mercedes, 41012 Sevilla, Espa\~{n}a}
\email{ebocchi@us.es}

\author[F. Fanelli]{Francesco Fanelli}
\address[F. Fanelli]{Univ. Lyon, Universit\'e Claude Bernard Lyon 1, CNRS UMR 5208, Institut Camille Jordan, 43 blvd. du 11 novembre 1918, F-69622 Villeurbanne cedex, France}
\email{fanelli@math.univ-lyon1.fr}

\author[C. Prange]{Christophe Prange}
\address[C. Prange]{Universit\'e de Bordeaux, CNRS UMR 5251, Institut de Math\'ematiques de Bordeaux, 351 Cours de la Liberation, 33405 Talence cedex, France}
\email{christophe.prange@math.u-bordeaux.fr}

\title[Anisotropy and stratification in compressible fluids]{Anisotropy and stratification effects in the dynamics of \\ fast rotating compressible fluids}

\date{\today}

\begin{abstract}
 The primary goal of this paper is to develop robust methods to handle two ubiquitous features appearing in the modeling of geophysical flows: (i) the \emph{anisotropy} of the viscous stress tensor,
(ii) \emph{stratification} effects. We focus on the barotropic Navier-Stokes equations with Coriolis and gravitational forces.
Two results are the main contributions of the paper. Firstly, we establish a local well-posedness result for finite-energy solutions, \textsl{via} a maximal regularity approach.
This method allows us to circumvent the use of the effective viscous flux, which plays a key role in the weak solutions theories of Lions-Feireisl and Hoff, but seems to be restricted
to isotropic viscous stress tensors. Moreover, our approach is sturdy enough to take into account non constant reference density states; this is crucial when dealing with stratification effects. Secondly, we study the structure of the solutions to the previous model in the regime when the Rossby, Mach and Froude numbers are of the same order of magnitude.
We prove an error estimate on the relative entropy between actual solutions and their approximation by a large-scale quasi-geostrophic
flow supplemented with Ekman boundary layers. Our analysis holds for a large class of barotropic pressure laws. 
\end{abstract}

\maketitle

\section{Introduction} \label{s:intro}

This paper is devoted to the study of a class of barotropic Navier-Stokes systems. Our focus is on developing robust methods to handle two ubiquitous effects appearing in the modeling
of geophysical fluids: (i) the strong \emph{anisotropy} of the viscous stress tensor, (ii) \emph{stratification} effects. As for the first aspect, we consider systems with anisotropic viscosity tensors
\begin{equation}\label{e.anisvisco}
\Delta_{\mu,\ep}:=\mu\Delta_h+\ep\partial_3^2=\mu(\partial_1^2+\partial_2^2)+\ep\partial_3^2\,,
\end{equation}
where the parameters $\mu$ and $\veps$ are dimensionless numbers such that $\ep\ll\mu$. Concerning the second point, we consider highly rotating fluids with a strong Coriolis force $\frac1\ep e_3\times \rho u$ and a strong gravitational potential $\frac{1}{\ep^2}\rho \nabla G=-\frac{1}{\ep^2}\rho e_3$, where $u=u(x,t)\in\R^3$ denotes the velocity field of the fluid, $\rho=\rho(x,t)\in\R$ represents the density of the fluid and $e_3=(0,0,1)^T$ is the unit vertical vector. 
We handle barotropic fluids, so the pressure of the fluid is assumed to be a function of the density only, i.e. $P=P(\rho)$; see \eqref{hyp:p} for precise assumptions on the pressure law. We consider the simplest possible geometrical set-ups and scalings enabling to study non trivial phenomena, such as boundary layers and vertical stratification. \\
\indent
Our goal in this paper is twofold. 
The first part of the paper is concerned with the existence and uniqueness of strong solutions for the barotropic Navier-Stokes equations with potential force $\nabla G$:
\begin{equation}
\label{eq:wp_nse}
\left\{ 
\begin{aligned}
&\partial_t\rho+\nabla\cdot(\rho u)=0,\\ 
& \partial_t(\rho u)+\nabla\cdot(\rho u\otimes u)+\nabla P(\rho)=\Delta_{\mu,\ep} u+\lambda\nabla(\nabla\cdot u)+\rho \nabla G\,.
\end{aligned}
\right.
\end{equation}
The term $\nabla G$ is responsible for stratification effects, thus the equilibrium density $\bar\rho$ of the previous system becomes non-constant.
As a consequence, there are two main difficulties to handle in the well-posedness study for system \eqref{eq:wp_nse}: on the one hand, 
the strong anisotropy of the viscous tensor, see \eqref{e.anisvisco}; on the other hand, the fact that $\bar\rho$ introduces variable coefficients in the equations (see more details below). 
Our main result in this direction is Theorem \ref{th:wp} on page \pageref{th:wp}.
We set system \eqref{eq:wp_nse} in the simple space-time domain $\R^3\times (0,T)$, with $T>0$; domains with boundaries, such as $(\R^2\times(0,1))\times (0,T)$,
may be handled by the same method, at the price of more technical difficulties. Here we work with fixed values of the parameters $(\mu,\lambda,\ep)$, and do not keep track
of how the estimates depend on  them: the existence of solutions to \eqref{eq:wp_nse} is a challenge in itself.
Our well-posedness result still holds if one  incorporates a Coriolis force $\rho e_3\times u$ in the left hand side of \eqref{eq:wp_nse}. \\
\indent
The second part of the paper is devoted to the asymptotic analysis of the barotropic Navier-Stokes equations in presence of fast rotation and gravitational stratification.
We consider the following system:
\begin{equation}
\label{e.nse}
\left\{\!
\begin{aligned}
&\partial_t\rho+\nabla\cdot(\rho u)=0,\\ 
& \partial_t(\rho u)+\nabla\cdot(\rho u\otimes u)+\frac\rho\ep e_3\times u+\frac{\nabla P(\rho)}{\ep^2}=\Delta_{\mu,\ep} u+\lambda\nabla(\nabla\cdot u)+\frac{\rho}{\ep^2} \nabla G.
\end{aligned}
\right.
\end{equation}
set on the strip $(\R^2\times(0,1))\times (0,T)$, with the no-slip boundary conditions
\begin{equation}\label{hyp:bound-cond}
u=0\qquad\mbox{ at }\quad x_3=0,\,1\,.
\end{equation}
Here $G$ is the gravitational potential, i.e. $G(x_3)=-x_3$. 
We describe the structure of weak solutions 
in the limit $\ep\rightarrow 0$ for well-prepared data, analyze the Ekman boundary layers and their effect on the limit quasi-geostrophic flow, and prove quantitative bounds based on
relative entropy estimates. Our main result in this direction is Theorem \ref{prop.main} on page \pageref{prop.main}. Our results hold for a large class of monotone pressure laws. Here our focus is on the asymptotic behavior for a family of weak solutions, under the assumption that such global-in-time weak solutions exist. \\
\indent
The two parts are connected. Indeed, the quantitative stability estimates obtained in the second part lay the ground for a large-time well-posedness result for system \eqref{e.nse}
in the limit when $\ep\rightarrow 0$. Inspired by previous results for incompressible flows  \cite{Gall,Char,PRST,Scrob}, we believe that strong solutions may be constructed
for large data close to the two-dimensional limit quasi-geostrophic flow, by using a variation of the well-posedness result of the first part.
Though, dealing with the compressible system \eqref{e.nse} requires much more work, which is left for a subsequent paper.

\subsection{Modelling of geophysical flows}

We consider mid-latitudes and high-latitudes motions of fast rotating compressible fluids, typically the Earth's atmosphere or oceans.
System  \eqref{e.nse} is a particular case of the general non-dimensional system
\begin{equation}
\label{e.nsenondim}
\left\{ 
\begin{aligned}
&\partial_t\rho+\nabla\cdot(\rho u)=0,\\ 
&\partial_t(\rho u)+\nabla\cdot(\rho u\otimes u)-\frac{1}{\Rey_h}\Delta_h u-\frac{1}{\Rey_3}\partial_3^2 u-\frac{1}{\Rey}\nabla\nabla\cdot u\\
&\qquad\qquad\qquad\qquad\qquad+\frac{1}{\Ro}e_3\times (\rho u)+\frac{1}{\Ma^2}\nabla P(\rho)=\frac{1}{\Fr^2}\rho\nabla G\,.
\end{aligned}
\right.
\end{equation}
As usual, the Reynolds number measures the ratio of inertial forces to viscous forces. System \eqref{e.nsenondim} modeling large-scales geophysical flows has, in particular,
a horizontal Reynolds number $\Rey_h=\frac{UL}{\nu_h}$ and a vertical Reynolds number $\Rey_3=\frac{UL}{\nu_3}$ that may be of different orders of magnitude;
see below the comments about the anisotropy in the viscous stress tensor. Here $U$ represents the typical speed of the flow, $L$ the typical length, $\nu_h$ is the horizontal
viscosity and $\nu_3$ the vertical one. The Mach number is defined as $\Ma=\frac{U}{c}$, where the constant $c$ 
 is the propagation speed of 
 acoustic waves. 
For strong jet streams near the tropopause $U=50\, m\cdot s^{-2}$, which corresponds to $\Ma=0.15$, see \cite{Klein05}. The Rossby number, defined as 
$\Ro=\frac{U}{2\Omega L}$,
measures the effect of the Earth's rotation; here $\Omega=7.3\cdot 10^{-5}s^{-1}$ is the
module of the Earth's angular velocity. 
In the case of the Gulf Stream, the length $L=100\, km$ and $U=1\, m\cdot s^{-1}$ are smaller than the typical oceanic scales, the Rossby number is about $\Ro=0.07$. Notice that here we consider the f-plane approximation of the Coriolis force.

\noindent\underline{Stratification.} The gravitational force deriving from the geopotential {$G=-g x_3$}
tends to lower regions of fluid with higher density and raise regions of fluid with
lower density.  
In the equilibrium configuration, the density profile decreases with respect to the vertical direction. The Froude number is then defined as 
$\Fr=\frac{U}{\sqrt{gH}}$, 
where $g=9.81\, m\cdot s^{-2}$ is the acceleration of gravity. It measures the ratio of inertial forces 
of a fluid element to its weight. The centrifugal force also derives from a potential. It is often neglected in models for the atmosphere \cite{CR, F-N_CPDE},
but in certain regimes it can have a dramatic effect. The mathematical analysis of the centrifugal force poses different challenges that we do not dwell upon in the present work.

\noindent\underline{Anisotropy.} In general the horizontal and vertical viscosities are not equal, in particular when dealing with large-scale motions of geophysical flows. 
For instance, in the ocean the horizontal turbulent viscosity $\nu_h$ ranges from $10^3$ to $10^8\, cm^2\cdot s^{-1}$, while the vertical viscosity $\nu_3$ is much smaller and ranges from $1$
to $10^3\, cm^2s^{-1}$. A justification of this fact can be seen in: (i) the anisotropy between
the horizontal and the vertical scales of the flows, and (ii) the stabilizing effect of the Coriolis force, which makes the large-scale motion be almost two-dimensional
(see also Section \ref{sec.quant}). Hence, a frequent (and crude) modeling assumption is to suppose that the diffusion in the vertical direction is much weaker
than the horizontal viscosity, which is enhanced by turbulent phenomena. For further insights about the physics of anisotropic diffusion, we refer to \cite{TZ}, in particular to equation (2.122),
to Chapter 4 of \cite{Ped}, to \cite{B-D-GV} and \cite{C-D-G-G_2002}.

\noindent\underline{Scaling.} In \eqref{e.nse}, the dimensionless number $\ep$ denotes the Rossby number, which measures
the strength of the rotation. We consider the scaling where the Rossby, the Mach and the Froude numbers are of the same order of magnitude, i.e. $\Ro=\Ma=\Fr=\ep$.
This is the richest scaling, since the effects due to the rotation are in balance with the compressible and gravitational effects.
Notice however that other scalings are considered in the physical literature \cite{Klein10}, for application to meteorology, as well as in the mathematical literature, see e.g. \cite{F-G-GV-N}, \cite{F-N_CPDE}.

\noindent\underline{Ekman layers.} Ekman boundary layers are regions near horizontal boundaries with no-slip boundary condition where viscous effects balance the Coriolis force.
The thickness of these boundary layers (see Part I of \cite{C-D-G-G}) is
\begin{equation*}
\delta_E=\left(\frac{\nu_3}{\Omega}\right)^\frac12,
\end{equation*}
which does not depend on the velocity. Remark that the faster
the rotation, the smaller is the layer affected by viscosity. See also Chapter 8 of \cite{CR}.

\noindent\underline{Pressure laws.} 
We consider barotropic flows, for which the pressure is a function of the density only.
A typical example is that of Boyle's laws $P(\rho)=a\rho^\gamma$, with $a>0$ and $\gamma\geq 1$. 
For the precise definition of the pressure law, we refer to \eqref{hyp:p}.

\subsection{Mathematical challenges related to anisotropy and stratification}
\label{sec.chall}

We outline here some aspects of the study of compressible viscous fluids. We focus on two points in particular: (i) well-posedness and the difficulties related to the anisotropy in the viscosity, (ii) asymptotic analysis in the presence of stratification.

\noindent\underline{Well-posedness.} For fluids modelled by the incompressible Navier-Stokes equations,
the fact that the viscous stress tensor is isotropic or anisotropic does not affect the well-posedness theory. One can prove the existence of weak (Leray-Hopf),
mild or strong solutions \emph{regardless} of the structure of the viscous stress tensor. \\
\indent
For compressible fluids modelled by the compressible Navier-Stokes system though, anisotropy represents a major hurdle.
Interestingly, the obstacle has similar roots for several well-posedness theories of weak solutions.  \\
\indent
In the isotropic case, $\mu=\ep$ in \eqref{eq:wp_nse}, one has a pointwise relation between the pressure and the divergence of $u$. Indeed, applying the divergence to the momentum equation
and using the algebraic relation $\nabla\cdot(\nabla\nabla\cdot u)=\Delta\nabla\cdot u=\nabla\cdot\Delta u$, one gets 
\begin{equation}\label{e.eqeffflux}
\Delta\big((\mu+\lambda)\nabla\cdot u-P(\rho)+P(1)\big)=\nabla\cdot\big(\rho\partial_tu+\rho u\cdot\nabla u\big).
\end{equation}
This suggests to define the following quantity
\begin{equation*}
F:=(\mu+\lambda)\nabla\cdot u-P(\rho)+P(1),
\end{equation*}
which is dubbed the \emph{effective viscous flux}. According to \eqref{e.eqeffflux}, one obtains the relation 
\begin{equation}\label{e.eqF}
\Delta F=\nabla\cdot\big(\rho\partial_tu+\rho u\cdot\nabla u\big).
\end{equation}  
The property \eqref{e.eqF} is key to the existence of global finite-energy weak solutions in the Lions \cite{PLLions} and Feireisl \cite{F-N,N-S} theory on the one hand,
and to the existence of global weak solutions {with bounded density} in the Hoff theory \cite{Hoff} (see also \cite{D-F-P}). \\

\indent
In the anisotropic case $\mu\neq \ep$, the above analysis breaks down, and the relation between $\nabla\cdot u$ and $P(\rho)$ becomes nonlocal. Indeed, \eqref{e.eqF} becomes 
\begin{equation*}
(\Delta_{\mu,\ep}+\lambda\Delta)\nabla\cdot u-\Delta P=\nabla\cdot\big(\rho\partial_tu+\rho u\cdot\nabla u\big),
\end{equation*}
so that the definition of a modified effective flux should read
\begin{equation*}
F_{ani}:=\Delta^{-1}(\Delta_{\mu,\ep}+\lambda\Delta)\nabla\cdot u-P(\rho)+P(1).
\end{equation*}
The nonlocal operator $\Delta^{-1}(\Delta_{\mu,\ep}+\lambda\Delta)$ changes the picture dramatically. In view of the existence of Hoff-type solutions with bounded density
\cite{Hoff,D-F-P}, one of the major flaws of $F_{ani}$ is the lack of boundedness of $\Delta^{-1}(\Delta_{\mu,\ep}+\lambda\Delta)$ on $L^\infty$.
Similarly,
for the existence of global finite-energy weak solutions, the nonlocality is a major obstacle to getting the compactness of a sequence of approximate solutions; see \cite{B-D-GV,B-J}.
In the breakthrough work of Bresch and Jabin \cite{B-J}, a totally new compactness criterion was proved that enables to prove the existence of global weak solutions to the compressible system \eqref{eq:wp_nse} in the anisotropic case. There remains a restriction that $|\mu-\ep|<\lambda-\frac\mu3$ which is compatible with the modeling of large-scale geophysical flows. A more important limitation of the result in \cite{B-J} is on the pressure law: $\gamma\geq 2 + \frac{\sqrt{10}}2$, where $\gamma$ is defined in \eqref{hyp:p}. On that subject Bresch and Burtea proved recently in \cite{Bre-Bur} the global existence of weak solutions for the quasi-stationary compressible Stokes equations and to the stationary compressible Navier-Stokes system \cite{Bre-Bur2}
with an anisotropic viscous tensor. Their approach is based on the control of defect measures.

The anisotropy prompts us to consider the framework of strong solutions, in particular those with minimal regularity assumptions as in \cite{D-F-P}. In that perspective, a further challenge for the well-posedness theory is the presence of non constant reference density states, due to the gravitational term. Indeed, the fact that the densities are perturbations 
of a \emph{constant} state plays a major role in the analysis of \cite{D-F-P}. This question seems to remain broadly unexplored in general. In our work, we are able to incorporate non constant reference density states in the approach of \cite{D-F-P}.

\noindent\underline{Asymptotic analysis.} 
We review here some of the literature concerned with the quantitative analysis of viscous barotropic fluids in high rotation (for an overview of this topic for
incompressible fluids, we refer to \cite{C-D-G-G}). For systems of the type \eqref{e.nse}, results on the combined low Mach and low Rossby limits were obtained in several directions:
well-prepared data (data close to the kernel of the penalization operator, see \eqref{e.wellpre} for what it means in our context), ill-prepared data, slip or no-slip boundary data, with or without stratification (centrifugal force or gravitational force), different scalings regimes. We do not attempt to be exhaustive here, but select some works which are relevant to our study here. \\
\indent
In the well-prepared case, paper \cite{B-D-GV} studies the limit for the same scaling as in \eqref{e.nse}, with no-slip boundary conditions and $G=0$. They study the Ekman layers and prove
stability estimates in the limit $\ep\rightarrow 0$ in the case of pressure laws with $\gamma=2$. As far as we know, this seems to be the only work concerned with the study of Ekman
layers for compressible fluids. Extending it to more general pressure laws, and taking into account gravitational stratification effect is an obvious motivation for our present work.
The analysis with a gravitational potential was considered in \cite{F-L-N}, for fluids slipping on the boundary, using the relative entropy method. \\

\indent
In the ill-prepared case, the fact that the initial data is away from the kernel of the penalization operator is responsible for the propagation of high frequency acoustic-Poincar\'e waves.
Different scaling regimes of $\Ma$ and $\Ro$ were analyzed in the works \cite{F-G-N,F-G-GV-N,F-N_CPDE}. All these works are concerned with fluids satisfying the slip boundary condition, hence
no boundary layers are needed in the asymptotic expansions. Let us point out that the work \cite{F-G-GV-N} manages to handle the centrifugal force.

\subsection{Novelty of our results}
\label{sec.outline}

We comment here on the main theorems of the paper.

\noindent\underline{Well-posedness of system \eqref{eq:wp_nse}: Theorem \ref{th:wp} on page \pageref{th:wp}.} We prove the short-time existence of finite-energy solutions
to system \eqref{eq:wp_nse}. We introduce a simple and sturdy method based on a priori bounds obtained via maximal regularity estimates,
following the approach of \cite{D-F-P, Shi18}. However, our techniques are robust enough to deal with both effects mentioned above: (i) the anisotropy in the viscous stress tensor,
(ii) the presence of a non-constant reference density state. Thus, our result represents a generalization of \cite{D-F-P, Shi18} in both directions. 
However, as already explained in Subsection \ref{sec.chall}, the anisotropy makes it impossible to use Hoff's effective viscous flux as in \cite{D-F-P},
so we need to look for higher regularity estimates for the density,
in order to get compactness for passing to the limit in the pressure term. For this reason, we need to require some regularity on the initial data: roughly,
the initial density $\rho_{in}$ is sufficiently well-localised around the reference profile $\bar\rho$, namely $\rho_{in}-\bar\rho\in H^2$, while the initial velocity $u_{in}$ is taken in $H^{3/2+}$.
So, the solutions that we construct are strong solutions with finite energy; roughly, they are halfway between the Hoff solutions \cite{Hoff,D-F-P} with bounded density and
the strong solutions of Matsumura-Nishida \cite{M-N} (which require $\rho_{in}-1\,\in\,H^3$). \\
\indent
We emphasize that the initial density $\rho_{in}$ can be taken close to an arbitrary reference density profile $\bar\rho=\bar\rho(x_3)\in W^{3,\infty}(\R)$. Thus, setting
$G=G(x_3)=H'\big(\bar\rho(x_3)\big)$, with $H$ defined in \eqref{def:H} below, we infer that $\bar\rho$ is a static state of system \eqref{eq:wp_nse},
namely $\bar\rho$ satisfies the logistic equation
\begin{equation}\label{e.eqpotG}
\nabla P(\bar\rho)=\bar\rho\nabla G.
\end{equation} 
However, the smallness condition rests only on the quantity $\|\rho_{in}-\bar\rho\|_{L^\infty}$, and not on higher-order derivatives. Finally, we point out that the assumption
$\bar\rho=\bar\rho(x_3)$ is made only for modelling purposes (we have in mind the case when $G$ is the gravitational potential), but our method works also in
the more general situation $\bar\rho=\bar\rho(x),\,x\in\R^3$. So, our theorem opens the way to achieving the well-posedness of systems with
stratification effects, such as \eqref{e.nse}, in the strip $\R^2\times (0,1)$
with the physical gravitational potential $G=-x_3$.

\noindent\underline{Asymptotic analysis of system \eqref{e.nse}: Theorem \ref{prop.main} on page \pageref{prop.main}.} We build an asymptotic expansion for
the solutions of the compressible system \eqref{e.nse} when $\ep\rightarrow 0$ and prove quantitative estimates on the errors.
We consider well-prepared initial data, i.e. close to the kernel of the penalization operator. The scaling $\Ro=\Ma=\Fr=\ep$ which is considered in \eqref{e.nse} is the richest scaling,
in the sense that the Coriolis force, the pressure and the gravitational force balance each other. In addition, we analyze the effect of boundary layers on the limiting two-dimensional
quasi-geostrophic equation. This limit equation, see \eqref{e.qg} below, represents the large-scale dynamics of the bulk flow.
It is a two-dimensional incompressible Navier-Stokes equation written in terms of the stream function $Q$, where $u=\nabla^\perp Q$ is the limit velocity.
Frictional effects dissipate energy in the Ekman boundary layer flow, so a damping term appears in \eqref{e.qg}. Such effects were already pointed out for incompressible
\cite{Masm_2000,C-D-G-G_2002,C-D-G-G,GV03} or compressible \cite{B-D-GV} fluids in high rotation. \\
\indent
As far as we know, that paper \cite{B-D-GV} is the only work dealing with the asymptotic analysis of compressible fluids in high rotation in the presence of Ekman layers.
Our result extends the state of the art in two main directions: (i) we take into account stratification effects due to gravitation, hence we handle non constant reference density states;
(ii) we consider general pressure laws $P(\rho)\sim\rho^\g$, with $\gamma\geq 3/2$ (see \eqref{hyp:p} below).
Concerning (i), let us stress that our work seems the first one able to tackle the combined effect of the gravitational force and boundary layers.
As for (ii), work \cite{B-D-GV} handles the case $\gamma=2$. Incidentally, the threshold $3/2$ for the number $\gamma$ is the same as the one for the Lions-Feireisl theory of
weak solutions theory. \\
\indent
Our approach is based on relative entropy estimates for system \eqref{e.nse}, see \cite{Germ,F-J-N,F-N-S,F-L-N}.
The progress achieved in this paper is made possible thanks to the introduction of a simple tool, which seems to be new in this context:
we rely on \emph{anisotropic Sobolev embeddings}, see Lemma \ref{lem.ani} below. This enables to compensate for the lack of coercivity for
$\partial_3u_h$, due to the strong anisotropy in the Lam\'e operator \eqref{e.lame}. Doing so, we are able to extend the range of values for the parameter $\gamma$,
and also to improve the quantitative error bounds.

\subsection{Outline of the paper}

The paper consists of two parts. The first one, treated in Subsection \ref{sec.wp}, is devoted to the proof of the well-posedness result for system \eqref{eq:wp_nse}.
The main result is Theorem \ref{th:wp} on page \pageref{th:wp}. The proof relies on maximal regularity estimates for a parabolic equation related to an
anisotropic and variable coefficients {Lam\'e operator}; see Proposition \ref{p:DFP}. 
The second part, which is the matter of Subsection \ref{sec.quant}, is concerned with the proof of the quantitative estimates for \eqref{e.nse} in the limit $\ep\rightarrow 0$.
We first build an expansion based on a formal multi-scale analysis. Second, we derive a relative entropy inequality. Finally, we carry out the quantitative estimates,
using in particular the anisotropic Sobolev embeddings of Lemma \ref{lem.ani}.

\subsection{Main notations and definitions}

Since our interest is in real fluid flows, the whole paper is written in space dimension $d=3$.
Notice though that some results, in particular those of Subsection \ref{sec.wp} such as the maximal regularity statement, can easily be extended to higher-dimensional systems.
{The domain $\Omega$ denotes an open set, usually $\Omega=\R^3$ or $\R^2\times(0,1)$ in this paper.} \\
\indent
When appropriate, we use Einstein's convention on repeated indices for summation. \\
\indent
Given a two-dimensional vector $v\,=\,(v_1,v_2)$, we define $v^\perp:=(-v_2,v_1)$. Given a vector field $v\in\R^3$, we will often use the notations $\nabla\cdot v$ and $\nabla\times v$ to denote respectively $\div v$ and $\curl v$. 
For a vector $x\in\R^3$, we often use the notation $x=(x_h,x_3)\in\R^3$ to denote the horizontal component $x_h\in\R^2$ and the vertical component $x_3\in\R$.
According to this decomposition, we define the horizontal differential operators $\nabla_h$, $\Delta_h$ and $\nabla_h\cdot$ as usually; 
we also set $\nabla_h^\perp:=(-\partial_2,\partial_1)$. These operators act just on the $x_h$ variables. Notice that the third
component of the vector $\nabla\times v$ is $\d_1v_2-\d_2v_1$: {this quantity will be denoted by $\nabla_h^\perp\cdot v_h$.} \\
\indent
We introduced $\Delta_{\mu,\ep}$ as in \eqref{e.anisvisco}. Similarly, we define the modified 
gradient operator $\nabla_{\mu,\veps}\,:=\,\bigl(\sqrt{\mu}\d_1,\sqrt{\mu}\d_2,\sqrt{\veps}\d_3\bigr)$. The anisotropic \emph{Lam\'e operator} $\mc L$ is defined by
\begin{equation}\label{e.lame}
\mc Lu\,=\,-\Delta_{\mu,\veps}u\,-\,\lambda\nabla\nabla\cdot u\,.
\end{equation}

\indent
Throughout this paper, given a Banach space $X$ and a sequence $(a^\veps)_\veps$ of elements of $X$, the notation $(a^{\veps})_\veps\,\subset\,X$ is to be understood as the fact
that the sequence~$(a^{\veps})_\veps$ is uniformly bounded in~$X$.
We will often denote, for any~$p\in[1,\infty]$ and any Banach space~$X$,
$L^p_T(X)\,:=\,L^p\bigl((0,T);X(\Omega)\bigr)$. When $T=+\infty$, we will simply write $L^p(X)$. \\
\indent 
{For the definition and basic properties of Besov spaces, we refer to Chapter 2 of \cite{B-C-D}, or to Subsections 2.2 and 2.3 of \cite{D-F-P}.}

\noindent\underline{Pressure law.} We consider barotropic flows, for which the pressure $P$ is supposed to be a smooth function of the density only. We assume
(see e.g. \cite{F-N-S}, \cite{F-G-N}, \cite{F-J-N}, \cite{F-N_CPDE}) that $P\in\mc C\bigl([0,\infty)\bigr)\cap\mc C^3\bigl((0,\infty)\bigr)$ enjoys
\begin{equation}\label{hyp:p}
P(0)=0\,,\qquad P'(\rho)>0\quad\forall\,\rho>0\,,\qquad\lim_{\rho\ra\infty}\frac{P'(\rho)}{\rho^{\g-1}}\,=\,a>0\,,
\end{equation}
for some $\g\geq1$. Given $P$, we define the \emph{internal energy function} $H$ by the formula
\begin{equation}\label{def:H}
H(\rho)\,:=\,\rho\int_1^\rho\frac{P(z)}{z^2}\,dz\qquad\qquad\mbox{ for all }\qquad\rho\in\,(0,\infty)\;.
\end{equation}
Notice that the relation $\rho\, H''(\rho)\,=\,P'(\rho)$ holds for all $\rho>0$.

\section{A well-posedness result in the presence of a strongly anisotropic viscous stress tensor and stratification}
\label{sec.wp}

In this section we show a well-posedness result for the barotropic Navier-Stokes system \eqref{eq:wp_nse}. 
As explained in Subection \ref{sec.chall}, there are two main difficulties. The first one is to handle 
the anisotropy of the viscous stress tensor. It prevents one from using classical compactness techniques to prove the existence of weak solutions
having finite energy (see e.g. \cite{PLLions,N-S,F-N-Petz} and the comments in \cite{B-D-GV,B-J}). It is also a major obstacle to the use of the effective flux \eqref{e.eqeffflux} as in \cite{Hoff,D-F-P}. The second one is the non-constant reference density state $\bar\rho$, due to the potential force $\nabla G$. \\
\indent
Our approach is reminiscent from the works \cite{D-F-P,Shi18}. It is based on maximal regularity estimates for the velocity field. This approach enables us to fully exploit the parabolic gain of regularity due to the momentum 
equation, and to use it in the mass equation in order to transport higher-order Sobolev norms of the density function. Moreover, it allows us to consider non constant density reference states, which is crucial in view
of studying stratification effects, see system \eqref{e.nse} and Section~~\ref{sec.quant}. \\

\indent
It is not clear whether or not the whole method of the paper \cite{D-F-P} works in presence of an anisotropic Lam\'e operator, 
since it deeply uses Hoff's effective flux \eqref{e.eqeffflux} and algebraic cancellations appearing in its equation, see \eqref{e.eqF}.
Thus, compared to \cite{D-F-P}, we will work with solutions in the energy space, so with less integrability, but which are more regular.
We are also able to deal with parabolic operators with variable coefficients when applying the maximal regularity results.\\

\indent
Our main result in this direction is the following statement. Our aim is to give a streamlined method for well-posedness, appropriate for proving the existence and uniqueness of
finite-energy solutions in the presence of: (i) an anisotropic viscous stress tensor given by \eqref{e.anisvisco}, (ii) non constant reference density states. A small compromise consists
on the one hand on the fact that we do not strive for optimality in the assumptions of the theorem, and on the other hand on the fact that we work in the space domain $\Omega=\R^3$.
Notice however that the case of more general domains can be treated with some technical adaptations.

\begin{theorem} \label{th:wp}
{Let $\gamma\geq 1$.} Let $\bar\rho\in W^{3,\infty}(\R)$. 
Assume that $\bar\rho$ is uniformly bounded from below, i.e. $\bar\rho\geq\kappa>0$. We define the potential by $G=H'(\bar\rho)$, where $H$ is defined by \eqref{def:H}.
Then for any $\eta>0$ which verifies
\begin{equation} \label{eq:d-small}
\eta\,\leq\, \min(1/(8C_0),\kappa/8),
\end{equation}
where $C_0$ is the constant given by Proposition \ref{p:DFP}, the property below holds. \\
\indent
Consider system \eqref{eq:wp_nse}, supplemented with the initial datum $\big(\rho,u\big)_{|t=0}\,=\,\big(\rho_{in},u_{in}\big)$. 
For any $\big(\rho_{in},u_{in}\big)$, with $\rho_{in}-\bar\rho\in H^2$ and $u_{in}\in B^{3/2}_{2,4/3}$ and such that
$$
\left\|\rho_{in}\,-\,\bar\rho\right\|_{L^\infty}\,\leq\,\eta\,,
$$
there exist a time $T^*{(\gamma,\mu,\ep,\lambda,\|\bar\rho\|_{W^{3,\infty}},\kappa,\|u_{in}\|_{B^{3/2}_{2,4/3}},\|\rho_{in}-\bar\rho\|_{H^2})}>0$
and a unique solution $\big(\rho,u\big)$ to \eqref{eq:wp_nse} on $[0,T^*]\times\Omega$, such that:
\begin{enumerate}
 \item $\rho-\bar\rho\,\in\,\mc C_{T^*}(H^2)$, with $\left\|\rho\,-\,\bar\rho\right\|_{L^\infty_{T^*}(L^\infty)}\,\leq\,4\,\eta$;
 \item $u\in L^\infty_T(L^2)\cap L^2_{T^*}(L^\infty)$, with in addition $\nabla u\in L^4_{T^*}(L^2)\cap L^2_{T^*}(L^\infty)$, $\nabla^2u\in L^4_{T^*}(L^2)$, $\nabla^3u\in L^{4/3}_{T^*}(L^2)$
 and $\d_tu\in L^{4/3}_{T^*}(H^1)$;
 \item $\big(\rho,u\big)$ satisfies the classical energy inequality (see  estimate \eqref{est:energy} below).\end{enumerate}
\end{theorem}

\begin{remark} \label{r:u_0}
Notice that we have some freedom on the regularity of the initial datum. In this sense, we take $B^{3/2}_{2,4/3}$ regularity for $u_0$ for simplicity of presentation, but this condition can be somehow
weakened. \\
Notice also that $H^s\hookrightarrow B^{3/2}_{2,4/3}$ for any $s>3/2$.
\end{remark}

The rest of this section is devoted to the proof of Theorem \ref{th:wp}. The proof is based only on elementary methods, namely energy estimates and maximal regularity. 
We will limit ourselves to showing a priori estimates for smooth solutions to system \eqref{eq:wp_nse}. The proof of existence by an approximation scheme is rather standard, see e.g. \cite{D-F-P} and references therein.

\subsection{Maximal regularity for an anisotropic Lam\'e operator} 
\label{ss:max}

Here we prove the following maximal regularity result for a parabolic equation with an anisotropic Lam\'e operator and vertical stratification due to the coefficient $\bar\rho$.
Pay attention to the fact that, in the definition of $W^{2,1}_{p_2,r_2}$ below, the indices for space and time are in reverse order, in order to stick to the classical definition
(see e.g. \cite{Ser06}). 

\begin{proposition} \label{p:DFP}
Let $\bar\rho\in W^{1,\infty}(\R)$. Assume that $\bar\rho$ is uniformly bounded from below, i.e. $\bar\rho\geq\kappa>0$. Let $\bigl((p_j,r_j)\bigr)_{j=0,1,2}$ 
satisfy  $1<p_2,r_2<+\infty$, $r_2<r_0<+\infty$, $r_2<r_1<+\infty$, $p_0\geq p_2,$ $p_1\geq p_2$, together with the relations 
$$ 
\frac{2}{r_2}\,+\,\frac{3}{p_2}\;=\;1\,+\,\frac{2}{r_1}\,+\,\frac{3}{p_1}\qquad\mbox{ and }\qquad\frac{2}{r_2}\,+\,\frac{3}{p_2}\;=\;2\,+\,\frac{2}{r_0}\,+\,\frac{3}{p_0}\,.
$$ 
Let  $h_{in}$ be in $\dot B^{s_2}_{p_2,r_2}$ with $s_2:=2-2/r_2,$ and 
let $f$ be in $L^{r_2}_{\rm loc}\bigl(\R_+;L^{p_2}(\R^d)\bigr)$. \\ 
Then, there exists a constant $C_0=C_0(\mu,\ep,\lambda,r_0,p_0,r_1,p_1,r_2,p_2,\|\bar\rho\|_{W^{1,\infty}},\kappa)>0$ such that, 
for all $\,T>0$ and all $h\in L^{r_2}_T(W^{1,p_2})$, 
with $h$ solution to the Lam\'e system
\begin{equation} \label{eq:lame}
\left\{\begin{array}{l}
\bar\rho(x_3)\d_t h\,+\,\mc L h\,=\,f \\[1ex]
h_{|t=0}\,=\,h_{in}
\end{array} \right.
\end{equation}
where $\mathcal L$ is defined by \eqref{e.lame}, 
one has the properties $h\in L^{r_0}\bigl([0,T];L^{p_0}\bigr)$ and $\nabla h\,\in\,L^{r_1}\bigl([0,T];L^{p_1}\bigr)$, together with the estimate
\begin{align}\label{est:max-reg_h}
\begin{split}
&\left\|h\right\|_{L^{\infty}_T(\dot B^{s_2}_{p_2,r_2})}\,+\,\left\|h\right\|_{L^{r_0}_T(L^{p_0})}\,+\,\left\|\nabla h\right\|_{L^{r_1}_T(L^{p_1})}\,+\,
\left\|\left(\d_th,\nabla^2h\right)\right\|_{L^{r_2}_T(L^{p_2})}\\
\leq\ &C_0
\left(\left\|h_{in}\right\|_{\dot B^{s_2}_{p_2,r_2}}\,+\,\|f\|_{L^{r_2}_T(L^{p_2})}+\|h\|_{L^{r_2}_T(L^{p_2})}+\|\nabla h\|_{L^{r_2}_T(L^{p_2})}\right).
\end{split}
\end{align}
\end{proposition}

To prove this result, we follow the idea of \cite{D-F-P}. We apply the Leray projector $\mathbb P$ to the system \eqref{eq:lame} and rely on the maximal regularity for a divergence-form parabolic equation.
Because of the stratification, additional commutators involving $\mathbb P$ and the reference density $\bar\rho$ have to be analyzed. 
We do not strive for making the dependence of $C_0$ explicit in $\ep,\, \mu,\, \lambda,\, \|\bar\rho\|_{W^{1,\infty}}$ and $\kappa$, 
because our aim is to obtain a well-posedness theorem for a fixed set of parameters, see Theorem \ref{th:wp}. 

\begin{remark} \label{r:low-order}
Notice that the two last terms in the right hand side of estimate \eqref{est:max-reg_h} cannot be get rid of. Their appearance is due to the commutator terms involving the Leray projector.
It is not possible to swallow them in the left hand side of \eqref{est:max-reg_h}. In spite of this, estimate \eqref{est:max-reg_h} can be used as such to prove the well-posedness result stated in Theorem \ref{th:wp}. Indeed, the terms $\|h\|_{L^{r_2}_T(L^{p_2})}$ and $\|\nabla h\|_{L^{r_2}_T(L^{p_2})}$ are of lower order and can be handled by interpolating with the finite energy. Such an analysis will also be done for the source term $f$, which contains in particular the nonlinear term $\rho u\cdot\nabla u$, see below in Section \ref{ss:wp-u}.
\end{remark}

\begin{proof}[Proof of Proposition \ref{p:DFP}]
We proceed in several steps.

\noindent\underline{Step 1: reduction to the heat equation.} We first rewrite system \eqref{eq:lame} in divergence form: we have
\begin{align}\label{e.lamebis}
&\partial_th-\nabla_{\mu,\ep}\cdot\left(\frac{1}{\bar\rho(x_3)}\nabla_{\mu,\ep} h\right)-\lambda\nabla\left(\frac{1}{\bar\rho(x_3)}\nabla\cdot h\right)\,=\,F\,,
\end{align}
where, for simplicity of notation, we have defined
\begin{align} \label{eq:def-F}
&F\,:=\ \frac1{\bar\rho(x_3)}f+\ep\frac{\bar\rho'(x_3)}{\bar\rho(x_3)^2}\partial_3h + \lambda\frac{\bar\rho'(x_3)}{\bar\rho(x_3)^2}\big(\nabla\cdot h\big)e_3.
\end{align}
As in \cite{D-F-P}, we apply the Leray projector $\mbb P$ on the equation. The difference with \cite{D-F-P} is that we now have commutator terms appearing.
Recall that
\begin{equation} \label{eq:P-Q}
\mbb P\,=\,\Id\,-\,\Q\,,\qquad\qquad \Q\,:=\,-\,\nabla(-\Delta)^{-1}\nabla\cdot\,,
\end{equation}
where the formulas have to be interpreted in the sense of Fourier multipliers. Applying $\mathbb P$ to \eqref{e.lamebis}, we get
\begin{align}\label{e.eqPu}
&\partial_t\mathbb P h-\nabla_{\mu,\ep}\cdot\left(\frac{1}{\bar\rho(x_3)}\nabla_{\mu,\ep} \mathbb P h\right)\,=\,\mbb PF\,+\,\mathbf{C}\,, 
\end{align} 
where the commutator term $\mathbf{C}$ is defined by
\begin{align}\label{e.commC}
\mathbf C:=\ &-\nabla_{\mu,\ep}\cdot\left(\left[\frac{1}{\bar\rho(x_3)}\,,\,\mbb{P}\right]\nabla_{\mu,\ep} h\right)\,.
\end{align}
Here above, the symbol $[A,B]:=AB-BA$ denotes the commutator between two operators $A$ and $B$. \\
We now compute the equation for $\Q h\,=\,h-\mathbb Ph$. Notice that, using \eqref{eq:P-Q}, we have
\[
\nabla\nabla\cdot h\,=\,\nabla \nabla\cdot \Q h\,=\,\Delta \Q h
\]
Therefore, rewriting \eqref{e.lamebis} in the form
\[
 \partial_th-\nabla_{\mu,\ep}\cdot\left(\frac{1}{\bar\rho(x_3)}\nabla_{\mu,\ep} h\right)-\lambda\frac{1}{\bar\rho(x_3)}\nabla\nabla\cdot h\,=\,
 \frac1{\bar\rho(x_3)}f+\ep\frac{\bar\rho'(x_3)}{\bar\rho(x_3)^2}\partial_3h
\]
and taking the difference of this equation with \eqref{e.eqPu}, we immediately find
\begin{align}\label{e.equ-Pu}
\begin{split}
&\partial_t\Q h-\nabla_{\mu+\lambda,\ep+\lambda}\cdot\left(\frac{1}{\bar\rho(x_3)}\nabla_{\mu+\lambda,\ep+\lambda}\Q h\right) \\
&=\Q\left(\frac1{\bar\rho(x_3)}f\right)-\mbb P\left(\ep\frac{\bar\rho'(x_3)}{\bar\rho(x_3)^2}\partial_3h + \lambda\frac{\bar\rho'(x_3)}{\bar\rho(x_3)^2}\big(\nabla\cdot h\big)e_3\right)-
\mathbf{C}+\lambda\frac{\bar\rho'(x_3)}{\bar\rho(x_3)^2}\d_3\Q h.
\end{split}
\end{align}
Both equations \eqref{e.eqPu} and \eqref{e.equ-Pu} are heat-type equations with a variable coefficient. The next step will show that all the quantities
appearing in their right-hand sides are of lower order in $h$.

\noindent\underline{Step 2: computation of the commutators.} 
Our goal is now to compute the commutator term $\mathbf C$, defined by \eqref{e.commC}. 
Owing to \eqref{eq:P-Q} and switching the position of the constant factors, we have
\[
\mathbf C:=-\nabla_{\mu,\ep}\cdot\left(\left[\frac{1}{\bar\rho(x_3)}\,,\,\mbb{P}\right]\nabla_{\mu,\ep} h\right) =
\nabla\cdot\left(\left[\frac{1}{\bar\rho(x_3)}\,,\,\Q\right]\nabla_{\mu^2,\ep^2} h\right)\,.
\]
Recall the definition of $\Q$ in \eqref{eq:P-Q}. We start by applying the divergence operator to the various quantities: after observing that
$\nabla\cdot\nabla_{\mu^2,\veps^2}=\Delta_{\mu,\veps}$, we find
\begin{align*}
\left[\frac{1}{\bar\rho}\,,\,\Q\right]\nabla_{\mu^2,\ep^2} h &= -\frac{1}{\bar\rho}\nabla(-\Delta)^{-1}\Delta_{\mu,\veps}h +
\nabla(-\Delta)^{-1}\nabla\cdot\left(\frac{1}{\bar\rho}\,\nabla_{\mu^2,\veps^2}h\right)\,.
\end{align*}
Therefore, we find 
\begin{align*}
\mathbf C\,&=\,\nabla\cdot\left(\left[\frac{1}{\bar\rho(x_3)}\,,\,\Q\right]\nabla_{\mu^2,\ep^2} h\right)\\
&=\,-\nabla\cdot\left(\frac{1}{\bar\rho(x_3)}\nabla(-\Delta)^{-1}\Delta_{\mu,\veps}h\right) +
\nabla\cdot\left(\nabla(-\Delta)^{-1}\nabla\cdot\left(\frac{1}{\bar\rho(x_3)}\,\nabla_{\mu^2,\veps^2}h\right)\right) \\
&=\,
\frac{\bar\rho'(x_3)}{\bar\rho(x_3)^2}\,\d_3(-\Delta)^{-1}\Delta_{\mu,\veps}h + \frac{1}{\bar\rho(x_3)}\Delta_{\mu,\veps}h - \nabla\cdot\left(\frac{1}{\bar\rho(x_3)}\,\nabla_{\mu^2,\veps^2}h\right)\,,
\end{align*}
which in the end leads us to
\begin{equation} \label{eq:comm}
\mathbf C\,=\,\frac{\bar\rho'(x_3)}{\bar\rho(x_3)^2}\,\Big(\d_3(-\Delta)^{-1}\Delta_{\mu,\veps}h + \veps\,\d_3h\Big)\,.
\end{equation}
Since the operator $(-\Delta)^{-1}\Delta_{\mu,\veps}$ is a singular integral operator, whose symbol is homogeneous of degree $0$, the previous computations show that
$\mathbf{C}$ is indeed a lower order term, as claimed.

\noindent\underline{Step 3: estimates via maximal regularity.}
\noindent We apply maximal regularity estimates of \cite{Hieb-Pruss} for the divergence-form parabolic operator
\begin{equation} \label{eq:parab}
\partial_t(\ \cdot\ )-\nabla_{\mu,\ep}\cdot\big(\bar\rho^{-1}\nabla_{\mu,\ep}(\ \cdot\  )\big)
\end{equation}
to equation \eqref{e.eqPu}. We obtain
\[
\left\|\left(\d_t\mathbb Ph,\nabla^2\mathbb Ph\right)\right\|_{L^{r_2}_T(L^{p_2})}
\leq\ C
\left(\left\|\mathbb P h_{in}\right\|_{\dot B^{s_2}_{p_2,r_2}}\,+\,\|\mbb PF+\mathbf{C}\|_{L^{r_2}_T(L^{p_2})}\right)\,.
\]
Using the continuity of $\mbb P$ and $(-\Delta)^{-1}\Delta_{\mu,\veps}$ on $L^{p_2}$ (since $1<p_2<+\infty$) and keeping in mind definition \eqref{eq:def-F} of $F$ and \eqref{eq:comm},
it is easy to bound
\[
\|\mbb PF+\mathbf{C}\|_{L^{r_2}_T(L^{p_2})}\ \leq\ C\,\left(\|f\|_{L^{r_2}_T(L^{p_2})}+\|\nabla h\|_{L^{r_2}_T(L^{p_2})}\right)\,.
\]
So, in the end we find
\begin{align}\label{e.estPh}
\begin{split}
\left\|\left(\d_t\mathbb Ph,\nabla^2\mathbb Ph\right)\right\|_{L^{r_2}_T(L^{p_2})}
\leq\ &C_0
\left(\left\|\mathbb P h_{in}\right\|_{\dot B^{s_2}_{p_2,r_2}}+\|f\|_{L^{r_2}_T(L^{p_2})}+\|\nabla h\|_{L^{r_2}_T(L^{p_2})}
\right).
\end{split}
\end{align}
Next, thanks to the gaussian bounds on the fundamental solution of the parabolic equation \eqref{eq:parab}, see \cite{Nash} and \cite{Aron}, arguing as in Lemma 2.4 of \cite{D-F-P}
we infer that
\begin{equation} \label{est:Ph_0}
\left\|\mathbb Ph\right\|_{L^{r_0}_T(L^{p_0})}\,\leq\,C_0
\left(\left\|\mathbb P h_{in}\right\|_{\dot B^{s_2}_{p_2,r_2}}+\|f\|_{L^{r_2}_T(L^{p_2})}+\|\nabla h\|_{L^{r_2}_T(L^{p_2})}\right).
\end{equation}
Thus, in order to complete the proof, it remains us to bound the $L^{r_1}_T(L^{p_1})$ norm of $\nabla\mbb Ph$. This is the goal of the next step.

\noindent\underline{Step 4: a functional inequality.}
\noindent We claim that, for any $w\in W^{2,1}_{p_2,r_2}\big(\R^3\times(0,T)\big)$, one has the estimate
\begin{equation}\label{e.parsobemb}
\left\|\nabla w\right\|_{L^{r_1}_T(L^{p_1})}\ \leq\ C\,\bigg(\|\nabla^2w\|_{L^{r_2}_TL^{p_2}(\R^3\times(0,T))}+\|\partial_tw\|_{L^{r_2}_TL^{p_2}(\R^3\times(0,T))}\bigg)\,. 
\end{equation}
In fact, estimate \eqref{e.parsobemb} is a functional inequality, which does not use the parabolic equation. We are going to deduce it
from the corresponding inequality in a bounded domain, see in particular \cite[Proposition 2.1]{Sol04}. \\
In order to prove \eqref{e.parsobemb}, we take $\varphi\in C^\infty_c(B(0,1))$ such that $\varphi=1$ on $B(0,\frac12)$. For $k\in\mathbb N$, we consider $w_k:=\varphi(\cdot/k)w$. Notice that $w_k$ is supported in $B(0,k)$. Then, we consider the rescaled function $w_k^\lambda$ for $\lambda\geq k$, according to the parabolic scaling:
\begin{equation*}
w_k^\lambda=w_k(\lambda\cdot,\lambda^2\cdot)=\varphi(\tfrac\lambda k\cdot)w(\lambda\cdot,\lambda^2\cdot).
\end{equation*}
Notice that $w_k^\lambda$ is supported in $B(0,1)$. 
By the mixed norm parabolic Sobolev embedding in the bounded domain $B(0,1)$, we have 
\begin{equation*}
\|\nabla w_k^\lambda\|_{L^{r_1}_TL^{p_1}(B(0,1)\times(0,T/\lambda^2))}\leq C\|w_k^\lambda\|_{W^{2,1}_{p_2,r_2}(B(0,1)\times(0,T/\lambda^2))},
\end{equation*}
where $C$ is a constant independent of $T$. Then, by rescaling, we obtain
\begin{multline*}
\frac\lambda{\lambda^{\frac3{p_1}+\frac2{r_1}}}\|\nabla w_k\|_{L^{r_1}_TL^{p_1}(B(0,k)\times(0,T))}\\
\leq C\bigg(\frac1{\lambda^{\frac3{p_2}+\frac2{r_2}}}\|w_k\|_{L^{r_2}_TL^{p_2}(B(0,k)\times(0,T))}+\frac\lambda{\lambda^{\frac3{p_2}+\frac2{r_2}}}\|\nabla w_k\|_{L^{r_2}_TL^{p_2}(B(0,k)\times(0,T))}\\
+\frac{\lambda^2}{\lambda^{\frac3{p_2}+\frac2{r_2}}}\|\nabla^2w_k\|_{L^{r_2}_TL^{p_2}(B(0,k)\times(0,T))}+
\frac{\lambda^2}{\lambda^{\frac3{p_2}+\frac2{r_2}}}\|\partial_tw_k\|_{L^{r_2}_TL^{p_2}(B(0,k)\times(0,T))}\bigg),
\end{multline*}
and letting $\lambda\rightarrow\infty$ we deduce
\begin{align}
&\|\nabla w_k\|_{L^{r_1}_TL^{p_1}(\R^3\times(0,T))}\label{e.estwk} \\ 
&\qquad\qquad\leq C\bigg(\|\nabla^2w_k\|_{L^{r_2}_TL^{p_2}(\R^3\times(0,T))}+\|\partial_tw_k\|_{L^{r_2}_TL^{p_2}(\R^3\times(0,T))}\bigg). \nonumber
\end{align}
From estimate \eqref{e.estwk} and the fact that $w\in W^{2,1}_{p_2,r_2}$, we get that $\nabla w_k$ is a Cauchy sequence in $L^{r_1}_T(L^{p_1})$.
Hence we can pass to the limit in \eqref{e.estwk}: thus, we obtain \eqref{e.parsobemb}, as claimed.

\noindent\underline{Step 5: end of the proof.}
\noindent Owing to the fact that $1<p_2<+\infty$ and to \eqref{e.estPh}, we have that $\mbb P h$ belongs to $W^{2,1}_{p_2,r_2}(\R^3\times(0,T))$. Therefore, we can apply
inequality \eqref{e.parsobemb} to $w=\mbb P h$. We find
\begin{equation} \label{est:Ph_1}
\left\|\nabla \mathbb Ph\right\|_{L^{r_1}_T(L^{p_1})}\,\leq\,C_0
\left(\left\|\mathbb P h_{in}\right\|_{\dot B^{s_2}_{p_2,r_2}}+\|f\|_{L^{r_2}_T(L^{p_2})}+\|\nabla h\|_{L^{r_2}_T(L^{p_2})}\right),
\end{equation}
for a possibly new constant $C_0>0$. In addition, since $\Q h$ solves equation \eqref{e.equ-Pu}, which is analogous to \eqref{e.eqPu}, estimates similar to \eqref{e.estPh},
\eqref{est:Ph_0} and \eqref{est:Ph_1} hold true for $\Q h$ and its derivatives. Then, writing $h=\mbb Ph + \Q h$ completes the proof of the proposition.
\end{proof}

\subsection{Basic energy estimates} \label{ss:wp-energy}

Let us take a smooth solution $(\rho,u)$ to system \eqref{eq:wp_nse}, such that $\rho-\bar\rho$ and $u$ decay sufficiently fast at space infinity. 
We want to find \tsl{a priori} estimates
in suitable norms. For this, we are going to work with the variables
$$
r(t)\,:=\,\rho(t)-\bar\rho\qquad\mbox{and}\qquad u.
$$
In the same way, we set $r_{in}\,=\,\rho_{in}-\bar\rho$. Recall that
\begin{equation} \label{est:r_in-size}
\left\|r_{in}\right\|_{L^\infty}\,\leq\,\eta\,. 
\end{equation}
\indent We start by performing classical energy estimates, which provides us with a bound for the low frequencies of the velocity field.
Namely, by multiplying the momentum equation in \eqref{eq:wp_nse} by $u$ and integrating by parts, we get, in a standard way, the control 
\begin{equation} \label{est:en_est}
\left\|\sqrt{\rho}\,u\right\|_{L^\infty(L^2)}\,+\,\left\|\nabla u\right\|_{L^2(L^2)}\,\leq\,C_{energy}\,.
\end{equation}
See Section \ref{ss:unif-bounds} for similar bounds. Of course, the constant $C_{energy}$ depends also on the (fixed) values of the coefficients $(\mu,\veps,\lambda)$.
From this control and
$$
\int_\Omega\bar\rho|u|^2\,=\,\int_{\Omega}\rho|u|^2\,-\,\int_\Omega r|u|^2\,,
$$
we infer that, for any $T>0$,
\begin{equation} \label{est:u_low-f}
\left\|u\right\|_{L^\infty_T(L^2)}\,\leq\,\kappa^{-1}\left(C_{energy}\,+\,\|r\|_{L^\infty_T(L^\infty)}\,\|u\|_{L^\infty_T(L^2)}\right)\,,
\end{equation}
where we recall that $\kappa$ is a lower bound for $\bar\rho$. 
Hence, if $T>0$ is such that
\begin{equation} \label{est:r_small}
\|r\|_{L^\infty_T(L^\infty)}\,\leq\,4\,\eta<\frac\kappa2\,,
\end{equation}
where $\eta>0$ is the size of the initial datum in the $L^\infty$ norm, recall \eqref{est:r_in-size}, we deduce
\begin{equation}\label{est:u_L^2}
\left\|u\right\|_{L^\infty_T(L^2)}\,\leq\,2C_{energy}\kappa^{-1}\,.
\end{equation}

The next goal is to exhibit a control on the density variation function $r$. We will work in higher Sobolev norms, namely in $H^2$. However, we are going to bound its $L^\infty$ norm
independently (i.e. without using Sobolev embeddings), in order to get, in view of \eqref{est:r_small}, a smallness condition only on $\|r_{in}\|_{L^\infty}$, and not on the higher order norm
of $r_{in}$.

\subsection{Estimates for the density function} \label{ss:wp-density}

In this subsection, we find transport estimates for the density variation function $r$. First of all, from the mass equation in \eqref{eq:wp_nse} we find that $r$ fulfills
\begin{equation} \label{eq:r}
\d_tr\,+\,u\cdot\nabla r\,+\,r\,\nabla\cdot u\,=\,-u_3\bar\rho'-\bar\rho\,\nabla\cdot u\,.
\end{equation}
A basic $L^p$ estimate for this equation gives, for any $t>0$,
\begin{align*}
\|r(t)\|_{L^p}\,\leq\ &\left\|r_{in}\right\|_{L^p}\,+\,\left(1-\frac{1}{p}\right)\int^t_0\|r(\tau)\|_{L^p}\,\|\nabla\cdot u(\tau)\|_{L^\infty}\,d\tau\\
&\qquad\qquad\qquad
+\,\int_0^t\|\bar\rho'\|_{L^\infty}\big(\|u_3(\tau)\|_{L^p}
+\|\nabla\cdot u(\tau)\|_{L^p}\big)\,d\tau\,.
\end{align*}
On the one hand, in the case $p=2$, we get
\begin{align} \label{est:r_L^2}
\begin{split}
\|r\|_{L^\infty_T(L^2)}\,\leq\ &\left\|r_{in}\right\|_{L^2}\,+\,\frac{1}{2}\int^T_0\|r\|_{L^2}\,\|\nabla\cdot u\|_{L^\infty}\,d\tau\,\\
&+\,\int_0^T\|\bar\rho'\|_{L^\infty}\big(\|u_3(\tau)\|_{L^2}+\|\nabla\cdot u(\tau)\|_{L^2}\big)\,d\tau
\end{split}
\end{align}
for any time $T>0$. In turn, Gr\"onwall's lemma gives the bound
\begin{multline} \label{est:r_L^2_final}
\|r\|_{L^\infty_T(L^2)}\,\leq\,e^{\int^T_0\|\nabla\cdot u(\tau)\|_{L^\infty}\,d\tau} \\
\times\Big(\left\|r_{in}\right\|_{L^2}\,+\int_0^T\|\bar\rho'\|_{L^\infty}\big(\|u_3(\tau)\|_{L^2}
+\|\nabla\cdot u(\tau)\|_{L^2}\big)\,d\tau\Big).
\end{multline}
On the other hand, by letting $p\ra+\infty$, from Gr\"onwall's lemma again we deduce
\begin{multline} \label{est:r_L^inf}
\|r\|_{L^\infty_T(L^\infty)}\,\leq\,e^{\int^T_0\|\nabla\cdot u(\tau)\|_{L^\infty}\,d\tau}\\
\times \Big(\left\|r_{in}\right\|_{L^\infty}\,+\int_0^T\|\bar\rho'\|_{L^\infty}\big(\|u_3(\tau)\|_{L^\infty}+\|\nabla\cdot u(\tau)\|_{L^\infty}\big)\,d\tau\Big).
\end{multline}
We now define the time $T>0$ as
\begin{multline} \label{def:T}
T\,:=\,\sup\bigg\{t>0\;\Big|\, {\int_0^t\|\bar\rho'\|_{L^\infty}\|u_3(\tau)\|_{L^\infty}\,d\tau\,}\\
{+\,\int_0^t(1+\|\bar\rho\|_{L^\infty})\|\nabla\cdot u(\tau)\|_{L^\infty}\,d\tau}\,\leq\,\min\big\{\eta,\log 2\big\}\bigg\}\,,
\end{multline}
where $\eta>0$ is as in \eqref{est:r_in-size}. From the previous bound we gather then
\begin{equation} \label{est:r_L^inf_final}
\|r\|_{L^\infty_T(L^\infty)}\,\leq\,2\left\|r_{in}\right\|_{L^\infty}\,+2\,\eta\,\leq\,4\,\eta\,.
\end{equation}

We now differentiate equation \eqref{eq:r} with respect to $x_j$, for any $j\in\{1,2,3\}$. We get that $\d_jr$ verifies the following equation:
\begin{multline}\label{eq:Dr}
\d_t\d_jr\,+\,u\cdot\nabla\d_jr\,+\,\d_jr\,\nabla\cdot u\,=\,\,-\,\d_ju\cdot\nabla r\,-\,r\,\nabla\cdot\d_ju\\
-\partial_ju_3\bar\rho'-u_3\partial_j\bar\rho'-\partial_j\bar\rho\nabla\cdot u-\bar\rho\nabla\cdot\partial_ju\,.
\end{multline}
The same $L^2$ energy estimate for the continuity equation as above yields
\begin{align*}
\left\|\d_jr(t)\right\|_{L^2}\,\leq\ &\left\|\d_jr_{in}\right\|_{L^2}+\frac{1}{2}\int^t_0\left\|\d_jr\right\|_{L^2}\left\|\nabla\cdot u\right\|_{L^\infty}\,d\tau\\
&+\int^t_0\Big(\|r\|_{L^\infty}\left\|\nabla\cdot\d_ju\right\|_{L^2}+\left\|\nabla r\right\|_{L^2}\|\d_ju\|_{L^\infty}+\|\partial_j\bar\rho'\|_{L^\infty}\|u_3\|_{L^2}\\
&+\|\bar\rho'\|_{L^\infty}\|\partial_ju_3\|_{L^2}+\|\partial_j\bar\rho\|_{L^\infty}\|\nabla\cdot u\|_{L^2}+\|\bar\rho\|_{L^\infty}\|\nabla\cdot\partial_ju\|_{L^2}
\Big)d\tau\,, \nonumber
\end{align*}
for any $t\geq 0$. 
Hence, for all $t\geq 0$, one has 
\begin{align}
&\left\|\nabla r(t)\right\|_{L^2}\,\leq\,\left\|\nabla r_{in}\right\|_{L^2}  \label{est:d_jr}\\
&\qquad+C\int^t_0\left(\|r\|_{L^\infty}\left\|\nabla^2u\right\|_{L^2}+\left\|\nabla r\right\|_{L^2}\left\|\nabla u\right\|_{L^\infty}+
\|\bar\rho\|_{W^{2,\infty}}\left\|u\right\|_{H^{2}}\right)\,d\tau\,. \nonumber
\end{align}
We do not apply Gr\"onwall's lemma directly on this inequality. It is better to bound first the second-order derivatives of $r$.
For this, let us differentiate equation \eqref{eq:Dr} with respect to $x_k$, for any $k\in\{1,2,3\}$. We deduce the following
equation for the quantity $r_{jk}\,:=\,\partial^2_{kj}r$:
\begin{align*}
\d_tr_{jk}+u\cdot\nabla r_{jk}+r_{jk}\nabla\cdot u=&-\Big(\partial^2_{jk}u\cdot\nabla r+\d_ju\cdot\nabla\d_kr+\d_kr\cdot\nabla\d_ju-r\nabla\cdot\partial^2_{kj}u\Big)\\
&-\partial_k\Big(\partial_ju_3\bar\rho'+u_3\partial_j\bar\rho'+\partial_j\bar\rho\nabla\cdot u+\bar\rho\nabla\cdot\partial_ju\Big)\,.
\end{align*}
Performing another energy estimate, we infer, for any $t\geq0$, the inequality
\begin{align*}
\left\|r_{jk}(t)\right\|_{L^2}\,\leq\ &\left\|r_{in,jk}\right\|_{L^2}\,+\,\frac{1}{2}\int^t_0\left\|r_{jk}\right\|_{L^2}\,\left\|\nabla\cdot u\right\|_{L^\infty}\,d\tau \\
&\qquad\qquad\quad+\int^t_0\Big(\|\nabla r\|_{L^4}\,\|\nabla^2u\|_{L^4}+\|\nabla u\|_{L^\infty}\,\|\nabla^2r\|_{L^2} \\
&\qquad\qquad\qquad\qquad\qquad+\|r\|_{L^\infty}\left\|\nabla^3u\right\|_{L^2}+\|\bar\rho\|_{W^{3,\infty}}\|u\|_{H^{3}}\Big)\,d\tau\,,
\end{align*}
where we have defined $r_{{in},jk}\,:=\,\partial^2_{jk}r_{in}$.
After using the interpolation inequality
\begin{equation} \label{est:interp}
\|\phi\|_{L^4}\,\leq\,\|\phi\|_{L^2}^{1/4}\;\|\phi\|_{L^6}^{3/4}
\end{equation}
together with the Sobolev embedding $\dot H^1(\R^3)\hookrightarrow L^6(\R^3)$, we obtain
\begin{align*}
\|\nabla r\|_{L^4}\,\|\nabla^2u\|_{L^4}\,&\leq\,\|\nabla r\|_{L^2}^{1/4}\;\left\|\nabla^2r\right\|_{L^2}^{3/4}\;\left\|\nabla^2u\right\|_{L^2}^{1/4}\;\left\|\nabla^3u\right\|_{L^2}^{3/4} \\
&\leq\,\left(\|\nabla r\|_{L^2}\,+\,\left\|\nabla^2 r\right\|_{L^2}\right)\,\left(\left\|\nabla^2u\right\|_{L^2}\,+\,\left\|\nabla^3u\right\|_{L^2}\right)\,.
\end{align*}
In view of this inequality, the previous bound on $r_{jk}$ yields, on the time interval $[0,T]$,
\begin{align}\label{est:DDr} 
\begin{split}
&\left\|\nabla^2r\right\|_{L^\infty_T(L^2)}\,\leq\ \left\|\nabla^2r_{in}\right\|_{L^2}\\
&\qquad\quad+C\int^t_0\Big(\|r\|_{L^\infty}\left\|\nabla^3u\right\|_{L^2}+\|\bar\rho\|_{W^{3,\infty}}\|u\|_{H^3}\\
&\qquad\qquad\quad+\left(\|\nabla r\|_{L^2}+\left\|\nabla^2 r\right\|_{L^2}\right)\left(\left\|\nabla u\right\|_{L^\infty}+\left\|\nabla^2u\right\|_{L^2}+\left\|\nabla^3u\right\|_{L^2}\right)\Big)
\,d\tau\,. 
\end{split}
\end{align}

Let us now introduce the notation 
$$
\mc R(t)\,:=\,\left\|r\right\|_{L^\infty_t(H^2)}\,+\,\left\|r\right\|_{L^\infty_t(L^\infty)}\,.
$$
Summing up \eqref{est:r_L^2_final}, \eqref{est:r_L^inf_final}, \eqref{est:d_jr} and \eqref{est:DDr} we deduce the following bound:
\begin{equation} \label{est:r_H^2}
\mc R(T)\,\leq\,C\left(\left\|r_{in}\right\|_{H^2}+\int^T_0\big(\|\bar\rho\|_{W^{3,\infty}}+\mc R(\tau)\big)\left(\left\|\nabla u\right\|_{L^\infty}+\left\|u\right\|_{H^{3}}\right)d\tau\right)\,.
\end{equation}

To end this part, we observe that, for almost all $t\in[0,T]$ and $x\in\R^3$, $|r(x,t)|\leq 4\eta$. So, by assumptions \eqref{hyp:p} on the pressure, there exists a constant $C>0$, {depending
on the function $P$ and $\|\bar\rho\|_{L^\infty}$,} such that, for any parameter $s\in[0,1]$, on $[0,T]\times\R^3$ one has
\begin{equation} \label{est:pressure}
\left|P(\bar\rho+rs)\right|\,+\,\left|P'(\bar\rho+rs)\right|\,+\,\left|P''(\bar\rho+rs)\right|\,+\,\left|P'''(\bar\rho+rs)\right|\,\leq\,C\,.
\end{equation}

\subsection{Maximal regularity estimates for the velocity field} \label{ss:wp-u}

It remains to find a bound on $u$. For this, we mimic an approach used in \cite{D-F-P}, based on the maximal regularity estimates of Subsection \ref{ss:max}.
To begin with, we recast the equation for $u$ as
\begin{equation} \label{eq:mr}
\bar\rho\d_tu\,+\,\mc Lu\,=\,-r\,\d_tu\,-\,f\,,
\end{equation}
where the anisotropic Lam\'e operator $\mc L$ is defined by \eqref{e.lame}, 
and where we have set
\begin{equation} \label{def:f}
 f\,:=\,(\bar\rho+r)\,u\cdot\nabla u\,+\,\nabla P(\rho)
-\rho\nabla G.
\end{equation}
In the computations below, when convenient, we will resort to the notation $\bar\rho+r\,=\,\rho$, and use the bounds \eqref{est:en_est} and \eqref{est:r_L^inf_final} for $\sqrt{\rho}u$ and $\sqrt{\rho}$
respectively.

In view of \eqref{est:r_H^2} above, we are interested in $H^3$ bounds for $u$: for this, we will apply Proposition \ref{p:DFP} to both $u$ and $\nabla u$.
To this end, we fix the following values of the parameters:
\begin{align} \label{eq:p-r}
\big(p_2,r_2\big)\,=\,(2,4/3)\,,\qquad (p_0,r_0)\,=\,(+\infty,2)\,,\qquad (p_1,r_1)\,=\,(2,4)\,.
\end{align}
With these choices, all the hypotheses in Proposition \ref{p:DFP} are satisfied. Indeed, thanks to the energy estimates \eqref{est:en_est} and \eqref{est:u_L^2} and maximal regularity
\eqref{e.estPh}, one can easily verify that, \textsl{a priori}, both $u$ and $\nabla u$ belong to the space $W^{2,1}_{2,4/3}$. 
We choose $p_0=+\infty$ in order to have, thanks to \eqref{est:u-L^2} above, a control on $u$ in any $L^p$,
$p\in[2,+\infty]$. 
Notice also that we have some freedom for the values of $r_2$ and $p_1$, which then determine $r_0$ and $r_1$. Here, we take the simple choice $r_2=4/3$ and $p_1=2$. This implies $s_2=1/2$.

For $t\geq0$, let us introduce the quantity
\begin{align*}
\mc U(t)\,&:=\,\|u\|_{L^\infty_t(L^2)\cap L^2_t(L^\infty)}\,+\,\left\|\nabla u\right\|_{L^4_t(L^2)\cap L^2_t(L^\infty)} \\
&\qquad\qquad\qquad\,+\,\left\|\nabla^2u\right\|_{L^4_t(L^2)}\,+\,\left\|\nabla^3u\right\|_{L^{4/3}_t(L^2)}\,+\,\left\|\d_tu\right\|_{L^{4/3}_t(H^1)}\,.
\end{align*}
On $[0,T]$, where $T>0$ is the time
defined in \eqref{def:T}, we have
\begin{multline} \label{est:U_1}
\mc U(T)\,\leq\,C_0\Big(\left\|u_{in}\right\|_{\dot B^{1/2}_{2,4/3}}\,+\,\left\|\nabla u_{in}\right\|_{\dot B^{1/2}_{2,4/3}}\\
+\,\|r\d_tu+f\|_{L^{4/3}_T(H^1)}+\|u\|_{L^{4/3}_T(L^{2})}+\|\nabla u\|_{L^{4/3}_T(L^{2})}+\|\nabla^2 u\|_{L^{4/3}_T(L^{2})}\Big)\,,
\end{multline}
where $f$ is defined in \eqref{def:f}. Notice that
\begin{align*}
\left\|u_{in}\right\|_{\dot B^{1/2}_{2,4/3}}\,+\,\left\|\nabla u_{in}\right\|_{\dot B^{1/2}_{2,4/3}}\,\leq\,\left\|u_{in}\right\|_{\dot B^{1/2}_{2,4/3}\cap\dot B^{3/2}_{2,4/3}}\,\leq\,
\left\|u_{in}\right\|_{B^{3/2}_{2,4/3}}\,,
\end{align*}
where the last inequality holds in view of the fact that $B^s_{p,r}\,=\,\dot B^{s}_{p,r}\cap L^p$ for any $s>0$, see also Chapter 2 of \cite{B-C-D}. \\
\indent
Next, we bound the term containing the time derivative. Notice that, for any $j\in\{1,2,3\}$, we have
$$
\d_j(r\,\d_tu)\,=\,r\,\d_t\d_ju\,+\,\d_jr\,\d_tu\,.
$$
Therefore, 
we have
$$
\left\|r\,\d_tu\right\|_{L^{4/3}_T(L^2)}\,+\,\left\|r\,\d_t\nabla u\right\|_{L^{4/3}_T(L^2)}\,\leq\,4\,\eta\,\left\|\d_tu\right\|_{L^{4/3}_T(H^1)}\,,
$$
so this term can be absorbed in the left-hand side of \eqref{est:U_1}, if $\eta>0$ is fixed so that condition \eqref{eq:d-small} is fulfilled.
As for the remaining term $\d_jr\,\d_tu$, 
we notice first of all that from \eqref{est:r_L^inf_final}, the lower bound for $\bar\rho$, the momentum equation in \eqref{eq:wp_nse} and \eqref{est:en_est}, we can bound
\begin{align*}
\left\|\d_tu\right\|_{L^{2}_T(L^2)}\,&\leq\,C\left\|\rho\d_tu\right\|_{L^{2}_T(L^2)}\,\leq\,C\,\left\|\mc Lu+\rho u\cdot\nabla u+ \nabla P-\rho\nabla G 
\right\|_{L^{2}_T(L^2)}\,.
\end{align*}
At this point, we obviously have
$$
\left\|\mc Lu\right\|_{L^{2}_T(L^2)}\,\leq\,T^{(1/2)-(1/4)}\,\left\|\mc Lu\right\|_{L^4_T(L^2)}\,\leq\,C\,T^{1/4}\,\mc U(T).
$$
In addition, one has
\begin{align}\label{e.estunablau}
\left\|\rho u\cdot\nabla u\right\|_{L^{2}_T(L^2)}\,&\leq\,C\,T^{(1/2)-(1/2)}\,\left\|\sqrt{\rho}\,u\right\|_{L^\infty_T(L^2)}\,\left\|\nabla u\right\|_{L^2_T(L^\infty)}\,\leq\,
C\,\mc U(T)\,.
\end{align}
Let us now turn to the pressure and potential force terms. In view of \eqref{e.eqpotG}, we have
\begin{equation} \label{eq:P-G}
\nabla P(\rho)-\rho\nabla G \ =\ \nabla\big(P(\rho) - P(\bar\rho)\big) - (\rho-\bar\rho)\nabla G \ =\ \nabla\big(P'(\bar\rho+sr) r\big) - r\nabla G\,,
\end{equation} 
for some $s\in\,]0,1[\,$. On the one hand, since $G=H'(\bar\rho)$, $G$ is bounded; so, for a constant $C>0$ depending on $\|\bar\rho\|_{W^{1,\infty}}$ and on the function $P$, we have
\begin{equation*}
\left\|r\nabla G\right\|_{L^2_T(L^2)} \leq C\ T^{1/2}\ \|r\|_{L^\infty_T(L^2)}\,.
\end{equation*}
On the other hand, by writing
\begin{equation} \label{eq:DP}
 \nabla\big(P'(\bar\rho+sr) r\big) = P'(\bar\rho+sr)\nabla r + P''(\bar\rho+sr) r \nabla\bar\rho + s P''(\bar\rho+sr) r \nabla r
\end{equation}
and making use of \eqref{est:pressure} and \eqref{est:r_L^inf_final}, direct computations show that
\begin{align*}
& \left\|\nabla\big(P'(\bar\rho+sr) r\big)\right\|_{L^2_T(L^2)} \leq C\, T^{1/2} \|r\|_{L^\infty_T(H^1)}\,.
\end{align*}
Here, the constant $C>0$ depends on the constant appearing in \eqref{est:pressure}, on $\eta$ and $\|\bar\rho\|_{W^{1,\infty}}$.
Assuming without loss of generality that $T\leq 1$, we have thus proved that
\begin{equation} \label{est:d_tu}
\left\|\d_tu\right\|_{L^{2}_T(L^2)}\,\leq\,C\,\big(
\mc R(T)\,+\,\mc U(T)\big)\,.
\end{equation}
With this control at hand, let us return to the bound of $\d_jr\,\d_tu$ in $L^{4/3}_T(L^2)$, for $j\in\{1,2,3\}$. By resorting again to the interpolation inequality
\eqref{est:interp}, in view of the Sobolev embedding $\dot H^1\hookrightarrow L^6$ and of Young's inequality we get
\begin{align*}
\left\|\d_jr\,\d_tu\right\|_{L^{4/3}_T(L^2)}\,&\leq\,\|\nabla r\|_{L^\infty_T(L^2)}^{1/4}\;\|\nabla^2 r\|_{L^\infty_T(L^2)}^{3/4}\;\|\d_tu\|_{L^{4/3}_T(L^2)}^{1/4}\;\|\d_t\nabla u\|_{L^{4/3}_T(L^2)}^{3/4} \\
&\leq\,C\,\mc R(T)\,\left(T^{(3/4)-(1/2)}\,\|\d_tu\|_{L^{2}_T(L^2)}\right)^{\!1/4}\;\|\d_t\nabla u\|_{L^{4/3}_T(L^2)}^{3/4}\\
&\leq\,C\,T^{1/16}\,\mc R(T)\,\big(
\mc R(T)\,+\,\mc U(T)\big)^{1/4}\,\big(\mc U(T)\big)^{3/4} \\
&\leq\,C\,T^{1/16}\,\big(
\mc R(T)\,+\,\mc U(T)\big)^2\,.
\end{align*}
Inserting all the previous bounds in \eqref{est:U_1}, under condition \eqref{eq:d-small} we deduce that
\begin{align} \label{est:U_2}
\begin{split}
\mc U(T)\leq &C\Big(\left\|u_{in}\right\|_{B^{3/2}_{2,4/3}}+T^{1/16}\big(
\mc R(T)+\mc U(T)\big)^2+\|f\|_{L^{4/3}_T(H^1)}+\|u\|_{L^{4/3}_T(H^{2})}\Big)\\
\leq&C\Big(\left\|u_{in}\right\|_{B^{3/2}_{2,4/3}}\,+\,T^{1/16}\,\big(
1+\mc R(T)+\mc U(T)\big)^2+\|f\|_{L^{4/3}_T(H^1)}\Big).
\end{split}
\end{align}
Our next goal is to bound $f$ in the $L^{4/3}_T(H^1)$ norm. First of all, let us focus on the $L^{4/3}_T(L^2)$ norm. Repeating the computations which led to \eqref{est:d_tu}, we easily find that
$$
\|f\|_{L^{4/3}_T(L^2)}\,\leq\,C\left(T^{3/4}\,\mc R(T)\,+\,T^{1/4}\,\mc U(T)\right)\,\leq\,C\,T^{1/4}\,\big(\mc R(T)\,+\,\mc U(T)\big)\,,
$$
where again, without loss of generality, we have assumed that $T\leq 1$.
Next, for $j\in\{1,2,3\}$, from the definition \eqref{def:f} of $f$, we get
\begin{align}\label{e.eqdjf}
\begin{split}
\d_jf\,&=\,\d_j\rho\,u\cdot\nabla u\,+\,\rho\,\d_ju\cdot\nabla u\,+\,\rho\,u\cdot\nabla\d_ju+\partial_j\big(\nabla P(\rho)-\rho\nabla G\big).
\end{split}
\end{align}
We first bound the convective terms in \eqref{e.eqdjf}. By \eqref{est:r_L^inf_final}, interpolation between Lebesgue spaces and \eqref{est:interp}, we have
\begin{align*}
\left\|\rho\,u\cdot\nabla\d_ju\right\|_{L^{4/3}_T(L^2)}\,&\leq\,C\,T^{(3/4)-(1/2)}\,\|u\|_{L^4_T(L^4)}\,\|\nabla^2 u\|_{L^{4}_T(L^4)} \\
&\leq\,C\,T^{1/4}\,\|u\|_{L^\infty_T(L^2)}^{1/2}\,\|u\|^{1/2}_{L^2_T(L^\infty)}\,\|\nabla^2 u\|^{1/4}_{L^{4}_T(L^2)}\,\left\|\nabla^3 u\right\|^{3/4}_{L^{4}_T(L^2)} \\
&\leq\,C\,T^{1/4}\,\big(\mc U(T)\big)^2\,.
\end{align*}
Arguing in a similar way, we obtain
\begin{align*}
\left\|\rho\,\d_ju\cdot\nabla u\right\|_{L^{4/3}_T(L^2)}\,&\leq\,C\,T^{(3/4)-(1/2)}\,\left\|\nabla u\right\|^2_{L^4_T(L^4)} \\
&\leq\,C\,T^{1/4}\,\left(\|\nabla u\|^{1/4}_{L^{4}_T(L^2)}\,\left\|\nabla^2 u\right\|^{3/4}_{L^{4}_T(L^2)}\right)^2\ \leq\,C\,T^{1/4}\,\big(\mc U(T)\big)^2\,.
\end{align*}
Finally, writing $\rho=\bar\rho+r$, we split the term $\d_j\rho\,u\cdot\nabla u=\d_j\bar\rho\,u\cdot\nabla u+\d_jr\,u\cdot\nabla u$. We estimate 
on the one hand, similarly to \eqref{e.estunablau},
\begin{align*}
\left\|\d_j\bar\rho\,u\cdot\nabla u\right\|_{L^{4/3}_T(L^2)}\leq C 
T^\frac14\mc U(T),
\end{align*}
and on the other hand
\begin{align*}
\left\|\d_jr\,u\cdot\nabla u\right\|_{L^{4/3}_T(L^2)}\,&\leq\,C\,T^{(3/4)-(1/4)-(5/12)}\,\|\nabla r\|_{L^\infty_T(L^4)}\,\left\|\nabla u\right\|_{L^4_T(L^{6})}\,\left\|u\right\|_{L^{12/5}_T(L^{12})} \\
&\leq\,C\,T^{1/12}\,\left\|\nabla r\right\|^{1/4}_{L^\infty_T(L^2)}\,\left\|\nabla^2r\right\|^{3/4}_{L^\infty_T(L^2)} \\
&\qquad\qquad\qquad\qquad\qquad \times \left\|\nabla^2u\right\|_{L^4_T(L^{2})}\left\|u\right\|^{1/6}_{L^{\infty}_T(L^{2})}\,\left\|u\right\|^{5/6}_{L^{2}_T(L^{\infty})} \\
&\leq\,C\,T^{1/12}\,\mc R(T)\,\big(\mc U(T)\big)^2\,.
\end{align*}
We now turn to the last term in \eqref{e.eqdjf}. Thanks to \eqref{eq:P-G} and \eqref{eq:DP}, we have
\begin{align*}
\partial_j\big(\nabla P(\rho)-\rho\nabla G\big)=\ & P''\ (\d_j\bar\rho+s\d_jr)\nabla r + P' \d_j\nabla r + P'''\ (\d_j\bar\rho+s\d_jr) r \nabla\bar\rho \\
&+ P''\ \d_jr \nabla\bar\rho + P'' \ r \d_j\nabla\bar\rho + s P'''\ (\d_j\bar\rho+s\d_jr) r \nabla r \\
&+ sP''\ \d_jr \nabla r + sP'' \ r \d_j\nabla r-\partial_jr\nabla G-r\partial_j\nabla G
\end{align*}
where all the functions $P'$, $P''$ and $P'''$ are computed at the point $\bar\rho+sr$.
Making repeated use of \eqref{est:pressure}, \eqref{est:r_L^inf_final} and \eqref{est:interp},
\begin{equation*}
\big\|\partial_j\big(\nabla P(\rho)-\rho\nabla G\big)\big\|_{L^{4/3}_T(L^2)}\leq CT^{3/4}\left(\mc{R}(T) + \big(\mc R(T)\big)^2\right)\,,
\end{equation*}
where $C=C(P,\|\bar\rho\|_{W^{3,\infty}})>0$.  \\

\indent
Putting all those bounds together, assuming again $T\leq 1$, we deduce that 
$$
\|\nabla f\|_{L^{4/3}_T(L^2)}\,\leq\,C\,T^{1/12}\,\big(1\,+\,\mc R(T)\,+\,\mc U(T)\big)^3\,,
$$
which in turn implies that
\begin{equation} \label{est:f}
\|f\|_{L^{4/3}_T(H^1)}\,\leq\,C\,T^{1/12}\,\big(1\,+\,\mc R(T)\,+\,\mc U(T)\big)^3\,.
\end{equation}
In the end, inserting \eqref{est:f} into \eqref{est:U_2}, we find
\begin{equation} \label{est:U_final}
\mc U(T)\,\leq\,C\Big(\left\|u_{in}\right\|_{B^{3/2}_{2,4/3}}\,+\,T^{1/16}\,\big(
1\,+\,\mc R(T)\,+\,\mc U(T)\big)^3
\Big)\,.
\end{equation}
On the other hand, the integral term in \eqref{est:r_H^2} can be bounded by
\begin{align*}
\big(\|\bar\rho\|_{W^{3,\infty}}+\mc R(T)\big)\,\Big(T^{1/2}\|\nabla u\|_{L^2_T(L^\infty)}+T^{3/4}\left\|u\right\|_{L^4_T(H^2)}+T^{1/4}\left\|\nabla^3u\right\|_{L^{4/3}_T(L^2)}\Big)\,,
\end{align*}
which implies that
\begin{equation} \label{est:R_final}
\mc R(T)\,\leq\,C\left(\left\|r_{in}\right\|_{H^2}\,+\,T^{1/4}\,\big(1\,+\,\mc R(T)\,+\,\mc U(T)\big)^2\right)\,.
\end{equation}
Define now, for all $t\geq0$, the quantity
$$
\mc N(t)\,:=\,\mc R(t)\,+\,\mc U(t)\,.
$$
Summing up estimates \eqref{est:U_final} and \eqref{est:R_final}, we infer that
$$
\mc N(T)\,\leq\,C\left(\left\|u_{in}\right\|_{B^{3/2}_{2,4/3}}\,+\,\left\|r_{in}\right\|_{H^2}\,+\,T^{1/16}\,\Big(1\,+\,\mc N(T)\Big)^3\right)\,.
$$
From this inequality, it is a standard matter to deduce the existence of a time $T^*{(\gamma,\mu,\ep,\lambda,\|\bar\rho\|_{W^{3,\infty}},
\kappa,\|u_{in}\|_{B^{3/2}_{2,4/3}},\|r_{in}\|_{H^2})}>0$, with $T^*\leq\,\min\{1,T\}$, such that
$$
\mc N(T^*)\,\leq\,2\,C\Big(\left\|u_{in}\right\|_{B^{3/2}_{2,4/3}}\,+\,\left\|r_{in}\right\|_{H^2}\Big)\,.
$$
The \textsl{a priori} estimates are hence proved in the interval $[0,T^*]$.

\subsection{Uniqueness}
The uniqueness of solutions, claimed in Theorem \ref{th:wp}, is a consequence of the following statement.
\begin{proposition} \label{p:uniqueness}
{Let $\gamma\geq 1$.} Let $\bar\rho\in W^{3,\infty}(\R)$. Assume that $\bar\rho$ is uniformly bounded from below, i.e. $\bar\rho\geq\kappa>0$.  We define the potential by $G=H'(\bar\rho)$, where $H$ is defined by \eqref{def:H}. 
Let $\big(\rho_{in},u_{in}\big)$ be such that $\rho_{in}-\bar\rho\in H^2$ and $u_{in}\in B^{3/2}_{2,4/3}$, with $\left\|\rho_{in}\,-\,\bar\rho\right\|_{L^\infty}\,\leq\,\eta$, for some $\eta>0$ satisfying \eqref{eq:d-small}.
Assume that $\big(\rho^1,u^1\big)$ and $\big(\rho^2,u^2\big)$ are two solutions to system \eqref{eq:wp_nse}, related to the same initial datum $\big(\rho_{in},u_{in}\big)$ and belonging to the space
\begin{align*}
X_T\,:=\,&\Big\{\big(\rho,u\big)\in L^\infty_T(L^\infty)\times L^\infty_T(L^2)\;\big|\;\mbox{Properties (1)-(2) of Theorem \ref{th:wp} hold true}\Big\}\,,
\end{align*}
for some $T>0$.  \\
\indent
Then $\rho^1=\rho^2$ and $u^1=u^2$ almost everywhere in $[0,T]\times\Omega$.
\end{proposition}

We show first a simple lemma.
\begin{lemma} \label{l:r-u_cont}
 Let $\big(\rho_{in},u_{in}\big)$ be as in Proposition \ref{p:uniqueness} above, and let $\big(\rho,u\big)$ a solution to system \eqref{eq:wp_nse}, related to the initial datum $\big(\rho_{in},u_{in}\big)$
and belonging to the space $X_T$, for some $T>0$. \\
\indent
Then $\rho\in\mc C\big([0,T]\times\Omega\big)$ and $u\in\mc C\big([0,T];H^1(\Omega)\big)$. In addition, the following estimate holds true,
for a ``universal'' constant $C>0$:
$$
\left\|u\right\|^2_{L^\infty_T(H^1)}\,\leq\,C\,\left(\left\|u_{in}\right\|_{B^{3/2}_{2,4/3}}^2\,+\,\left\|\d_tu\right\|_{L^2_T(L^2)}\,\left\|u\right\|_{L^2_T(H^2)}\right)\,.
$$
\end{lemma}

\begin{proof}[Proof of Lemma \ref{l:r-u_cont}]
By definition of the space $X_T$, we know that $r\,:=\,\rho-\bar\rho$ belongs to $\mc C_T(H^2)$, which, by Sobolev embeddings, is continuously embedded in $\mc C_T(\mc C\cap L^\infty)$. Thus $\rho\,=\,\bar\rho+r$
is continuous with respect both $x$ and $t$. \\
\indent
Next, let us consider $u$. We have shown in \eqref{est:d_tu} above that $\d_tu\in L^2_T(L^2)$ and that $u\in L^2_T(H^2)$. From those properties, one easily derives (see e.g. Section 5.9 of \cite{Evans})
that $u\in\mc C_T(H^1)$. The quantitative estimate is a simple consequence of the bound given in \cite{Evans} (see Theorem 3 page 305 therein), combined with the embedding
$B^{3/2}_{2,4/3}\hookrightarrow H^1$. The lemma is hence proved.
\end{proof}

We can now prove uniqueness.

\begin{proof}[Proof of Proposition \ref{p:uniqueness}]
For $j=1,2$, let us define $r^j\,:=\,\rho^j-\bar\rho$. We also set
$$
\de r\,:=\,r^1\,-\,r^2=\,\rho^1\,-\,\rho^2\qquad\mbox{ and }\qquad \de u\,:=\,u^1\,-\,u^2\,.
$$
Let us deal with $\de r$ first. Its equation reads as follows:
\begin{align*}
\d_t\de r\,+\,u^1\cdot\nabla\de r\,+\,\de r\,\nabla\cdot u^1\,=\,-\left(\bar\rho\nabla\cdot\de u\,+\,\delta u_3\bar\rho'\,+\,\de u\cdot\nabla r^2\,+\,r^2\,\nabla\cdot\de u\right)\,.
\end{align*}
An $L^2$ estimate for this equation (see the beginning of Section \ref{ss:wp-density} above) yields, for all $t\in[0,T]$, the bound
\begin{align*}
\begin{split}
\left\|\de r(t)\right\|_{L^2}\,\leq\ &\left\|\delta r_{in}\right\|_{L^2}\,+\,\frac{1}{2}\int^t_0\|\delta r\|_{L^2}\,\|\nabla\cdot u^1\|_{L^\infty}\,d\tau\,\\
&+\,\int^t_0\Big(\left\|r^2\right\|_{L^\infty}\left\|\nabla\cdot\de u\right\|_{L^2}\,+\,
\left\|\de u\right\|_{L^\infty}\,\left\|\nabla r^2\right\|_{L^2}\Big)\,d\tau\\
&+\int_0^t\Big(\|\bar\rho'\|_{L^\infty}\|\delta u_3\|_{L^2}\,+\,\|\bar\rho\|_{L^\infty}\|\nabla\cdot\delta u\|_{L^2}\Big)\,d\tau.
\end{split}
\end{align*}
We also look for an $L^2$ bound on $\nabla\de r$. To this end, we differentiate the equation for $\de r$ with respect to the space
variables and we perform an energy estimate. We get
\begin{align}
\begin{split}
&\left\|\nabla\de r(t)\right\|_{L^2}\,\leq\,\left\|\nabla\de r_{in}\right\|_{L^2}\,+\,\frac{1}{2}\int^t_0\left\|\nabla\de r\right\|_{L^2}\,\left\|\nabla\cdot u^1\right\|_{L^\infty}\,d\tau \\ \label{est:de_r}
&\;+\int^t_0\Big(\left\|\nabla u^1\right\|_{L^\infty}\,\left\|\nabla\de r\right\|_{L^2}\,+\,\left\|\nabla\nabla\cdot u^1\right\|_{L^4}\,\left\|\de r\right\|_{L^4}\,+\,
\left\|\nabla\nabla\cdot\de u\right\|_{L^2}\,\left\|r^2\right\|_{L^\infty} \\
&\qquad\quad
+\,
\left\|\nabla\de u\right\|_{L^4}\,\left\|\nabla r^2\right\|_{L^4}\,+\,\left\|\de u\right\|_{L^\infty}\,\left\|\nabla^2r^2\right\|_{L^2}\,+\,\|\bar\rho\|_{W^{2,\infty}}\|\delta u\|_{H^2}\Big)\,d\tau\,.
\end{split}
\end{align}
Using the embedding $H^1(\R^3)\hookrightarrow L^4(\R^3)$ (see \eqref{est:interp} above), and summing up the previous estimate with \eqref{est:de_r}, we finally get
\begin{align}
\left\|\de r(t)\right\|_{H^1}\,&\leq\,\left\|\de r_{in}\right\|_{H^1}\,+\,C\int^t_0\left\|\de r\right\|_{H^1}\,\Big(\left\|\nabla u^1\right\|_{L^\infty}+\left\|\nabla^2u^1\right\|_{L^4}\Big)\,d\tau
\label{est:Dde_r} \\
&\;+C\int^t_0\Big(\big(\left\|r^2\right\|_{L^\infty}+\left\|\nabla r^2\right\|_{L^4}\big)\left(\left\|\nabla\de u\right\|_{L^2}+\left\|\nabla^2\de u\right\|_{L^2}\right) \nonumber \\
&\quad
+\,\left\|\de u\right\|_{L^\infty}\,\left\|\nabla^2r^2\right\|_{L^2}\,+\,\|\bar\rho\|_{W^{2,\infty}}\|\delta u\|_{H^2}\Big)\,d\tau\,. \nonumber
\end{align}
Observe that, by definition of the space $X_T$, we have $\nabla u^1\in L^2_T(L^\infty)$. In addition, from the control \eqref{est:interp} applied to $\left\|\nabla^2u^1\right\|_{L^4}$
and Young's inequality, we infer that $\nabla^2u^1\in L^{4/3}_T(L^4)$. By the same token we get $\nabla r^2\in L^\infty_T(L^4)$, while we already know that $r^2\in L^\infty_T(L^\infty)\cap L^\infty_T(H^2)$.
Hence, estimate \eqref{est:Dde_r} above tells us that the quantity
$$
\theta(t)\,:=\,\sup_{\tau\in[0,t]}\left\|\de r(\tau)\right\|_{H^1}
$$
verifies the following bound, for a suitable exponent $\alpha>0$:
\begin{align} \label{est:gamma}
\theta(t)\,&\leq\,C\,\left(t^{\alpha}\,\theta(t)\,+\,\int^t_0\left(\left\|\de u\right\|_{H^2}+\left\|\de u\right\|_{L^\infty}\right)\,d\tau\right)\,,
\end{align}
for all $t\in\big[0,\min\{T,1\}\big]$. Notice that here we used the fact that $\de r_{in}=0$. \\
\indent
We now turn to velocity estimates. Taking the difference of the equations for $u^1$ and $u^2$ yields the following equation for $\de u$:
\begin{align} 
&\rho^1\,\d_t\de u\,+\,\rho^1\,u^1\cdot\nabla \de u\,+\,\mc L\de u
\label{eq:de_u} \\
&\quad
=\,-\,\de r\,\d_tu^2\,-\,\nabla\big(P(\rho^1)-P(\rho^2)\big)+\delta r\nabla G\,-\,\big(\rho^1u^1-\rho^2u^2)\cdot\nabla u^2\,
. \nonumber
\end{align}
We observe that
\begin{align*}
\big(\rho^1u^1-\rho^2u^2)\cdot\nabla u^2\,&=\,\rho^1\,\de u\cdot\nabla u^2\,+\,\de r\,u^2\cdot\nabla u^2 \\
\nabla\big(P(\rho^1)-P(\rho^2)\big)\,&=\,P'(\rho^1)\,\nabla\de r\,+\,P''(\zeta)\,\de r\,\nabla(\bar\rho+r^2)\,,
\end{align*}
for some $\zeta=\zeta(x,t)$ in between the values of $\rho^1(x,t)$ and $\rho^2(x,t)$. 
A basic energy estimate for the equation then gives, for any $t\in[0,T]$, the bound
\begin{align*}
&\frac{1}{2}\,\frac{d}{dt}\left\|\sqrt{\rho^1}\,\de u\right\|^2_{L^2}+\int_\Omega\left|\nabla_{\mu,\veps}\de u\right|^2\,+\,\lambda\,\int_\Omega\left|\nabla\cdot \de u\right|^2 \\
&\quad\leq\,\left\|\de u\right\|_{L^2}\Big(\left\|\de r\right\|_{L^4}\,\left\|\d_tu^2\right\|_{L^4}\,+\,
\left\|\rho^1\right\|_{L^\infty}\,\left\|\de u\right\|_{L^2}\,\left\|\nabla u^2\right\|_{L^\infty} \\
&\qquad
+\,\left\|\de r\right\|_{L^4}\,\left\|u^2\right\|_{L^4}\,\left\|\nabla u^2\right\|_{L^\infty}\,+\,C\left\|\nabla\de r\right\|_{L^2}\,+\,C
\left\|\de r\right\|_{L^2}+C\left\|\de r\right\|_{L^4}\left\|\nabla r^2\right\|_{L^4}\Big)\,,
\end{align*}
where we have also used the $L^\infty_T(L^\infty)$ boundedness of $\rho^1$ and $\rho^2$ in order to control the terms involving derivatives of the pressure function.
Let us forget about the viscosity terms for a while. A simple argument allows one to deduce the following control, for any $t\in[0,T]$:
\begin{align*}
\left\|\de u\right\|_{L^\infty_t(L^2)}\,&\leq\,C\int^t_0\Big(\left\|\de u\right\|_{L^2}\,\left\|\nabla u^2\right\|_{L^\infty} \\
&\qquad\qquad +\,\theta(\tau)\,\big(1+\left\|\d_tu^2\right\|_{H^1}+\left\|u^2\right\|_{L^4}\left\|\nabla u^2\right\|_{L^\infty}+
\|\nabla r^2\|_{L^4}\big)\Big)\,d\tau\,,
\end{align*}
where we have used the fact that $\de u_{in}=0$. At this point, we observe that $u^2$ and $\nabla u^2$ both belong to $L^2_T(L^\infty)$. Moreover, by \eqref{est:interp} and the
fact that $u^2\in L^\infty_T(L^2)$ and $\nabla u^2\in L^4_T(L^2)$, one gathers that $u^2\in L^{16/3}_T(L^4)$. Finally, $\d_tu^2\in L^{4/3}_T(H^1)$.
In the end, similarly to what is done in \eqref{est:gamma}, we deduce
the existence of a positive exponent, that we keep calling $\alpha$ without loss of generality, such that
\begin{align} \label{est:de-u_energy}
\left\|\de u\right\|_{L^\infty_t(L^2)}\,&\leq\,C\,t^\alpha\,\left(\theta(t)\,+\,\left\|\de u\right\|_{L^\infty_t(L^2)}\right)
\end{align}
for all $t\in\big[0,\min\{T,1\}\big]$. \\
\indent
As a last step, we rewrite equation \eqref{eq:de_u} in the following form:
\begin{equation} \label{eq:de_u_MR}
\bar\rho\d_t\de u\,+\,\mc L\de u\,=\,-\left(r^1\,\d_t\de u\,+\,\de r\,\d_tu^2\,+\,\de f\right)\,,
\end{equation}
where we have defined
\begin{align*}
\de f\,&:=\,\rho^1\,u^1\cdot\nabla\de u\,+\,\rho^1\,\de u\cdot\nabla u^2\,+\,\de r\,u^2\cdot\nabla u^2 \\
&\qquad+\,P'\big(\rho^1\big)\,\nabla\de r\,+\,P''(\zeta)\,\de r\,\nabla(\bar\rho+r^2)-\delta r\nabla G\,.
\end{align*}
Applying Proposition \ref{p:DFP} to equation \eqref{eq:de_u_MR}, with the choice \eqref{eq:p-r} of the parameters, and using the smallness condition \eqref{eq:d-small}, we get,
for any $t\in[0,T]$, the inequality
\begin{align} \label{est:max-reg_u}
\begin{split}
&\left\|\de u\right\|_{L^{2}_t(L^{\infty})}\,+\,\left\|\nabla\de u\right\|_{L^{4}_t(L^{2})}\,+\,\left\|\left(\d_t\de u,\nabla^2\de u\right)\right\|_{L^{4/3}_t(L^{2})}\\
&\,\leq\, C\,\Big(\left\|\de u_{in}\right\|_{\dot B^{1/2}_{2,4/3}}\,+\,\left\|\de r\,\d_tu^2\right\|_{L^{4/3}_t(L^2)}\,+\,\|\de f\|_{L^{4/3}_t(L^{2})}\\
&\qquad\qquad\qquad\qquad\qquad\qquad\qquad\qquad\qquad
+\|\de u\|_{L^{4/3}_T(L^{2})}+\|\nabla \de u\|_{L^{4/3}_T(L^{2})}\Big)\,.
\end{split}
\end{align}
Recall that $\d_tu^2\in L^{2}_T(L^2)$ by virtue of \eqref{est:d_tu} and $\d_t\nabla u^2\in L^{4/3}_T(L^2)$ by definition of $X_T$.
Hence, using Sobolev embedding and interpolation \eqref{est:interp}, we deduce that $\d_tu^2\in L^{16/11}_T(L^4)$. Therefore, we can bound
\begin{align*}
\left\|\de r\,\d_tu^2\right\|_{L^{4/3}_t(L^2)}\,&\leq\,t^{3/4 - 11/16}\,\left\|\de r\right\|_{L^{\infty}_t(L^4)}\,\left\|\d_tu^2\right\|_{L^{16/11}_t(L^4)} \\
&\leq\,t^{1/16}\,\left\|\de r\right\|_{L^{\infty}_t(H^1)}\,\left\|\d_tu^2\right\|^{1/4}_{L^{2}_T(L^2)}\,\left\|\d_t\nabla u^2\right\|^{3/4}_{L^{4/3}_T(L^2)}\,\leq\,C\,t^{1/16}\,\theta(t)
\,.
\end{align*}
Next, we are going to bound $\de f$ in $L^{4/3}_t(L^2)$. This was already done for the energy estimate \eqref{est:de-u_energy} above, but here we have to take special care of the integrability
in time. For the pressure terms and the potential force term we have
\begin{align*}
&\left\|P''(\zeta)\,\de r\,\nabla(\bar\rho+r^2)\right\|_{L^{4/3}_t(L^2)} \\
&\qquad\leq\ C\,t^{3/4}\,\left(\left\|\de r\right\|_{L^\infty_t(L^2)}\|\nabla\bar\rho\|_{L^\infty_T(L^\infty)}+\left\|\de r\right\|_{L^\infty_t(L^4)}\left\|\nabla r^2\right\|_{L^\infty_T(L^4)}\right)
\leq\ C\,t^{3/4}\,\theta(t) \\
&\left\|P'\big(\rho^1\big)\,\nabla\de r\right\|_{L^{4/3}_t(L^2)}\,\leq\ C\,t^{3/4}\,\left\|\nabla\de r\right\|_{L^\infty_t(L^2)}\,\leq\,C\,t^{3/4}\,\theta(t) \\
&\|\delta r\nabla G\|_{L^{4/3}_t(L^2)}\leq Ct^{3/4}\theta(t)\,.
\end{align*}
Thus, ut remains us to bound the convective terms. First of all, recall that we showed above that $\left\|u^2\right\|_{L^4}\left\|\nabla u^2\right\|_{L^\infty}$ belongs to $L^{16/11}_T$. Therefore, 
\begin{align*}
\left\|\de r\,u^2\cdot\nabla u^2\right\|_{L^{4/3}_t(L^2)}\,\leq\,t^{(3/4)-(11/16)}\left\|\de r\right\|_{L^\infty_t(L^4)}\left\|u^2\cdot\nabla u^2\right\|_{L^{16/11}_T(L^4)}\,\leq\,
C t^{1/16}\theta(t)\,.
\end{align*}
In addition, we have
\begin{align*}
\left\|\rho^1\,\de u\cdot\nabla u^2\right\|_{L^{4/3}_t(L^2)}\,&\leq\,t^{(3/4)-(1/2)}\,\left\|\rho^1\right\|_{L^\infty_T(L^\infty)}\,\left\|\de u\right\|_{L^\infty_t(L^2)}\,
\left\|\nabla u^2\right\|_{L^2_T(L^\infty)} \\
&\leq\,C\,t^{1/4}\,\left\|\de u\right\|_{L^\infty_t(L^2)}\,.
\end{align*}
Finally, for the last term we use the following Gagliardo-Nirenberg type inequality (see e.g. Lemma II.3.3 of \cite{Galdi}):
$
\left\|u\right\|_{L^\infty}\,\leq\,\left\|u\right\|_{L^2}^{1/4}\,\left\|\nabla^2 u\right\|_{L^2}^{3/4}\,.
$
This, combined with Young's inequality, implies $u^2\in L^4_T(L^\infty)$ (recall the definition of the space $X_T$). Using this bound, we can estimate
\begin{align*}
\left\|\rho^1\,u^1\cdot\nabla\de u\right\|_{L^{4/3}_t(L^2)}\,&\leq\,t^{(3/4)-(1/2)}\left\|\rho^1\right\|_{L^\infty_T(L^\infty)}\,\left\|u^1\right\|_{L^4_T(L^\infty)}\,\left\|\nabla \de u\right\|_{L^4_t(L^2)} \\
&\leq\,C\,t^{1/4}\,\left\|\nabla \de u\right\|_{L^4_t(L^2)}\,.
\end{align*}
Putting all the previous estimates into \eqref{est:max-reg_u}, we have shown that there exists a positive exponent $\alpha>0$ for which, for all $t\in\big[0,\min\{1,T\}\big]$, one has
\begin{align} \label{est:de-u_MaxReg}
&\left\|\de u\right\|_{L^{2}_t(L^{\infty})}\,+\,\left\|\nabla\de u\right\|_{L^{4}_t(L^{2})}\,+\,\left\|\left(\d_t\de u,\nabla^2\de u\right)\right\|_{L^{4/3}_t(L^{2})} \\
&\qquad\qquad\qquad\qquad\qquad\qquad \leq\,C\,t^{\alpha}\,\Big(\theta(t)\,+\,\left\|\de u\right\|_{L^\infty_t(L^2)}\,+\,\left\|\nabla \de u\right\|_{L^4_t(L^2)}\Big)\,. \nonumber
\end{align}
\indent Let us now introduce the quantity
$$
\mc D(t)\,:=\,\left\|\de u\right\|_{L^\infty_t(L^2)}\,+\,\left\|\de u\right\|_{L^{2}_t(L^{\infty})}\,+\,\left\|\nabla\de u\right\|_{L^{4}_t(L^{2})}\,+\,
\left\|\left(\d_t\de u,\nabla^2\de u\right)\right\|_{L^{4/3}_t(L^{2})}\,.
$$
Summing up inequalities \eqref{est:gamma}, \eqref{est:de-u_energy} and \eqref{est:de-u_MaxReg}, we finally deduce that, for all $t\in\big[0,\min\{1,T\}\big]$, we have
$$
\theta(t)\,+\,\mc D(t)\,\leq\,C\,t^{\alpha}\,\big(\theta(t)\,+\,\mc D(t)\big)\,,
$$
for a suitable exponent $\alpha>0$. Therefore, if $t$ is now small enough, we can absorbe the right hand side into the left-hand side, deducing that both $\theta(t)$ and $\mc D(t)$ have to be $0$.
In particular, we deduce that $\rho^1\equiv\rho^2$ and $u^1\equiv u^2$ almost everywhere on $[0,t]\times\Omega$.  In this way, we also infer that
$$
\left\|\nabla\de u\right\|_{L^2_t(L^\infty)}\,+\,\left\|\left(\nabla^3\de u\,,\,\d_t\nabla\de u\right)\right\|_{L^{4/3}_t(L^2)}\,=\,0\,,
$$
from which we deduce that
$$
\left\|\de u\right\|_{L^\infty_t(H^1)}\,=\,0\,,
$$
where we have used also \eqref{est:d_tu} and the estimate in Lemma \ref{l:r-u_cont}.  \\
\indent
To complete the argument, let us define the set
$$
I\,:=\,\left\{t\in[0,T]\;\big|\quad \left\|\rho^1(\tau)-\rho^2(\tau)\right\|_{H^1}\,+\,\left\|u^1(\tau)-u^2(\tau)\right\|_{H^1}\,=\,0,\quad \forall\;\tau\in[0,t]\right\}\,.
$$
Of course, $I\neq\emptyset$, since $0\in I$. In addition,
the previous argument, combined with Lemma \ref{l:r-u_cont}, shows that $I$ is open. On the other hand, by continuity in time of the norms appearing in the definition of the space $X_T$
(see again Lemma \ref{l:r-u_cont} above), we infer that
$I$ is also closed. Then, by connectedness, we must have $I=[0,T]$. This completes the proof of the proposition.
\end{proof}

\section{Quantitative asymptotic analysis with stratification effects and anisotropic diffusion} \label{sec.quant}

The goal of this section is to analyze the structure of the solutions to the highly rotating compressible system \eqref{e.nse} with vertical stratification.
The main result of this part is contained in Theorem \ref{prop.main}, see Subsection \ref{ss:rel-entropy}: there we derive an asymptotic expansion and quantify the error in terms of the parameter $\ep$.
This stability result relies on relative entropy estimates.

\subsection{Formal asymptotic expansion}
\label{sec.formal}

In this section we perform formal computations in order to have a grasp on the structure of the solutions to sytem \eqref{e.nse}. Because of the no-slip boundary conditions \eqref{hyp:bound-cond},
boundary layers appear in the limit $\ep\rightarrow 0$ both in the vicinity of the top boundary $\R^2\times\{1\}$ and of the bottom boundary $\R^2\times\{0\}$. 

We will specify later on the precise hypotheses on the initial conditions
\begin{equation*}
\rho_{|t=0}\,=\,\rho_{in}\qquad\mbox{ and }\qquad u_{|t=0}\,=\,u_{in}
\end{equation*}
at time $t=0$.
For the purpose of the formal analysis, let us say in a loose way that we impose the following \emph{far field conditions}, for $|x|\ra\infty$:
$$
\rho_{in}(x)\,\longrightarrow\,\oline\rho(x_3)\qquad\mbox{ and }\qquad u_{in}(x)\,\longrightarrow\,0\,,
$$
where $\oline\rho$ is a strictly positive function satisfying the logistic equation
\begin{equation} \label{eq:bar-rho}
\nabla P(\oline\rho)\,=\,\oline\rho\,\nabla G\,.
\end{equation}
In addition, the initial densities are assumed to be far away from vacuum.
Moreover, we focus on \emph{well-prepared initial data}, in the sense specified in Subsection \ref{ss:rel-entropy} below.

\subsubsection{Construction of the ansatz} \label{ss:ansatz}

We start with expanding the solution $(u^\ep,\rho^\ep)$ to \eqref{e.nse} as
\begin{align}
\begin{split}
\label{e.ansatz}
u^\ep\,&=\,u_0(x_h,x_3,t)+u_{0,b}^{bl}(x_h,\tfrac{x_3}\ep,t)+u_{0,t}^{bl}(x_h,\tfrac{1-x_3}\ep,t)\\
&\qquad+\ep\left(u_1(x_h,x_3,t)+u_{1,b}^{bl}(x_h,\tfrac{x_3}\ep,t)+u_{1,t}^{bl}(x_h,\tfrac{1-x_3}\ep,t)\right)+O(\ep^2)\\
\rho^\ep\,&=\,\rho_0(x_h,x_3,t)+\rho_{0,b}^{bl}(x_h,\tfrac{x_3}\ep,t)+\rho_{0,t}^{bl}(x_h,\tfrac{1-x_3}\ep,t) \\
&\qquad+\ep\left(\rho_1(x_h,x_3,t)+\rho_{1,b}^{bl}(x_h,\tfrac{x_3}\ep,t)+\rho_{1,t}^{bl}(x_h,\tfrac{1-x_3}\ep,t)\right)\\&\qquad
+\ep^2\left(\rho_2(x_h,x_3,t)+\rho_{2,b}^{bl}(x_h,\tfrac{x_3}\ep,t)+\rho_{2,t}^{bl}(x_h,\tfrac{1-x_3}\ep,t)\right)+O(\ep^3)\,.
\end{split}
\end{align}
The superscript $bl$ stands for ``boundary layer'', while the subscripts $b$ and $t$ stand for ``bottom'' and ``top'' respectively.
For simplicity of the presentation, in the next computations we are going to consider only the boundary layer near the bottom, since the terms related to the top boundary layer
are dealt with in the exact same way. Therefore, from now on we omit the subscript $b$ for the boundary layer terms. However, when needed, we will explicitly write $t$ or $b$ subscripts to avoid confusion.\\
Below we denote by $\zeta=\tfrac {x_3}\ep$ the fast vertical variable in the boundary layer. The boundary layer profiles are supposed to decay to $0$ at exponential rate when $\z\ra\infty$, since their effect is almost negligible in the interior of the domain: we will use this fact repeatedly in the following computations.\\
We remark that, at this level, \eqref{e.ansatz} is just a formal ansatz.
As is usual, we will first formally derive the equations for the profiles: this is the purpose of the present section. After that, we will prove quantitative estimates for the difference
between the solution and the profiles we have constructed, using the relative entropy method: this will be done in Section \ref{sec.quant}.

\noindent\underline{Identification of the profiles.}
In order to identify the profiles, we plug the ansatz \eqref{e.ansatz} into \eqref{e.nse} and identify the terms of the same order of magnitude in $\ep$.
We immediately notice that the highest order term is a term of order $\veps^{-3}$, which appears in the third component of the momentum equation:
$$
P'(\rho_0+\rho_0^{bl})\,\d_\z\rho_0^{bl}\,=\,0\,.
$$
We assume that $\rho_0+\rho_0^{bl}$ stays bounded away from zero. This hypothesis is fully justified here below.
In view of the hypothesis \eqref{hyp:p} on the pressure 
and the fact that $\rho_0^{bl}$ has to vanish for $\z\ra\infty$, we immediately deduce that $\rho_0^{bl}\,\equiv\,0$. 
Thanks to that property, and ignoring the terms of order $O(\ep^2)$, which have been neglected in \eqref{e.ansatz} in the expansion of the velocity fields,
we find the following cascade of equations: from the conservation of mass equation, we get
\begin{align}
&\rho_0\partial_\zeta u_{0,3}^{bl}\,=\,0\label{e.ep-1a} \tag{mass-$\ep^{-1}$}\\
&\partial_t\rho_0+\nabla_h\cdot\left(\rho_0(u_{0,h}+u_{0,h}^{bl})\right)+\partial_3(\rho_0u_{0,3})+ \partial_3\rho_0\,u_{0,3}^{bl}\label{e.ep0a}\tag{mass-$\ep^{0}$}\\
&\qquad\qquad\qquad\qquad+\rho_1\partial_\zeta u_{0,3}^{bl}+\partial_\zeta(\rho_1^{bl}u_{0,3}^{bl})+ \partial_\zeta\rho_1^{bl}\, u_{0,3}
+\rho_0\partial_\zeta u_{1,3}^{bl}=0\,, \nonumber
\end{align}
and from the momentum equation we get
\begin{align}
&\nabla P(\rho_0)+ \left(\begin{array}{c}0\\[1ex]
P'(\rho_0)\partial_\zeta\rho_1^{bl}\end{array}\right) =\left(\begin{array}{c}0\\[1ex]{\lambda}\partial_\zeta^2u_{0,3}^{bl}\end{array}\right) + \rho_0 \nabla G\label{e.ep-2}
\tag{mom-$\ep^{-2}$}
\end{align}
\begin{align}
&\rho_0(u_{0,3}+u_{0,3}^{bl})\cdot\partial_\zeta u_{0}^{bl}+e_3\times\rho_0(u_0+u_0^{bl})+\label{e.ep-1}
\tag{mom-$\ep^{-1}$}\\[1ex]
&\left(\begin{array}{c}\nabla_h(P'(\rho_0)(\rho_1+\rho_1^{bl}))\\[1ex]
\partial_3(P'(\rho_0)\rho_1) +\partial_3(P'(\rho_0))\rho_1^{bl}\, +\,P''(\rho_0)\,\rho_1^{bl}\,\partial_\zeta\rho_1^{bl}\,+\,P'(\rho_0)\,\partial_\zeta \rho_2^{bl}\end{array}\right)\notag\\
&=\partial_\zeta^2\left(\begin{array}{c}
u_{0,h}^{bl}\\[1ex] 0
\end{array}\right)+\lambda\left(\begin{array}{c}
\nabla_h\partial_\zeta u_{0,3}^{bl}\\[1ex]
\partial_\zeta\nabla_h\cdot u_{0,h}^{bl}+\partial_\zeta^2u_{1,3}^{bl}
\end{array}\right)\,+ (\rho_1 + \rho_1^{bl})\nabla G.\notag
\end{align}

We will examine the equation at order $O(\ep^0)$ coming from the momentum equation later. Let us first infer some properties for the profiles. 

\noindent\underline{The terms in the interior.} Recall that the boundary layer profiles are expected to go to zero when $\z\ra \infty$. Therefore, it follows from \eqref{e.ep-2} that
\begin{equation} \label{eq:rho_0}
\nabla P(\rho_0)\,=\rho_0 \, \nabla G\,,
\end{equation}
which yields, by using \eqref{hyp:p}, the properties
\begin{equation}\label{e.static}
H'(\rho_0)= G + c(t) \qquad \qquad \mbox{ and }\qquad\qquad \nabla_h\rho_0=0\,.
\end{equation}
Hence $\rho_0$ is independent of $x_h$, namely $\rho_0\,=\,\rho_0(x_3,t)$, and satisfies the ODE 
\begin{equation}\label{e.oderho0}
P'(\rho_0)\partial_3 \rho_0= -\rho_0.
\end{equation}
Since $P'\in \mc C^1\left((0,\infty)\right)$ and non-zero, we can use Cauchy-Lipschitz theorem to get that $\rho_0(t)\in \mc C^1\left((0,1)\right)$, and hence bounded.
Moreover, from \eqref{e.ep-1} we infer
\begin{equation*}
\rho_0 \left(\begin{array}{c}
u_{0,h}^\perp\\[1ex] 0
\end{array}\right)+\left(\begin{array}{c}
P'(\rho_0)\nabla_h\rho_1\\[1ex] \partial_3(P'(\rho_0)\rho_1)
\end{array}\right)=\,\rho_1 \, \nabla G\,.
\end{equation*}
This equation is called \emph{geostrophic balance}; it implies the Taylor-Proudman theorem (see Section \ref{s:intro}). 
In particular, its third component reads
\begin{equation*}
\partial_3(P'(\rho_0)\rho_1)\,=\,-\rho_1\,.
\end{equation*}
Using the previous relation together with \eqref{eq:rho_0}, we get 
\begin{equation}\label{e.qindepx3}
\partial_3\left(\dfrac{P'(\rho_0)}{\rho_0}\rho_1\right)= 0\,,
\end{equation} hence the quantity 
\begin{equation}\label{e.defq}
Q:=\frac{P'(\rho_0)}{\rho_0}\rho_1\qquad\mbox{ verifies }\qquad Q\,=\,Q(t,x_h)\,,
\end{equation} 
\textit{i.e.} $Q$ is independent of the vertical variable. From the horizontal component, instead, we get (recall that $\nabla_h\rho_0=0$)
\begin{equation}\label{e.u0h}
u_{0,h}\,=\,\nabla_h^\perp\,\left(\frac{P'(\rho_0)}{\rho_0}\,\rho_1\right)\,=\,\nabla^\perp_hQ\,.
\end{equation}
In particular, we deduce that $u_{0,h}=u_{0,h}(x_h,t)$, which justifies the introduction of boundary layer terms in order to enforce the no-slip boundary conditions on $x_3=0,\ 1$.
In addition, applying the horizontal divergence we obtain
\begin{equation*}
\nabla_h\cdot u_{0,h}\,=\,\nabla_h\cdot\nabla_h^\perp\left(\frac{P'(\rho_0)}{\rho_0}\,\rho_1\right)\,=\,0\,,
\end{equation*}
so that $u_{0,h}$ is a $2$-D horizontal divergence-free vector field.

We now exploit \eqref{e.ep0a}: considering it in the interior of the domain (i.e., neglecting the boundary terms) and using the inequalities just proved, after an integration
in the vertical variable we infer that
\begin{equation}\label{e.averagerho0}
\int_0^1\partial_t\rho_0\,dx_3\,=\,-\int_0^1\partial_3\left(\rho_0\,u_{0,3}\right)\,dx_3\,=\,0\,.
\end{equation}
By taking the time derivative of \eqref{e.static} and using \eqref{hyp:p} we have 
\begin{equation*}
\partial_t \rho_0= \frac{\partial_t c}{H''(\rho_0)}
\end{equation*} We integrate in the vertical variable and, from \eqref{e.averagerho0}, we get $\partial_t c=0$, hence $\partial_t \rho_0=0$.
This implies that $\rho_0$ has to be independent also of time, and hence it is equal to a positive function $\oline\rho(x_3)$, solution of \eqref{eq:rho_0}, or equivalently \eqref{eq:bar-rho}.
Thanks to this fact, we have now that $\partial_3(\rho_0 \,u_{0,3})=0$. Using the no-slip boundary condition and the positivity of $\rho_0\,=\,\oline \rho$, we find $u_{0,3}\equiv0$.\\
\emph{From now on $\bar\rho$ denotes $\rho_0$}.
Let us now consider the equations outside the boundary layers, at order $O(\ep)$ in the mass equation, \begin{equation}\label{e.ep1a}\tag{mass-$\ep^1$}
\partial_t\rho_1\,+\,\nabla\cdot (\oline \rho \,u_1)+\nabla_h\cdot(\rho_1 \,u_{0,h} )=0\,,
\end{equation}
and at order $O(\ep^0)$ in the momentum equation,
\begin{align}\label{e.ep0}\tag{mom-$\ep^0$}
&\oline \rho\,\partial_t u_0+\nabla\cdot(\oline \rho\,u_0\otimes u_0)+e_3\times\left(\rho_1u_0+\oline \rho\,u_1\right) \\
&+\,\nabla\left(\frac{P''(\oline \rho)}{2}\rho_1^2 + P'(\oline\rho)\,\rho_2\right)\,=\,\mu\Delta_hu_0+\lambda\nabla(\nabla\cdot u_0)\,+\,\rho_2\,\nabla G. \notag
\end{align}
Recall that $u_0\,=\,\left(u_{0,h}(t,x_h),0\right)$.
Taking the curl of the horizontal component in \eqref{e.ep0}, we obtain an equation for the horizontal vorticity $\omega_0=\nabla_h^\perp\cdot u_{0,h}$:
\begin{align*}
\oline \rho\,\partial_t\omega_0+\oline \rho\,u_{0,h}\cdot\nabla_h\omega_0+\nabla_h\cdot(\rho_1\,u_{0,h})+\nabla_h\cdot (\oline \rho\,u_{1,h})-\mu\Delta_h\omega_0\,=\,0\,.
\end{align*}
Notice that, by \eqref{e.u0h}, we get
\begin{equation*}
\omega_0\,=\,\omega_0(t,x_h)\,=\,\Delta_hQ\,, 
\end{equation*}
where $Q$ is defined in \eqref{e.defq}; from the previous relation it follows that
\begin{equation}\label{e.qga}
\oline\rho\,\partial_t\Delta_h Q+\,\oline\rho\,\nabla_h^\perp Q\cdot\nabla_h\Delta_h Q+\nabla_h\cdot(\oline\rho\, u_{1,h})-\mu\Delta_h^2 Q=0\,,
\end{equation}
where we have used the cancellation
\begin{equation}\label{e.cancel}
\nabla_h\cdot(\rho_1\nabla_h^\perp\rho_1)=\dfrac12\nabla_h\cdot\nabla_h^\perp(\rho_1^2)=0
\end{equation}
in order to get rid of the term $\nabla_h\cdot(\rho_1\,u_{0,h})$.
In order to compute the term $\nabla_h\cdot (\oline\rho\,u_{1,h})$ in \eqref{e.qga}, we use equation \eqref{e.ep1a} and the cancellation \eqref{e.cancel} again: we find
\begin{equation}\label{e.d3u13}
\nabla_h\cdot \left(\oline \rho\, u_{1,h}\right)=-\partial_t\rho_1-\nabla_h\cdot(\rho_1u_{0,h})-\partial_3(\oline\rho \,u_{1,3})=-\partial_t\rho_1-\partial_3(\oline \rho\,u_{1,3})\,.
\end{equation}
After integrating in $x_3$ both \eqref{e.qga} and \eqref{e.d3u13} and summing up the resulting expressions, we eventually obtain
\begin{equation}\label{e.qgb}
\begin{aligned}
\partial_t\left(\left\langle\oline\rho\right\rangle\Delta_h Q-\left\langle \tfrac{\oline\rho}{P'(\oline\rho)}\right\rangle Q\right)&+
\left\langle\oline\rho\right\rangle\nabla_h^\perp Q\cdot\nabla_h\Delta_h Q-\mu\Delta_h^2 Q\\
&=\oline\rho(1)\,u_{1,3}(x_h,1,t)-\oline\rho(0)\,u_{1,3}(x_h,0,t)\,,
\end{aligned}
\end{equation}
where $\langle f\rangle=\int_{0}^{1}f(x_3)\,dx_3$ denotes the vertical mean of $f$.

\noindent\underline{Boundary layer terms.} 
We now consider the boundary layer terms. These terms are crucial to compute the right hand side of \eqref{e.qgb}: indeed
\begin{align}
\label{e.bdarycondu2}
u_{j,3,b}^{bl}(x_h,0,t)=-u_{j,3}(x_h,0,t)\qquad\mbox{ and }\qquad u_{j,3,t}^{bl}(x_h,0,t)=-u_{j,3}(x_h,1,t)
\end{align}
for $j=0,1$, in order to enforce the no-slip boundary condition on the bottom and top boundaries.\\ 
First of all, \eqref{e.ep-1a} yields $u_{0,3}^{bl}=u_{0,3}^{bl}(x_h,t)$, and hence 
\begin{equation}\label{e.u0bl}
u_{0,3}^{bl}\,\equiv\,0\,.
\end{equation}
Using \eqref{e.u0bl}, we obtain from \eqref{e.ep-2} that
\begin{equation*}
P'(\oline\rho)\,\partial_\zeta\rho_1^{bl}=\lambda\,\partial_\zeta^2u_{0,3}^{bl}=0\,.
\end{equation*}
Hence, thanks to \eqref{hyp:p}, $\rho_1^{bl}=\rho_1^{bl}(x_h,t)$ is constant in the boundary layer and goes to zero when $\zeta\rightarrow\infty$, therefore $\rho_1^{bl}\equiv0$.
Taking into account this last equality and reading the horizontal component of \eqref{e.ep-1}, one has
\begin{equation}\label{e.bl}
\oline \rho\,(u_{0,h}^{bl})^\perp=\partial_\zeta^2u_{0,h}^{bl}\,.
\end{equation} Notice that, in \eqref{e.bl}, $x_h$ is a parameter. 
We use Taylor formula at first order
\begin{equation*}
\oline\rho(x_3)\,=\, \oline \rho (0)\,+\,x_3\int_{0}^{1}\partial_3 \oline\rho(s\,x_3)\,ds
\end{equation*}
to write \eqref{e.bl} as
\begin{equation}\label{e.bl2}
\oline \rho(0)\,(u_{0,h}^{bl})^\perp + \left(x_3\int_{0}^{1}\partial_3 \oline\rho\,(s\,x_3)\,ds\right)\, (u_{0,h}^{bl})^\perp =\partial_\zeta^2u_{0,h}^{bl}\,.
\end{equation}
Let us now consider the equation
\begin{equation}\label{e.blspiral}
\oline \rho(0)\,(u_{0,h}^{bl})^\perp =\partial_\zeta^2u_{0,h}^{bl}\,,
\end{equation}
supplemented with the boundary condition 
\begin{equation}\label{e.bdaryu0bl}
u_{0,h}^{bl}(x_h,0,t)=-u_{0,h}(x_h,t)
\end{equation}
at $\zeta=0$, in view of \eqref{hyp:bound-cond} and  \eqref{e.bdarycondu2}.
We remark that the system of ODEs \eqref{e.blspiral}-\eqref{e.bdaryu0bl} is the same (here in general $\oline \rho(0)\neq 1$) as in the incompressible case, see e.g. Chapter 7 of \cite{C-D-G-G}
and references therein.
Its solutions are exponentially decaying and have a spiral structure. Indeed, we have the following formula:
\begin{align*}
&u_{0,h,b}^{bl}(x_h,\zeta,t)\,= \\
&-\left(
\begin{array}{c}
e^{-\zeta\sqrt{\frac{\oline \rho(0)}{2}}}\left[u_{0,1}(x_h,t)\cos\left(\zeta\sqrt{\frac{\oline \rho(0)}{2}}\right)+u_{0,2}(x_h,t)\sin\left(\zeta\sqrt{\frac{\oline \rho(0)}{2}}\right)\right]\\[2ex]
e^{-\zeta\sqrt{\frac{\oline \rho(0)}{2}}}\left[-u_{0,1}(x_h,t)\sin\left(\zeta\sqrt{\frac{\oline \rho(0)}{2}}\right)+u_{0,2}(x_h,t)\cos\left(\zeta\sqrt{\frac{\oline \rho(0)}{2}}\right)\right]
\end{array}
\right)\,.
\end{align*}
Let us move further. The vertical component in \eqref{e.ep-1} is
\begin{equation}\label{e.ep-10}
0=\lambda\,\left(\partial_\zeta\nabla_h\cdot u_{0,h}^{bl}+\partial_\zeta^2u_{1,3}^{bl}\right)\,+\,P'(\oline\rho)\,\partial_\zeta \rho_2^{bl}\,.
\end{equation}
Equation \eqref{e.ep0a}, together with the fact the $\oline\rho$ is strictly positive, yields
\begin{equation}\label{e.u0u1}
\nabla_h\cdot u_{0,h}^{bl}+\partial_\zeta u_{1,3}^{bl}=0\,.
\end{equation}
Hence $P'(\oline\rho)\,\partial_\zeta \rho_2^{bl}=0$ and, similarly to the argument used for $\rho_1^{bl}$, we get $\rho_2^{bl}\equiv0$.
The previous equality \eqref{e.u0u1} determines $u_{1,3}^{bl}$ up to a constant in $\zeta$, which we take so that $u_{1,3}^{bl}$ converges to zero when $\zeta\rightarrow\infty$:
\begin{equation*}
u_{1,3,b}^{bl}(x_h,\zeta,t)\,=\,-\,\frac{\,e^{-\zeta\sqrt{\frac{\oline \rho(0)}{2}}}}{\sqrt{2\oline\rho(0)}}\,\left(\cos\left(\zeta\sqrt{\frac{\oline \rho(0)}{2}}\right)+\sin\left(\zeta\sqrt{\frac{\oline \rho(0)}{2}}\right)\right)\,\nabla_h^\perp\cdot u_{0,h}(x_h,t)\,.
\end{equation*}
Similar computations can be done for the top boundary layers. 
Indeed, denoting by $\eta=\frac{1-x_3}{\ep}$ the fast vertical variable in the upper boundary layer, we use Taylor formula at first order \begin{equation}
\oline \rho (x_3)= \oline \rho (1) - (1-x_3)\int_{0}^{1}\partial_3 \,\oline\rho (1-s(1-x_3))ds
\end{equation} to define $u^{bl}_{0,h,t}$ as the solution to the equation
\begin{equation}\label{e.blspiraltop}
\oline \rho(1)\,(u_{0,h,t}^{bl})^\perp =\partial_\eta^2 u_{0,h,t}^{bl}\,,
\end{equation}
supplemented with the boundary condition 
\begin{equation}\label{e.bdaryu0bltop}
u_{0,h,t}^{bl}(x_h,0,t)=-u_{0,h}(x_h,t)
\end{equation}
at $\eta=0$, recall \eqref{e.bdarycondu2}. We have 
\begin{align*}
&u_{0,h,t}^{bl}(x_h,\eta,t)\,= \\
&-\left(
\begin{array}{c}
e^{-\eta\sqrt{\frac{\oline \rho(1)}{2}}}\left[u_{0,1}(x_h,t)\cos\left(\eta\sqrt{\tfrac{\oline \rho(1)}{2}}\right)+u_{0,2}(x_h,t)\sin\left(\eta\sqrt{\tfrac{\oline \rho(1)}{2}}\right)\right]\\[2ex]
e^{-\eta\sqrt{\frac{\oline \rho(1)}{2}}}\left[-u_{0,1}(x_h,t)\sin\left(\eta\sqrt{\tfrac{\oline \rho(1)}{2}}\right)+u_{0,2}(x_h,t)\cos\left(\eta\sqrt{\tfrac{\oline \rho(1)}{2}}\right)\right]
\end{array}
\right)\,.
\end{align*}
and, from \eqref{e.u0u1} with $\partial_\zeta$ replaced by $-\partial_\eta$,
\begin{equation*}
u_{1,3,t}^{bl}(x_h,\eta,t)\,=\,\,\frac{\,e^{-\eta\sqrt{\frac{\oline \rho(1)}{2}}}}{\sqrt{2\oline\rho(1)}}\,\left(\cos\left(\eta\sqrt{\tfrac{\oline \rho(1)}{2}}\right)+\sin\left(\eta\sqrt{\tfrac{\oline \rho(1)}{2}}\right)\right)\,\nabla_h^\perp\cdot u_{0,h}(x_h,t)\,.
\end{equation*}
Hence, using \eqref{e.bdarycondu2}, one can now compute the right hand side of equation \eqref{e.qgb}:
\begin{equation}
\begin{aligned}
\oline\rho(1)\,u_{1,3}(x_h,1,t)-\oline\rho(0)\,u_{1,3}(x_h,0,t)\,&= \,-\oline\rho(1)u_{1,3,t}^{bl}(x_h,0,t)+\oline\rho(0)u_{1,3,b}^{bl}(x_h,0,t)
\\& =\,-\,\tfrac{\sqrt{\oline\rho(0)}+\sqrt{\oline\rho(1)}}{\sqrt{2}}\,\omega_0\,=\,-\,\tfrac{\sqrt{\oline\rho(0)}+\sqrt{\oline\rho(1)}}{\sqrt{2}}\,\Delta_hQ\,.
\end{aligned}
\label{e.damping}
\end{equation}
This is the so-called \emph{Ekman pumping term}, which represents the secondary (global) circulation created by the boundary layer. It appears as a damping term for the quasi-geostrophic dynamics,
described by equation \eqref{e.qgb}.

\noindent\underline{Final choices for correctors.} 
It remains to choose the functions $\rho_2$, $u_1$ and $u_{1,h}^{bl}$. These terms are auxiliary terms which do not appear in the final result.\\ 
We choose the interior terms in order to make the terms of order $O(\ep)$ in the mass equation and the terms of order $O(\ep^0)$ in the momentum equation vanish
identically. Notice that \eqref{e.ep0} 
determines $u_{1,h}$ in terms of $u_0$, $\rho_1$ and $\rho_2$, and hence, through relation \eqref{e.u0h}, in terms of $\rho_1$ and $\rho_2$ only.
Specifically,
\begin{align}\label{e.equ1h}
u_{1,h}&:=\frac1{\oline\rho}\,\bigg(-\mu\Delta_hu_{0,h}^\perp+\oline\rho\partial_tu_{0,h}^\perp \\
&\qquad\qquad+\oline\rho u_{0,h}\cdot\nabla_hu_{0,h}^\perp-u_{0,h}\rho_1+\nabla_h^\perp\bigg(
P'(\oline\rho)\rho_2+\frac{P''(\oline \rho)}{2}\rho_1^2\bigg)\bigg)\,. \nonumber
\end{align}
Next, the vertical component of \eqref{e.ep0} reads
\begin{equation}\label{e.verticalep0}
\partial_3\,\left(P'(\oline\rho)\,\rho_2\,+\,\frac{P''(\oline \rho)}{2}\,\rho_1^2\right)=\,-\,\rho_2
\end{equation}where we have used that $u_{0,3}\equiv0.$ Since, by \eqref{hyp:p}, $P'(\oline\rho)>0$, $\rho_2$ can be defined as the solution of the ODE 
\begin{equation}\label{e.oderho2}
\partial_3 \,\rho_2 \,+\,\frac{\partial_3\left(P'(\oline\rho)\right)+1}{P'(\oline\rho)}\,\rho_2\,=\,-\,\frac{\partial_3\left(P''(\oline \rho)\,\rho_1^2\right)}{2\,P'(\oline\rho)}\,,\end{equation}
up to an arbitrary constant $c(x_h,t)$, that we take equal to zero for simplicity.
We remark that this choice does not affect the choice of the other quantities since $\rho_2$ appears only in $\eqref{e.ep0}$ or higher order equations.
Moreover, since $\rho_1$ and $\nabla \rho_1$ are bounded in time and space ($Q$, defined in \eqref{e.defq}, satisfies the quasi-geostrophic equation \eqref{e.qgb},
which admits regular solutions, see Lemma \ref{highenergyqg} later), $\rho_2$ and $\nabla \rho_2$ are also bounded in time and space.\\
Next, equation \eqref{e.ep1a} determines $u_{1,3}$ up to a constant in $x_3$, which we take equal to $-u_{1,3}^{bl}(x_h,0,t)$ in order to enforce the no-slip boundary condition
for the vertical component at order $O(\ep)$. Therefore, thanks to \eqref{e.d3u13} we get 
\begin{align} \label{eq:u_1-3}
\oline\rho (x_3)\,u_{1,3}(x_h,x_3,t)=\ &-\oline\rho (0)\,u_{1,3}^{bl}(x_h,0,t)-\int_0^{x_3}\left(\partial_t\rho_1+\oline\rho\,\nabla_h\cdot u_{1,h}\right)dz\,.
\end{align} 
Differently from the case without the gravitational potential, this term does not have an affine structure as in the incompressible case (see again Chapter 7 of \cite{C-D-G-G}),
since $u_{1,h}$ does not depend only on $x_h$.\\
In order to enforce the no-slip boundary condition at order $O(\ep)$ also for the horizontal component, we impose
\begin{align}
\label{e.bdarycondu1}
u_{1,h}^{bl}(x_h,0,t)=\ &-u_{1,h}(x_h,0,t)
\end{align}
at $\zeta=0$. It remains to choose the boundary layer term $u_{1,h}^{bl}$. The specifications for the boundary layer term $u_{1,h}^{bl}$ are that it is exponentially decaying to
$0$ for $\zeta\ra\infty$ and satisfies \eqref{e.bdarycondu1} at the boundary $\zeta=0$. Hence, we define $u_{1,h,b}^{bl}$ in the following way: for all $\zeta\in[0,\infty)$ and
$x_h\in\R^2$,
\begin{equation*}
u_{1,h,b}^{bl}(x_h,\zeta,t):=-u_{1,h}(x_h,0,t)e^{-\zeta\sqrt{\frac{\oline \rho(0)}{2}}}\,.
\end{equation*}
Analogously, $u_{1,h,t}^{bl}$ is defined for all $\eta\in[0,\infty)$ and
$x_h\in\R^2$ by
\begin{equation*}
u_{1,h,t}^{bl}(x_h,\eta,t):=-u_{1,h}(x_h,1,t)e^{-\eta\sqrt{\frac{\oline \rho(1)}{2}}}\,.
\end{equation*}
\begin{remark} \label{r:u1bl}
	Contrary to the interior terms, it is not possible to make the terms of order $O(\ep)$ in the mass
	equation and the terms of order $O(\ep^0)$ in the momentum equation vanish identically. Indeed that would come down to imposing  
	\begin{align}
	\label{e.noncompatible}
	\begin{split}
	\oline\rho\,\nabla_h\cdot \,u_{1,h}^{bl}=\ &-\nabla_h\rho_1\cdot u_{0,h}^{bl}\,-\,\partial_3 \oline\rho\,u_{1,3}^{bl} \\
	\lambda\partial_\zeta(\nabla_h\cdot u_{1,h}^{bl})=\ &-\partial_\zeta^2u_{1,3}^{bl}=\partial_\zeta(\nabla_h\cdot u_{0,h}^{bl})\,,
	\end{split}
	\end{align}
	which is overdetermined. This fact is due to the lack of higher-order correctors, since we truncate the expansion at order one in $\ep$. 
\end{remark}
Notice that, due to exponential decay to zero in the interior of the domain, the boundary layer terms are small. Moreover, we can exploit their decay by relying
on Hardy's inequality (see the computations in Section \ref{sec.quant}). 
The final stability estimate, though, will be worse 
than in the absence of boundary layer phenomena (as \textit{e.g.} for complete slip boundary conditions). Improving
this estimate would require to consider higher-order correctors in the ansatz \eqref{e.ansatz}.\\
Notice also that, using \eqref{e.bdarycondu2}, we have at the bottom $x_3=0$ 
\begin{equation}
\begin{aligned}
u_{0,h}(x_h,t)+u_{0,h,b}^{bl}(x_h,0,t)+u_{0,h,t}^{bl}(x_h,\tfrac{1}{\ep},t)&=u_{0,h,t}^{bl}(x_h,\tfrac{1}{\ep},t)\\
u_{1}(x_h,0,t)+u_{1,b}^{bl}(x_h,0,t)+u_{1,t}^{bl}(x_h,\tfrac{1}{\ep},t)&=u_{1,t}^{bl}(x_h,\tfrac{1}{\ep},t)\,,
\end{aligned}
\end{equation}
and at the top $x_3=1$
\begin{equation}
\begin{aligned}
u_{0,h}(x_h,t)+u_{0,h,b}^{bl}(x_h,\tfrac{1}{\ep},t)+u_{0,h,t}^{bl}(x_h,0,t)&=u_{0,h,b}^{bl}(x_h,\tfrac{1}{\ep},t)\\
u_{1}(x_h,1,t)+u_{1,b}^{bl}(x_h,\tfrac{1}{\ep},t)+u_{1,t}^{bl}(x_h,0,t)&=u_{1,b}^{bl}(x_h,\tfrac{1}{\ep},t).
\end{aligned}
\end{equation}
It means that we have a (exponentially small, but still non-zero) trace of the top boundary layer on the bottom boundary and vice-versa. 
Hence, we will add corrector terms in the ansatz \eqref{e.ansatzfinal} below, in order to keep homogeneous boundary conditions.
This is a technical point, but needed to apply Hardy's inequality later.

\noindent\underline{The ansatz.} 
To put it in a nutshell, we have obtained the following ansatz for the structure of the solutions to \eqref{e.nse}-\eqref{hyp:bound-cond}:
\begin{align}
\begin{split}
\label{e.ansatzfinal}
&\rho^\ep_{app}(x_h,x_3,t)=\,\oline\rho(x_3)\,+\,\ep\,\rho_1(x_h,x_3,t)\,+\,\ep^2 \rho_2(x_h,x_3,t)\\
&u^\ep_{app}(x_h,x_3,\zeta,\eta,t)=\\&\left(\begin{array}{c}
\nabla_h^\perp Q(x_h,t)+u_{0,h,b}^{bl}(x_h,\zeta,t)+u_{0,h,t}^{bl}(x_h,\eta,t)- u_{0,h,1/\ep}^{bl}(x_h,x_3,t)\\[1ex]
0
\end{array}\right)\\
&+\ep\left(\begin{array}{c}
u_{1,h}(x_h,x_3,t)+u_{1,h,b}^{bl}(x_h,\zeta,t)+u_{1,h,t}^{bl}(x_h,\eta,t)- u_{1,h,1/\ep}^{bl}(x_h,x_3,t)\\[1ex]
u_{1,3}(x_h,x_3,t)+u_{1,3,b}^{bl}(x_h,\zeta,t)+u_{1,3,t}^{bl}(x_h,\eta,t)-u_{1,3,1/\ep}^{bl}(x_h,x_3,t)
\end{array}\right)
\end{split}
\end{align}
with $Q$ defined in \eqref{e.defq} and
\begin{equation} \label{eq:high-correct}
\begin{aligned}
u^{bl}_{0,h,1/\ep}(x_h,x_3,t)&=x_3\,u_{0,h,b}^{bl}(x_h,\tfrac{1}{\ep},t)\,+\,(1-x_3)\,u_{0,h,t}^{bl}(x_h,\tfrac{1}{\ep},t) \\
u^{bl}_{1,1/\ep}(x_h,x_3,t)&=x_3\,u_{1,b}^{bl}(x_h,\tfrac{1}{\ep},t)\,+\,(1-x_3)\,u_{1,t}^{bl}(x_h,\tfrac{1}{\ep},t)\,.
\end{aligned}
\end{equation}
In addition,  
it follows from \eqref{e.qgb} that
\begin{equation}\label{e.qg}
\begin{aligned}
\partial_t\left(\left\langle \tfrac{\oline\rho}{P'(\oline\rho)}\right\rangle Q-\left\langle\oline\rho\right\rangle\Delta_h Q \right)-
\left\langle\oline\rho\right\rangle\nabla_h^\perp Q\cdot\nabla_h\Delta_h Q+\mu\Delta_h^2 Q-\tfrac{\sqrt{\oline\rho(0) }+ \sqrt{\oline\rho(1)}}{\sqrt{2}}\Delta_hQ =0\,.
\end{aligned}
\end{equation}
This is the \emph{quasi-geostrophic equation}. Similar limit equations without damping term have been shown in e.g. \cite{DeA-F}, \cite{F-G-N} and \cite{F-L-N}, where the boundary layers do not appear
due to the complete slip condition. Notice that in \cite{F-L-N} the parabolic term disappears, since the authors consider also the inviscid limit.
We state here the well-posedness and the regularity results for the quasi-geostrophic equation \eqref{e.qg}, whose detailed proofs are given in \cite{EB-phd}. 
\begin{proposition} \label{th:QGwellposed}
	Let $Q_{in}\in H^1(\R^2)$. Then, there exists a unique global weak solution $Q$ to the quasi-geostrophic equation \eqref{e.qg} with initial datum $Q_{in}$, such that
	\begin{equation*}
	Q\in \mc C(\R_+;H^1(\R^2)) \cap L^\infty(\R_+;H^1(\R^2))\quad\mbox{ and }\quad  \nabla_hQ\in L^2(\R_+;H^1(\R^2))\,.
	\end{equation*}
\end{proposition}
\begin{lemma}\label{highenergyqg}
	Let $n\geq 1$ be an integer and $Q_{in}\in H^n(\R^2)$. Then, there exists a constant $C_{n-1}>0$ such that any weak solution to \eqref{e.qg} with initial datum $Q_{in}$ satisfies the following inequality for all $t\geq 0$:
	\begin{equation}\begin{aligned}
	&\sum_{j=0}^{n-1}\left(\|\nabla^j_h Q(t)\|^2_{L^2}+ \|\nabla_h^{j+1}Q(t)\|^2_{L^2}\right)+\sum_{j=0}^{n-1}\left(\int_{0}^{t} \|\nabla_h^{j+1}Q\|^2_{L^2} +
	\|\nabla_h^{j+2}Q\|^2_{L^2}\right)\\
	&\qquad\qquad\qquad\qquad\qquad\qquad\quad\leq C_{n-1} \sum_{j=0}^{n-1}\left(\|\nabla^j_h Q_{in}\|_{L^2} + \|\nabla_h^{j+1}Q_{in}\|^2_{L^2}\right),
	\end{aligned}
	\label{generalhigherenergyqg}
	\end{equation}
	where $C_0=C_1=1$ and $C_{n-1}=C_{n-1}(\|Q_{in}\|_{H^{n-1}})$ for $n-1\geq 2$.
\end{lemma}

The boundary layer profiles $u_{0,h,b}^{bl}$ and $u_{0,h,t}^{bl}$  are solutions of the systems \eqref{e.blspiral} - \eqref{e.bdaryu0bl} and \eqref{e.blspiraltop} - \eqref{e.bdaryu0bltop} respectively.
We refer to the previous computations for the precise definitions of the higher-order terms.\\
We conclude this part by remarking that, according to the previous computations, we have that $(\rho^\ep_{app},u^\ep_{app})$ solves the following system:
\begin{equation}
\label{e.nseapp}
\left\{ 
\begin{aligned}
&\partial_t\rho^\ep_{app}+\nabla\cdot(\rho^\ep_{app} u^\ep_{app})=\ep R^{bl}+\ep^2R^\ep\\[1ex] 
& \rho^\ep_{app} \partial_tu^\ep_{app}+\rho^\ep_{app} u^\ep_{app}\cdot\nabla u^\ep_{app}+\frac1\ep e_3\times \rho^\ep_{app} u^\ep_{app}+\frac1{\ep^2}\nabla P(\rho^\ep_{app})=\,
\frac{1}{\ep^2}\rho_{app}^\ep \nabla G \\ 
&+\frac{x_3}{\ep}\int_{0}^{1}\partial_3 \oline\rho\,(s\,x_3)\,ds\, e_3\times u^{bl}_{0,h,b}\,-\,\frac{1-x_3}{\ep}\int_{0}^{1}\partial_3 \oline\rho\,(1-s(1-x_3))\,ds\,e_3\times u^{bl}_{0,h,t}\\
&+\,\Delta_{\mu,\ep} u^\ep_{app}+\lambda\nabla(\nabla\cdot u^\ep_{app})+ S^{bl}+\ep S^\ep\,
\end{aligned}
\right.
\end{equation}
in the slab $\Omega$ with no-slip boundary conditions \eqref{hyp:bound-cond}. The remainder terms $R^\ep$ and $S^\ep$ are of the form
\begin{equation*}
R^\ep=R^\ep(x_h,x_3,\tfrac{x_3}{\ep},\tfrac{1-x_3}{\ep},t)\qquad\mbox{ and }\qquad S^\ep=S^\ep(x_h,x_3,\tfrac{x_3}{\ep},\tfrac{1-x_3}{\ep},t)
\end{equation*}
while the boundary layer terms are
\begin{equation*}\begin{aligned}
R^{bl}(x_h,x_3,\tfrac{x_3}{\ep},\tfrac{1-x_3}{\ep},t)\,=&\,\oline\rho\,\nabla_h \cdot(u_{1,h,b}^{bl}+u_{1,h,t}^{bl})\\&+\, \nabla_h \rho_1 \cdot (u_{0,h,b}^{bl}+u_{0,h,t}^{bl})\, +
\,\partial_3\oline\rho \, (u_{1,3,b}^{bl}+u_{1,3,t}^{bl}) \\[1ex]
S^{bl}(x_h,x_3, \tfrac{x_3}{\ep},\tfrac{1-x_3}{\ep},t)=&\,\oline\rho\,\partial_t (u_{0,b}^{bl}+ u_{0,t}^{bl})\,+\, \oline\rho\, u_{0,h} \cdot \nabla_h (u_{0,b}^{bl}+u_{0,t}^{bl})\,\\
&+\, \oline\rho\, (u_{0,b}^{bl}+u_{0,t}^{bl}) \cdot \nabla_h u_0\, +\, \oline\rho \, (u_{0,b}^{bl}+u_{0,t}^{bl})  \cdot \nabla_h (u_{0,b}^{bl}+u_{0,t}^{bl}) \\
&+ \oline\rho\, (u_{1,3}+u_{1,3,b}^{bl}+u_{1,3,t}^{bl})\,\partial_\eta (u_{0,b}^{bl}+u_{0,t}^{bl})\,-\,\mu \Delta_h (u_{0,t}^{bl}+u_{0,b}^{bl})\\&- \partial_\eta^2 (u_{1,b}^{bl}+ u_{1,t}^{bl})
-\,\lambda \binom{0}{\partial_\eta \nabla_h\cdot (u_{1,h,b}^{bl}+u_{1,h,t}^{bl})} \,\\&+\,e_3\times \left(\rho_1 \,( u_{0,b}^{bl}+u_{0,t}^{bl}) \,+\, \oline\rho ( u_{1,b}^{bl}+u_{1,t}^{bl}) \right).
\end{aligned}
\end{equation*}
The remainders $\ep^2R^\ep$ and $\ep S^\ep$ contain also the terms of order $O(e^{-1/\ep})$ coming from the correctors  $u_{0,h,1/\ep}^{bl}$ and $u_{1,1/\ep}^{bl}$,
defined in \eqref{eq:high-correct}.
Notice that $S^{bl}$ appears at order $O(1)$, but has fast, exponential, decay inside $\Omega$: more precisely, we have
$\|S^{bl}\|_{L^p}\leq C\,\ep^{\frac1p}$ for all $p\in[1,\infty]$.\\
The choice of the regularity of the initial datum $Q_{in}$ guarantees enough regularity for the approximate solution $(\rho^\ep_{app}, u^\ep_{app})$ in order to derive the stability estimates later
in Subsection \ref{ss:rel-entropy}. This is stated in the following lemma, which is a straightforward consequence of Lemma \ref{highenergyqg} above.
\begin{lemma}\label{regularityuapp}
	The approximated density $\rho^\ep_{app}$ and velocity field $u^\ep_{app}$ can be written as 
\begin{align*}
 \rho^\ep_{app}(x_h, x_3,t) &=\oline\rho(x_3) + Q(t,x_h)\frac{\oline\rho}{P'(\oline\rho)}(x_3) + Q(t,x_h)l(x_3) \\
u^\ep_{app}(x_h, x_3,\zeta, \eta,t)&=\sum_{i= 1}^{N}f_i(t,x_h)g_i(x_3)h_i(\zeta)w_i(\eta)\,,
\end{align*}
for some $N\geq 1$, with $ \overline{\rho}, l,g_i\in C^1([0,1])$ and $h_i, w_i\in C^\infty(\R_+)$.
In addition, for $Q_{in}\in H^5(\R^2)$, we have $$Q\in L^\infty(\R_+;H^5(\R^2)), \quad f_i \in L^\infty(\R_+;H^{k_i}(\R^2)) \mbox{ with } k_i\geq 1.$$
\end{lemma}

\subsubsection{Large-scale quasi-geostrophic equation}

We recover here the equation for $u_0$ from \eqref{e.qg}.
For this we need the following standard lemma, which gives the Helmholtz decomposition for two-dimensional vector fields.
\begin{lemma}\label{lem.2dfields}
	Let $p\in(1,\infty)$. Let $a$ and $b$ be two scalar fields in $L^p(\R^2)$. \\	
	Then there exists a unique vector field $F$, belonging to the homogeneous Sobolev space $\dot W^{1,p}(\R^2;\R^2)$, which solves the system
	\begin{equation} \label{eq:div-curl}
	\left\{\begin{array}{l}
	\nabla_h^\perp\cdot F\,=\,a \\[1ex]
	\nabla_h\cdot F\,=\,b\,.
	\end{array}\right.
	\end{equation}
	Moreover, the following formula holds:
	\begin{equation*}
	F\,=\,-\,\nabla_h^\perp(-\Delta_h)^{-1}a\,-\,\nabla_h(-\Delta_h)^{-1}b\,.
	\end{equation*}
\end{lemma}
The previous result being classical, we do not give the proof here: we rather refer to \cite{Ch1995} (see Sections 1.2 and 1.3) and \cite{F-N} (see Section 10.6) for details.
We just give some explanations about the uniqueness, which will be needed below. By linearity, let us suppose that $F$ solves
\eqref{eq:div-curl} with $a=b=0$. In particular $\nabla\times F=0$, hence (see Corollary 1.2.1 of \cite{Ch1995}) $F\,=\,\nabla h$, for some $h\in L^p$. But from
$\nabla\cdot F=0$, we deduce that $-\Delta h=0$, which admits the only solution $h=0$ in $L^p$

Now, let $\pi\in\dot H^1(\R^2)$ be the (unique, up to additive constants) solution to
\begin{equation}\label{e.pressureq}
-\Delta_h\pi\,=\,\langle\oline\rho\rangle\,\nabla_h\cdot\left(u_{0,h}\cdot\nabla_hu_{0,h}\right)\,=\,\langle\oline\rho\rangle\,\nabla_hu_{0,h}:\nabla_hu_{0,h}\,.
\end{equation}
We then define $F(\cdot,t)\in L^2(\R^2;\R^2)$  for almost every $t\geq0$ by the formula
\begin{equation*}
F\,:=\,\langle\oline\rho\rangle\,\partial_tu_{0,h}\,+\,\langle\oline\rho\rangle\,u_{0,h}\cdot\nabla_hu_{0,h}\,-\,\mu\Delta_hu_{0,h}\,+\,
\tfrac{\sqrt{\oline\rho(0)}+\sqrt{\oline\rho(1)}}{\sqrt2}u_{0,h}\,+\,\nabla_h\pi\,.
\end{equation*}
Notice that, thanks to equations \eqref{e.pressureq} and \eqref{e.qg} and the divergence-free condition $\nabla_h\cdot u_{0,h}=0$, we have
\begin{align*}
\nabla_h^\perp\cdot F\,=\,\left\langle \tfrac{\oline\rho}{P'(\oline\rho)}\right\rangle\,\partial_tQ\qquad\mbox{ and }\qquad \nabla_h\cdot F\,=\,0\,.
\end{align*}
Therefore, the uniqueness part of Lemma \ref{lem.2dfields} implies that
\begin{equation*}
F\,=\,\nabla_h^\perp(\Delta_h)^{-1}\left\langle \tfrac{\oline\rho}{P'(\oline\rho)}\right\rangle\,\partial_tQ\,=\,
\left\langle \tfrac{\oline\rho}{P'(\oline\rho)}\right\rangle\,\d_t(\Delta_h)^{-1}u_{0,h}\,,
\end{equation*}
where we have also used \eqref{e.u0h}. 
Eventually, we find that $u_{0,h}$ solves the system
\begin{equation}
\label{e.ns2dqg}
\left\{ 
\begin{aligned}
& \partial_t\left(\langle\oline\rho\rangle-\left\langle \tfrac{\oline\rho}{P'(\oline\rho)}\right\rangle(\Delta_h)^{-1}\right)u_{0,h} \\[1ex]
&\qquad\quad+\langle\oline\rho\rangle\,u_{0,h}\cdot\nabla_hu_{0,h}-\mu\Delta_hu_{0,h}+\tfrac{\sqrt{\oline\rho(0)}+\sqrt{\oline\rho(1)}}{\sqrt2}u_{0,h}+\nabla_h\pi=0\\[1ex]
&\nabla_h\cdot u_{0,h}\,=\,0
\end{aligned}
\right.
\end{equation}
in $\R^2$. The second term appearing in the time derivative is a consequence of the combination of the effects due to density stratification and fast rotation. 
Notice that both \eqref{e.qgb} and \eqref{e.ns2dqg} are averaged (in $x_3$) versions of \eqref{e.ep0}.

\subsection{Weak solutions and uniform a priori bounds} \label{ss:unif-bounds}

We recall here some basics about \emph{finite energy weak solutions} to system \eqref{e.nse}. We refer e.g. to \cite{PLLions}, \cite{N-S} and \cite{F-N} for details.

\begin{definition} \label{d:weak-sol}
Let $\oline\rho>0$ be the solution to the logistic equation \eqref{eq:bar-rho}, and let $(\rho_{in},u_{in})$ verify
$$
\int_\Omega\left(\frac{1}{2}\rho_{in}|u_{in}|^2+\frac{1}{\veps^2}E\left(\rho_{in},\oline\rho\right)\right)dx < \infty\,.
$$
A couple $(\rho,u)$ is a \emph{finite-energy weak solution} to system \eqref{e.nse} on $[0,T]\times\Omega$, related to the initial datum $(\rho_{in},u_{in})$,
if the following conditions are satisfied:
\begin{itemize}
 \item $\rho\geq0$, with $\rho-\oline\rho\in L^\infty\bigl((0,T);(L^2+L^\g)(\Omega)\bigr)$, with $\g>1$ appearing in \eqref{hyp:p}, and $u\in L^2\bigl((0,T);H^1(\Omega;\R^3)\bigr)$;
 \item the mass equation 
is satisfied in the weak sense: namely, for any test-function  $\vphi\in\mc C^\infty_0\bigl([0,T)\times\Omega\bigr)$, one has
$$
-\int^T_0\!\!\int_\Omega\bigl(\rho\,\d_t\vphi\,+\,\rho\,u\cdot\nabla\vphi\bigr)\,dx\,dt\,=\,\int_\Omega\rho_{in}\,\vphi(0)\,dx\,;
$$
\item $P(\rho)\in L^1_{\rm loc}\bigl((0,T)\,\times\Omega\bigr)$, and the momentum equation is verified in the weak sense: for any $\psi\in\mc C^\infty_0\bigl([0,T)\times\Omega;\R^3\bigr)$, one has
\begin{align*}
& \int^T_0\!\!\int_\Omega\biggl(-\,\rho\,u\cdot\d_t\psi\,-\,\rho\,u\otimes u:\nabla\psi\,+\,\frac{1}{\veps}\,e^3\times(\rho\,u)\cdot\psi\,-\,\dfrac1{\ep^2} \,P(\rho)\,\nabla\cdot\psi \\
&+\,\nabla_{\mu,\veps}u:\nabla_{\mu,\veps}\psi\,+\,\lambda\,\nabla\cdot u\;\nabla\cdot\psi\,-\,\frac{1}{\veps^2}\, \rho\,\nabla G \cdot \psi\biggr)\,dx\,dt\,=\,\int_\Omega\rho_{in}\,u_{in}\cdot\psi(0)\,;
\end{align*}
\item the following energy inequality holds true for almost every $t\in(0,T)$:
\begin{align}
& \int_\Omega\left(\frac{1}{2}\rho(t)|u(t)|^2+\frac{1}{\veps^2}E\big(\rho(t),\oline\rho\bigr)\right)dx  \label{est:energy} \\
&\qquad\qquad+\int^t_0\int_\Omega\left(\mu|\nabla_hu|^2+\veps|\d_3u|^2+\lambda|\nabla\cdot u|^2\right)dx\,d\tau \nonumber \\
&\qquad\qquad\qquad\qquad\qquad\qquad \leq \int_\Omega\left(\frac{1}{2}\rho_{in}|u_{in}|^2+\frac{1}{\veps^2}E\left(\rho_{in},\oline\rho\right)\right)dx\,, \nonumber
\end{align}
where we have defined the \emph{relative energy functional}
\begin{equation}\label{def:rel-en}
E\big(\rho,\oline\rho\big)\,:=\,H\big(\rho\big)\,-\,H\big(\oline\rho\big)\,-\,H'\big(\oline\rho\big)\,\big(\rho-\oline\rho)\,.
\end{equation}
\end{itemize}
The solution is said to be \emph{global} if the previous conditions hold true for all $T>0$.
\end{definition}

Consider now a family of global in time finite-energy weak solutions $\left(\rho^\veps,u^\veps\right)_\veps$ to system \eqref{e.nse}. Recall that the existence of such a family is only assumed here,
but open in general, because of anisotropy of the viscous stress tensor.
We collect here some uniform bounds verified by that family. We refer e.g. to \cite{F-N}, \cite{F-G-N} and \cite{F-G-GV-N} for details.
These bounds will be important in the next subsection, when proving stability estimates.

By assumption, for any $\veps\in(0,1]$ the energy inequality
\begin{align}
&\int_\Omega\!\!\left(\rho^\veps(t)|u^\veps(t)|^2+\frac{1}{\veps^2}E\left(\rho^\veps(t),\oline \rho\right)\right) \label{est:energy_eps}\\
&+\int^t_0\!\!\!\int_\Omega\left(\mu|\nabla_hu^\veps|^2+\veps|\d_3u^\veps|^2+\lambda|\nabla\cdot u^\veps|^2\right)\leq\ 
\int_\Omega\!\!\left(\rho_{in}^\veps|u_{in}^\ep|^2+\frac{2}{\veps^2}E\left(\rho_{in}^\ep,\oline \rho\right)\right) \nonumber
\end{align}
holds for almost every $t>0$. According to inequality (4.15) of \cite{F-N-S}, we have the following control, which holds for any positive scalar functions $\rho(x,t)$ and $r(x,t)$,
with $0 < r_- \leq r(x,t)\leq r_+$, for some real numbers $r_-, r_+$: there exist constants $c_1,\ c_2\,>\,0$ such that, for almost all $(x,t)\in\Omega\times\R_+$, one has
\begin{align}\label{est:rel-entr}
\begin{split}
&c_1\,\left(|\rho(x,t)-r(x,t)|^2\,\mathbf 1_{\{|\rho-r|(\cdot,t)<1\}}\,+\,|\rho(x,t)-r(x,t)|^\gamma\,\mathbf 1_{\left\{|\rho-r|(\cdot,t)\geq 1\right\}}\right) \\
&\leq\,E\left(\rho(x,t),r(x,t)\right)\\
&\leq\, c_2\,\left(|\rho(x,t)-r(x,t)|^2\,\mathbf 1_{\left\{|\rho-r|(\cdot,t)<1\right\}}\,+\,|\rho(x,t)-r(x,t)|^\gamma\,\mathbf 1_{\left\{|\rho-r|(\cdot,t)\geq 1\right\}}\right),
\end{split}
\end{align}
where the notation $\{|\rho-r|(\cdot,t)<1\}$ stands for the set of $x\in\Omega$ such that $|\rho(x,t)-r(x,t)|<1$ (and analogously for the $\geq$ symbol) and $\mathbf 1_A$ denotes the characteristic
function of a set $A\subset\Omega$. Notice that the same inequalities hold if we  replace 1 by any constant $M>0$, up to change the value of the constants $c_1$ and $c_2$.\\
Now, following Chapters 4 and 5 of \cite{F-N}, let us introduce the \emph{essential set} and the \emph{residual set} as follows: for almost every $t>0$, we set
\begin{align} \label{def:ess-res}
&\Omega_{\mathrm{ess}}(t):=\left\{x\in\Omega\;\big|\quad |\rho^\veps(x,t)-\oline \rho (x_3)|<\sigma \right\}\quad\mbox{ and }\quad
\Omega_{\mathrm{res}}(t):=\Omega\setminus\Omega_{\mathrm{ess}}(t)\,, 
\end{align}
for some $\sigma$ (to be fixed later) such that \begin{equation*}
0<\sigma<\inf_{(0,1)} \oline \rho\,.
\end{equation*}
Accordingly, given any function $h$, we define its \emph{essential part} and \emph{residual part} as
$$
[h]_{\mathrm{ess}}\,:=\,h\,\mathbf 1_{\Omega_{\mathrm{ess}}}\qquad\mbox{ and }\qquad
[h]_{\mathrm{res}}\,:=\,h\,\mathbf 1_{\Omega_{\mathrm{res}}}\,=\,h\,-\,[h]_{\mathrm{ess}}\,.
$$
Keep in mind that such a decomposition depends on $\rho^\veps$.

After this preparation, let us establish uniform bounds for $\left(\rho^\veps,u^\veps\right)_\veps$. First of all, in view of the assumptions we will fix on the initial data
$\left(\rho_{in}^\ep,u_{in}^\ep\right)_\veps$ in the next subsection, we can assume that the right hand side of \eqref{est:energy_eps} is uniformly bounded for $\veps\in(0,1]$.
Then, using \eqref{est:rel-entr}, we deduce the existence of a constant $C>0$ such that, for all $T>0$ fixed and all $0<\veps\leq1$, one has
\begin{align}
\left\|\sqrt{\rho^\veps}\,u^\veps\right\|_{L^\infty_T(L^2)}\,&\leq\,C \label{est:rho-u_eps} \\
\frac{1}{\veps}\,\left\|\left[\rho^\veps\,-\,\oline \rho\right]_{\mathrm{ess}}\right\|_{L^\infty_T(L^2)}\,&\leq\,C \label{est:rho_ess} \\
\sup_{t\in[0,T]}\mc L\left(\Omega_{\mathrm{res}}(t)\right)\,+\,\left\|\left[\rho^\veps\right]_{\mathrm{res}}\right\|^\g_{L^\infty_T(L^\g)}\,&\leq\,\,C\,\veps^2\,, \label{est:rho_res}
\end{align}
where $\mc L(A)$ denotes the Lebesgue measure of a set $A\subset\Omega$.
We refer to Section 2 of \cite{F-G-N} and Section 4 of \cite{F-N_CPDE} for details. \\
Next, let us consider the viscosity terms: recalling that $\mu>0$ and $\lambda>0$, from \eqref{est:energy_eps} we immediately get
\begin{align}
\left\|\nabla_hu^\veps\right\|_{L^2_T(L^2)}\,+\,\left\|\nabla\cdot u^\veps\right\|_{L^2_T(L^2)}\,\leq\,C  \label{est:D_h-div} \\
\sqrt{\veps}\;\left\|\d_3u_h^\veps\right\|_{L^2_T(L^2)}\,\leq\,C\,, \label{est:d_3u^h}
\end{align}
for some universal constant $C>0$ independent of $\veps$ and of the fixed time $T>0$. In addition, owing to the identity
$$
\d_3u_3^\veps\,=\,\nabla\cdot u^\veps\,-\,\nabla_h\cdot u_h^\veps\,,
$$
we also deduce that
\begin{equation} \label{est:d_3u^3}
\left\|\d_3u_3^\veps\right\|_{L^2_T(L^2)}\,\leq\,C\,.
\end{equation}

Finally, arguing exactly as in Section 2 of \cite{F-G-N}, we deduce that there exists a constant $C>0$ such that, for all $\veps>0$ and all $T>0$,
one has
\begin{equation} \label{est:u-L^2}
\left\|u^\veps\right\|_{L^2_T(L^2)}\,\leq\,C\,.
\end{equation}

\subsection{Stability estimates} \label{ss:rel-entropy}

This section is devoted to estimating the error between weak solutions to \eqref{e.nse} and their smooth approximation built in Subsection \ref{sec.formal}.
We consider \emph{well-prepared} initial data. Specifically, the initial density $(\rho_{{in}}^\ep)_\veps$ and velocity fields $(u_{{in}}^\ep)_\veps$ 
satisfy the following requirements:
\begin{itemize}
 \item for all $\veps\in\,(0,1]$, one has
\begin{equation}\label{e.rhoin}
\rho_{{in}}^\ep\,=\,\oline\rho\,+\,\veps\,r_{{in}}^\ep\,,\qquad\qquad\mbox{ with }\qquad (r_{{in}}^\ep)_\veps\,\subset\,\left(L^2\cap L^\infty\right)(\Omega)\,;
\end{equation}
\item we have $(u_{{in}}^\ep)_\veps\,\subset\,L^2(\Omega)$;
\item there exists $Q_{in}\in H^5(\R^2)$ such that, after defining
\begin{equation}\label{e.structureu0}
r_{in}:=\frac{\bar\rho}{P'(\bar\rho)}Q_{in}\qquad\mbox{ and }\qquad u_{in}:=\left(-\partial_2Q_{in},\partial_1Q_{in},0\right)\,,
\end{equation}
we have the strong convergence properties
\begin{align}
&r_{in}^\ep\longrightarrow r_{in}\quad\mbox{ and }\quad u_{in}^\ep\longrightarrow u_{in} \quad\qquad\mbox{ in }\qquad L^2(\Omega)\,.\label{e.cvr0} 
\end{align}
\end{itemize}

\begin{remark}
Condition \eqref{e.structureu0} implies in particular that  
\begin{equation}\label{e.wellpre}
\bar\rho \left(\begin{array}{c}
u_{in,h}^\perp\\[1ex] 0
\end{array}\right)+\left(\begin{array}{c}
P'(\bar\rho)\nabla_hr_{in}\\[1ex] \partial_3(P'(\bar\rho)r_{in})
\end{array}\right)=\,r_{in} \, \nabla G\,.
\end{equation}
\end{remark}

From the uniform bounds in Section \ref{ss:unif-bounds}, it is classical to derive that, up to extraction of a suitable subsequence, weak solutions $\left(\rho^\veps,u^\veps\right)_\veps$ converge to a limit state $(\bar\rho,\bar u)$ which belongs to the kernel of the singular perturbation operator. 
We refer to e.g. \cite{F-G-N}, \cite{F-G-GV-N}, \cite{F-N_CPDE} for details. 
The goal of the present subsection is to make this convergence quantitative, to show the general structure of the solutions and to take into account the correctors due to Ekman's boundary layers. 
We aim at proving the following result. Recall that the relative entropy $E$ is defined in \eqref{def:rel-en}.

\begin{theorem}\label{prop.main}
	For $\gamma\geq3/2$, 
	suppose that there exists a finite-energy weak solution $(\rho^\ep,u^\ep)_\ep$ to \eqref{e.nse} with well-prepared initial data $(\rho_{in}^\ep,u_{in}^\ep)_\ep\in L^\infty\times L^2$
	verifying hypotheses \eqref{e.rhoin}, \eqref{e.structureu0} and \eqref{e.cvr0}.
	Let $(\rho^\ep_{app}, u^\ep_{app})_\ep$ be defined as in \eqref{e.ansatzfinal}, and define $\de u^\ep= u^\ep - u^\ep_{app}$.
	Then, there exist functions $C_1(t), C_2(t) \in L^1([0,T))$ for all $T>0$, and constants $C>0$ and  $\ep_0\in(0,1)$ such that, for all $\ep\in(0,\ep_0)$, the following estimate holds,
	for almost every $t>0$:	\begin{equation*}
	\begin{aligned}
	&\int_{\Omega}\rho^\ep(t)|\de u^\ep(t)|^2dx+\frac1{\ep^2}\int_{\Omega}E(\rho^\ep(t),\rho^\ep_{app}(t))\,dx \\
	&\qquad+\,\int_{0}^{t}\!\!\!\int_{\Omega}\left(\mu|\nabla_h\de u^\ep|^2\,+\,\ep|\partial_3\de u^\ep|^2 \,+\,\lambda|\nabla\cdot\de u^\ep|^2\right)dx\\
	&\leq\,Ce^{2\int_{0}^{t}C_1(s)ds}\left(\int_{\Omega}\rho_{in}^\ep|\de u^\ep_{in}|^2dx+\frac1{\ep^2}\int_{\Omega}E(\rho_{in}^\ep,\rho^\ep_{{in},app})dx+
	\ep\int_{0}^t C_2(\tau)\,d\tau\right).
	\end{aligned}
	\end{equation*}
\end{theorem}
\begin{remark}
	The lower bound for the exponent $\gamma$ comes from the control of the source term in the relative entropy inequality \eqref{e.entroest}: in particular, in \eqref{ineq32} we need
	$\gamma\geq 3/2$ to apply H\"{o}lder's inequality and get the estimate \eqref{est:I_7}.
\end{remark}
In order to prove the previous result, we resort to the technique of the \emph{relative entropy/relative energy} inequality, see e.g. \cite{Germ}, \cite{F-N-S},
\cite{F-J-N}, \cite{F-N_CPDE} and \cite{F-L-N}.
The relative entropy estimate of those works 
is directly applicable in our framework, but it is not
immediately clear how to take advantage of the small remainders in \eqref{e.nseapp}.
Instead, we directly derive the entropy inequality on the system for $(\de\rho^\ep,\de u^\ep)$ and take into account from the beginning 
that $(\rho^\ep_{app},u^\ep_{app})$ is almost a solution to \eqref{e.nse}. On the contrary, the relative entropy inequality of e.g. \cite{F-N_CPDE} holds for a much wider class of smooth functions.

\subsubsection{The relative entropy inequality} \label{sss:relative-entropy}
We set
$$
\de\rho^\veps\,:=\,\rho^\ep-\rho^\ep_{app}\qquad\qquad\mbox{ and }\qquad\qquad \de u^\veps\,:=\,u^\ep-u^\ep_{app}\,.
$$
From systems \eqref{e.nse} and \eqref{e.nseapp}, it is easy to find an equation for  $\de\rho^\veps$ and $\de u^\ep$: after setting $\de P^\veps\,:=\,P(\rho^\ep)-P(\rho^\ep_{app})$,
we get
\begin{align}
&\partial_t\de\rho^\ep+\nabla\cdot(u^\ep_{app}\de\rho^\ep)=-\nabla\cdot(\rho^\ep\de u^\ep)-\ep R^{bl}-\ep^2R^\ep \label{e.eqrhoepapp-} \\
&\rho^\ep\partial_t\de u^\ep+\rho^\ep u^\ep\cdot\nabla\de u^\ep+\frac1\ep e_3\times \rho^\ep\de u^\ep+\frac1{\ep^2}\nabla\de P^\ep-\Delta_{\mu,\veps}\de u^\ep-\lambda\nabla\nabla\cdot\de u^\ep \label{e.equepapp-} \\
&= \frac{1}{\ep^2}\, \de \rho^\ep \,\nabla G-\de\rho^\ep\partial_t u^\ep_{app}+(\rho^\ep_{app} u^\ep_{app}-\rho^\ep u^\ep)\cdot\nabla
u^\ep_{app} \nonumber \\
&\qquad -\frac1\ep e_3\times\de\rho^\ep u^\ep_{app}-S^{bl}-\ep S^\ep-\frac{x_3}{\ep}\int_{0}^{1}\partial_3 \oline\rho(sx_3)ds e_3\times u^{bl}_{0,h,b} \nonumber \\
&\qquad\qquad\qquad\qquad\qquad\qquad+\frac{1-x_3}{\ep}\int_{0}^{1}\partial_3 \oline\rho(1-s(1-x_3))dse_3\times u^{bl}_{0,h,t}\,. \nonumber
\end{align}

From the point of view of energy estimates, the main term to work on is the difference of the pressure terms. Testing it against $\de u^\veps$ yields
\begin{align}\label{e.diffpress}
\int_{\Omega}\nabla\de P^\ep\cdot\de u^\ep dx\,&=\,\int_{\Omega}\nabla P(\rho^\ep)\cdot u^\ep dx-\int_{\Omega}\nabla P(\rho^\ep_{app})\cdot u^\ep_{app}dx \\
&\qquad\qquad\qquad +\int_{\Omega}\nabla\cdot u^\ep_{app}\,\de P^\ep\,dx-\int_{\Omega}\nabla P(\rho^\ep_{app})\cdot\de u^\ep dx\,. \nonumber
\end{align}
By standard computations, using the mass equation in \eqref{e.nse}, we get
\begin{align*}
\int_{\Omega}\nabla P(\rho^\ep)\cdot u^\ep dx\,&=\,\int_{\Omega}\nabla \left(H'(\rho^\ep)\right)\cdot \rho^\ep u^\ep dx\,=\,\frac{d}{dt}\int_{\Omega}H(\rho^\ep)dx\,.
\end{align*}
Similarly, from the first equation in \eqref{e.nseapp} we gather
\begin{equation*}
\int_{\Omega}\nabla P(\rho^\ep_{app})\cdot u^\ep_{app}dx\,=\,\frac{d}{dt}\int_{\Omega}H(\rho^\ep_{app})dx-\ep\int_{\Omega}H'(\rho^\ep_{app})(R^{bl}+\ep R^\ep) dx\,.
\end{equation*}
In identity \eqref{e.diffpress}, we now add and substract the term $\frac{d}{dt}\int H'(\rho^\ep_{app})\de\rho^\ep dx$, in order to make the relative entropy $E\left(\rho^\veps(t),\rho^\veps_{app}(t)\right)$ appear.
Then, from \eqref{e.diffpress} and the previous computations we infer
\begin{align*}
\int_{\Omega}\nabla\de P^\ep\cdot\de u^\ep dx\,&=\,\frac{d}{dt}\int_\Omega E\left(\rho^\ep,\rho^\ep_{app}\right)dx+\int_{\Omega}\nabla\cdot u^\ep_{app}\,\de P^\ep dx-
\int_{\Omega}\nabla P(\rho^\ep_{app})\cdot\de u^\ep dx\\
&\qquad +\frac{d}{dt}\int_{\Omega}H'(\rho^\ep_{app})\,\de\rho^\ep dx+\ep\int_{\Omega}H'(\rho^\ep_{app})\left(R^{bl}+\ep R^\ep \right)dx\,.
\end{align*}
Using again the mass equations in \eqref{e.nse} and \eqref{e.nseapp}, we get
\begin{align*}
\frac{d}{dt}\int_{\Omega}H'(\rho^\ep_{app})\,\de\rho^\ep dx\,&=\,\int_{\Omega}\partial_tH'(\rho^\ep_{app})\,\de\rho^\ep dx+\int_{\Omega}H'(\rho^\ep_{app})\d_t\de\rho^\ep dx\\
&=\,\int_{\Omega}\partial_tH'(\rho^\ep_{app})\,\de\rho^\ep dx+\int_{\Omega}\nabla H'(\rho^\ep_{app})\cdot(\rho^\ep u^\ep-\rho^\ep_{app}u^\ep_{app})dx\\
&\qquad\qquad\qquad\qquad\qquad\qquad -\ep\int_{\Omega}H'(\rho^\ep_{app})\left(R^{bl}+\ep R^\ep\right) dx\,.
\end{align*}
This relation yields
\begin{align}
\int_{\Omega}\nabla\de P^\ep\cdot\de u^\ep\,dx\,&=\,\frac{d}{dt}\int_\Omega E\left(\rho^\ep,\rho^\ep_{app}\right)dx 
-\int_{\Omega}\nabla P(\rho^\ep_{app})\cdot\de u^\ep\,dx+ I\,, \label{eq:press_rel-entr} 
\end{align}
where we have defined
$$
I:=\int_{\Omega}\nabla\cdot u^\ep_{app}\de P^\ep dx +\int_{\Omega}\partial_tH'(\rho^\ep_{app})\de\rho^\ep dx +\int_{\Omega}\nabla H'(\rho^\ep_{app})\cdot(\rho^\ep u^\ep-\rho^\ep_{app}u^\ep_{app})dx.
$$
Let us work on this term for a while. We use the following Taylor expansion,
\begin{align}\label{eq:def-P}
P\left(\rho^\ep,\rho^\ep_{app}\right)\,&:=\,P(\rho^\ep)-P(\rho^\veps_{app})-P'(\rho^\ep_{app})\,\de\rho^\ep\\
&=\,\frac{1}{2}\left(\de\rho^\ep\right)^2\int_0^1(1-s)\,P''\left(\rho^\ep_{app}+s\de\rho^\ep\right)ds\,, \nonumber
\end{align}
and the fact that $H''(z)=P'(z)/z$ according to \eqref{def:H}, 
to obtain the next series of equalities:
\begin{equation*}
\begin{aligned}
I\,&=\,\int_{\Omega}\nabla\cdot u^\ep_{app}P'(\rho^\ep_{app})\de\rho^\ep dx+\int_{\Omega}\nabla\cdot u^\ep_{app}P\left(\rho^\ep,\rho^\ep_{app}\right)dx \\
&\quad+\int_{\Omega}H''(\rho^\ep_{app})\partial_t\rho^\ep_{app}\,\de\rho^\ep dx+\int_{\Omega}H''(\rho^\ep_{app})\nabla\rho^\ep_{app}\cdot(\rho^\ep u^\ep-\rho^\ep_{app}u^\ep_{app})dx\\
&=\,\int_{\Omega}\nabla\cdot u^\ep_{app}P'(\rho^\ep_{app})\,\de\rho^\ep dx-\int_{\Omega}H''(\rho^\ep_{app})\rho^\ep_{app}\nabla\cdot u^\ep_{app}\,\de\rho^\ep dx\\
&\quad+\int_{\Omega}H''(\rho^\ep_{app})\left(\partial_t\rho^\ep_{app}+\nabla\cdot(\rho^\ep_{app}u^\ep_{app})\right)\,\de\rho^\ep dx\\
&\quad+\int_{\Omega}H''(\rho^\ep_{app})\nabla\rho^\ep_{app}\cdot\de u^\ep\,\rho^\ep dx+\int_{\Omega}\nabla\cdot u^\ep_{app}P\left(\rho^\ep,\rho^\ep_{app}\right)dx\\
&=\,\int_{\Omega}H''(\rho^\ep_{app})\nabla\rho^\ep_{app}\cdot\de u^\ep\,\rho^\ep dx+\ep\int_{\Omega}H''(\rho^\ep_{app})\de\rho^\ep\left(R^{bl}+\ep R^\ep\right) dx\\
&\quad+\int_{\Omega}\nabla\cdot u^\ep_{app}\,P\left(\rho^\ep,\rho^\ep_{app}\right)dx\,.
\end{aligned}\end{equation*}
The last two terms in the above identity are small (in a sense to be made precise later). So, let us focus on the first term in the right hand side: we have
\begin{align*}
\int_{\Omega}H''(\rho^\ep_{app})\nabla\rho^\ep_{app}\cdot\de u^\ep\rho^\ep dx-\int_{\Omega}\nabla P(\rho^\ep_{app})\cdot\de u^\ep dx\,=\,
\int_{\Omega}H''(\rho^\ep_{app})\nabla\rho^\ep_{app}\cdot\de u^\ep\de\rho^\ep dx\,.
\end{align*}
Inserting this expression into the last equality for $I$, from \eqref{eq:press_rel-entr} we finally find
\begin{align}\label{e.diffpressfinal}
\begin{split}
&\int_{\Omega}\nabla\de P^\ep\cdot\de u^\ep dx\,=\,\frac{d}{dt}\int_\Omega E\left(\rho^\ep,\rho^\ep_{app}\right)dx+\int_{\Omega}H''(\rho^\ep_{app})\nabla\rho^\ep_{app}\cdot\de u^\ep\,\de\rho^\ep dx\\
&\qquad\qquad +\ep\int_{\Omega}H''(\rho^\ep_{app})\de\rho^\ep\left(R^{bl}+\ep R^\ep \right) dx+\int_{\Omega}\nabla\cdot u^\ep_{app}\,P\left(\rho^\ep,\rho^\ep_{app}\right)dx\,.
\end{split}
\end{align}
At this point, we can perform energy estimates directly on equations \eqref{e.eqrhoepapp-}-\eqref{e.equepapp-}. Using \eqref{e.diffpressfinal} above, we obtain
\begin{align}\label{e.entroest}
\begin{split}
&\frac d{dt}\int_\Omega\left(\frac12\,\rho^\ep|\de u^\ep|^2+\frac1{\ep^2}E\left(\rho^\ep,\rho^\ep_{app}\right)\right)dx\\
&\quad+\mu\int_{\Omega}|\nabla_h\de u^\ep|^2 dx+\ep\int_{\Omega}|\partial_3\de u^\ep|^2 dx+\lambda\int_{\Omega}|\nabla\cdot\de u^\ep|^2 dx\\
&\leq \, \frac1{\ep^2}\int_{\Omega} \de \rho^\ep \,\nabla G \cdot \de u^\ep\,dx\,-\frac1{\ep^2}\int_{\Omega}H''(\rho^\ep_{app})\nabla\rho^\ep_{app}\cdot \de u^\ep\,\de\rho^\ep dx \, \\
&\quad-\frac1{\ep}\int_{\Omega}H''(\rho^\ep_{app})\de\rho^\ep\left(R^{bl}+\ep R^\ep \right) dx-\frac1{\ep^2}\int_{\Omega}\nabla\cdot u^\ep_{app}\,P\left(\rho^\ep,\rho^\ep_{app}\right)dx \\
&\quad-\frac1\ep\int_{\Omega}e_3\times\de\rho^\ep u^\ep_{app}\cdot\de u^\ep dx-\int_{\Omega}\de\rho^\ep\partial_tu^\ep_{app}\cdot\de u^\ep dx\\
&\quad+\int_{\Omega}(\rho^\ep_{app} u^\ep_{app}-\rho^\ep u^\ep)\cdot\nabla u^\ep_{app}\cdot\de u^\ep dx-\int_{\Omega}S^{bl} \cdot \de u^\ep \,dx-\ep\int_{\Omega}S^\ep\cdot\de u^\ep dx\,\\
&\quad -\,\frac{1}{\ep}\int_{\Omega}x_3\int_{0}^{1}\partial_3 \oline\rho\,(s\,x_3)\,ds\, (u^{bl}_{0,h,b})^\perp\cdot \delta u^\ep_h\,dx\\
&\quad +\,\frac{1}{\ep}\,\int_{\Omega}\,(1-x_3)\int_{0}^{1}\partial_3 \oline\rho\,(1-s(1-x_3))\,ds\,(u^{bl}_{0,h,t})^\perp\cdot \delta u^\ep_h\,dx\,=\,\sum_{j=1}^{11}\,I_j\,.
\end{split}
\end{align}

\begin{remark}
In order to rigorously justify the relative entropy inequality \eqref{e.entroest}, where the equality holds if the solutions are regular enough,
one may either proceed as in \cite{F-J-N}, or use a regularization argument (see for instance \cite{F-N-S} and \cite{Germ}).
\end{remark}

Our next goal is to bound each term appearing in the sum $\sum_{j=1}^{11}\,I_j$ in the right hand side of \eqref{e.entroest}.
Before doing that, let us remark that, since $\rho_1,\rho_2\,\in\,L^\infty\left(\Omega\times\R_+\right)$,  up to restrict our attention to all $\veps\leq\veps_0$, with $\ep_0$ depending on $\|\rho_1\|_{L^\infty_{t,x}}$ and $\|\rho_2\|_{L^\infty_{t,x}}$, we can assume that 
$-\frac{\sigma}{2}\leq \ep \rho_1 +\ep^2 \rho_2\leq \frac{\sigma}{2}$ with $\sigma>0$ as in \eqref{def:ess-res}. Consequently, we can suppose that 
$$
0<\rho^-_{app} \leq \rho^\ep_{app} (x,t) \leq \rho^+_{app}\qquad\qquad\mbox{ for all }\qquad \veps>0\,,
$$
with  $\rho^-_{app}= \inf_{(0,1)} \oline \rho -\sigma$ and $\rho^+_{app}= \sup_{(0,1)} \oline \rho +\sigma$.
Then, in view of \eqref{est:rel-entr}, we have the following control:
\begin{align}
\label{e.entrocontrol}
E\left(\rho^\veps,\rho^\veps_{app}\right)(x,t)\,\geq\,c\,\left(|\de\rho^\veps(x,t)|^2\,\mathbf 1_{\left\{|\de\rho^\veps|(\cdot,t)<1\right\}}\,+\,
|\de\rho^\veps(x,t)|^\gamma\,\mathbf 1_{\left\{|\de\rho^\veps|(\cdot,t)\geq 1\right\}}\right)\,.
\end{align}
Resorting to the definitions \eqref{def:ess-res}, from \eqref{e.entrocontrol} 
we derive the following lower bound. 
\begin{lemma}
There exist $\sigma>0$ small enough (depending on $\inf\limits_{(0,1)} \oline \rho$) and 
a positive constant $c>0$, independent of $\veps\in\,]0,\veps_0]$, such that, for almost all $(x,t)\in\Omega\times\R_+$, the following bound holds:
	\begin{align}
	E\left(\rho^\veps(x,t),\rho^\veps_{app}(x,t)\right)\,\geq\,c\,\left(\left[\delta\rho^\veps\right]^2_{\mathrm{ess}}(x,t)\,+\,\mathbf{1}_{\Omega_{\mathrm{res}}(t)}(x)\right)\,.
	\label{est:E_low-b}
	\end{align}
	\label{lemmaestE}
\end{lemma}
\begin{proof}
	We divide the proof of the inequality into two steps. First we show that 
	$$E\left(\rho^\veps(x,t),\rho^\veps_{app}(x,t)\right)\,\geq\,c\,|\de\rho^\veps(x,t)|^2\,\mathbf 1_{\left\{|\de\rho^\veps(x,t)|<1\right\}}$$ implies the lower bound
	\begin{equation}
	E\left(\rho^\veps(x,t),\rho^\veps_{app}(x,t)\right)\,\geq\,c\,\left[\delta\rho^\veps\right]^2_{\mathrm{ess}}(x,t).\label{boundess}
	\end{equation}
	For this, we just need to show that $\Omega_{\mathrm{ess}}(t)\subseteq \{|\, |\de\rho^\ep(x,t)|<1 \}$. Let $x\in \Omega_{\mathrm{ess}}(t)$, then 
	\begin{equation*}
	-\frac{3}{2}\sigma \leq -\sigma-\ep\rho_1(x,t)-\ep^2 \rho_2(x,t)<\de \rho^\ep(x,t)< -\ep\rho_1(x,t) -\ep^2 \rho_2(x,t)+\sigma \leq \frac{3}{2}\sigma,
	\end{equation*}where we have used that $-\frac{\sigma}{2}\leq\ep \rho_1(x,t) +\ep^2 \rho_2(x,t)\leq \frac{\sigma}{2}$. By choosing $\sigma$ such that
	$\sigma < \min \left(2/3\,,\,\inf\limits_{(0,1)} \oline \rho\right)$,
we deduce that $|\delta \rho^\ep(x,t)|<1$. Thus, \eqref{boundess} is proved. \\ Afterwards, we prove that for, $x\in \Omega_{\mathrm{res}}(t)$, one has
	\begin{equation}E\left(\rho^\veps(x,t),\rho^\veps_{app}(x,t)\right)\geq c\,,
	\label{boundres}\end{equation}
	where $c$ is a positive constant independent of $\veps$, $t$ and $x$.
	By the definition of $\Omega_{\mathrm{res}}(t)$, either $\rho^\ep(x,t)\leq \oline \rho(x_3) -\sigma $ or $\rho^\ep(x,t)\geq \oline \rho(x_3) +\sigma$. Hence, since $E(\cdot, \rho^\ep_{app}(x,t))$
	is strictly decreasing before $\rho^\ep_{app}(x,t)$ and strictly increasing after $\rho^\ep_{app}(x,t)$, we get 
	\begin{equation*}
	E\left(\rho^\veps(x,t),\rho^\veps_{app}(x,t)\right)\geq 	E\left(\oline\rho(x_3)-\sigma,\rho^\veps_{app}(x,t)\right)
	\end{equation*}
	if $\rho^\ep(x,t)\leq \oline \rho(x_3) -\sigma $, and
	\begin{equation*}
	E\left(\rho^\veps(x,t),\rho^\veps_{app}(x,t)\right)\geq 	E\left(\oline\rho(x_3)+\sigma,\rho^\veps_{app}(x,t)\right)
	\end{equation*} 
	if $\rho^\ep(x,t)\geq \oline \rho(x_3) +\sigma$.
	Now, by Taylor's formula, up to taking a smaller $\sigma$ (which amounts to choosing a smaller $\veps_0$), we have
	\begin{align*}
	E\left(\oline\rho(x_3)-\sigma,\rho^\veps_{app}(x,t)\right)&\geq \frac{H^{''}(\rho^\ep_{app}(x,t))}{4}(-\sigma-\ep \rho_1(x,t)-\ep^2\rho_2(x,t))^2\\&\geq \frac{H^{''}(\rho^\ep_{app}(x,t))\sigma^2}{16} \\
	E\left(\oline\rho(x_3)+\sigma,\rho^\veps_{app}(x,t)\right)&\geq \frac{H^{''}(\rho^\ep_{app}(x,t))}{4}(\sigma -\ep \rho_1(x,t)-\ep^2\rho_2(x,t))^2\\&\geq \frac{H^{''}(\rho^\ep_{app}(x,t))\sigma^2}{16}.
	\end{align*}
	Then, using the uniform boundedness in time and space of $\rho^\ep_{app}$ and hypothesis \eqref{hyp:p}, we get \eqref{boundres}. The lemma is proved.
\end{proof}
Notice that $\left[|\delta\rho^\veps|\right]_{\mathrm{ess}}$ is uniformly bounded.
Next, we claim that there exists a constant $C>0$ such that, for all $T>0$ fixed, one has
\begin{equation} \label{est:d-rho_res}
\left\|\left[\de\rho^\veps\right]_{\mathrm{res}}\right\|_{L^\infty_T(L^p)}\,\leq\,C\,\veps^{2/p}\qquad\qquad\forall\;p\in[1,\g]\,.
\end{equation}
Indeed, 
by H\"older's inequality, the $L^\infty$ control on $\rho^\veps_{app}$ and \eqref{est:rho_res}, we deduce
\begin{align*}
\int_\Omega\left|\left[\de\rho^\veps\right]_{\mathrm{res}}\right|\,&\leq\,\int_\Omega\left[\rho^\veps\right]_{\mathrm{res}}\,+\,\int_\Omega\left[\rho^\veps_{app}\right]_{\mathrm{res}} \\
&\leq\,\left(\int_\Omega\left(\rho^\veps\right)^\g\,\mathbf 1_{\Omega_{\mathrm{res}}}\right)^{\!1/\g}\,\left(\mc L(\Omega_{\mathrm{res}})\right)^{\!1/\g'}\,+\,C\,\mc L(\Omega_{\mathrm{res}})\,\leq\,C\,\veps^2\,,
\end{align*}
which yields \eqref{est:d-rho_res} for $p=1$. As for the $L^\g$ norm, we write
\begin{equation} \label{eq:Omega_res}
\Omega_{\mathrm{res}}(t)\,=\,\left\{0<\rho^\veps(x,t)\leq\oline \rho(x_3)-\sigma\right\}\cup\left\{\rho^\veps(x,t)\geq \oline \rho(x_3) + \sigma\right\}\,.
\end{equation}
For the first set, we just apply \eqref{est:rho_res} again,
since $\rho^\veps$ is bounded therein. For the second set, we use the fact that, for $a\geq\de$ and $b\geq0$, with $b\leq b^*$, one has
$|a-b|^\g\,\leq\,(a+b)^\g\,\leq\,C_{\de,b^*}\,\left(a^\g+1\right)$. The case $1<p<\g$ follows from interpolation.

\subsubsection{Anisotropic Sobolev embedding}

We introduce an anisotropic version of the standard Sobolev embedding $\dot{H}^1 \hookrightarrow L^6$.
This estimate enables us to handle the anisotropy of the viscosity in the stability estimates below, see in particular the treatment of $I_7$.

\begin{lemma}[\textbf{anisotropic Sobolev embedding}]\label{lem.ani}
	Let $\Omega=\R^2\times(0,1)$.
	There exists a universal constant $C>0$ such that, for all $\kappa>0$ and all $u\in H^1_0(\Omega)$, one has
	\begin{equation}\label{e.estkappa}
	\|u\|_{L^6(\Omega)}\leq C\left(\kappa^{-\frac12}\|\nabla_hu\|_{L^2(\Omega)}+\kappa\|\partial_3u\|_{L^2(\Omega)}\right).
	\end{equation}
\end{lemma}
\begin{proof}[Proof of Lemma \ref{lem.ani}]
	Let $\kappa>0$ and $u\in H^1_0(\Omega)$. We first extend $u$ by zero on $\R^3\setminus\Omega$ and still denote the extended function by $u$.
	Now $u\in H^1(\R^3)$. We then consider the rescaled function 
	\begin{equation*}
	u_\kappa(y_h,y_3)=u\left(\frac{y_h}{\kappa^{\frac12}},\kappa y_3\right),\quad (y_h,y_3)\in\R^3.
	\end{equation*}
	By Sobolev's inequality \cite[estimate (II.3.7)]{Galdi} for the whole space, there exists a universal constant $C\in(0,\infty)$ such that 
	\begin{equation*}
	\|u_\kappa\|_{L^6(\R^3)}\leq C\|\nabla u_\kappa\|_{L^2(\R^3)}.
	\end{equation*}
	Estimate \eqref{e.estkappa} then follows by a change of variables and the fact that $u$ is zero outside the strip $\R^2\times (0,1)$,
	\begin{align*}
	\|u_\kappa\|_{L^6(\R^3)}=\ &\|u\|_{L^6(\Omega)},\\
	\|\nabla u_\kappa\|_{L^2(\R^3)}=\ &\kappa^{-\frac12}\|\nabla_hu\|_{L^2(\Omega)}+\kappa\|\partial_3u\|_{L^2(\Omega)}.
	\end{align*}
	This concludes the proof. 
\end{proof}

\subsubsection{Conclusion of the stability estimates} \label{sss:rel-entr_est}
Below we estimate every source term $I_j$ appearing in \eqref{e.entroest}, for $0<\ep\leq \ep_0$ ($\ep_0$ is given by Lemma \ref{lemmaestE}). 
For the terms $I_1, I_2, I_3, I_5$ and $I_6$ we need to treat separately the cases $\gamma\geq2$ and $3/2\leq\gamma<2$, since we use different estimates, whereas
the terms $I_4, I_8, I_9, I_{10}$ and $I_{11}$ can be controlled in the same way for any $\gamma\geq3/2$. The term $I_7$ is more intricate; it is written as a sum of five terms: for some of them,
we need to distinguish again the case $\gamma\geq2$ and $3/2\leq \gamma<2$.\\ 
The easiest terms to handle are $I_3,\ I_4,\ I_6$ and $I_9$. The terms $I_1$ and $I_2$ are combined with the Coriolis term $I_5$;
For the remaining part of $I_5$, we rely on Hardy's inequality, which is also useful to deal with $I_7$, $I_8$, $I_{10}$ and $I_{11}$. The basic idea, borrowed from \cite{B-D-GV}, is that,
whenever there is a boundary layer term $G^{bl}(\tfrac{x_3}\ep)$, 
we gain one additional $\ep$ by using the decay of $G^{bl}$ in $\zeta$.
The price to pay is a $\d_3 $ derivative on $\de u^\ep$, which however can be swallowed by the third term in the left hand side of \eqref{e.entroest}.\\
For every term, we decompose $u^\ep_{app}$ according to \eqref{e.ansatzfinal}. The terms which require more care are those of order $O(1)$, which involve in general $u_{0,h}$ and
$u_{0,h}^{bl}$, except for $I_7$ where the product $u^\ep_{app}\cdot\nabla u^\ep_{app}$ also involves $u_{1,3}$ and $u_{1,3}^{bl}$ at order $O(1)$.
For the terms which are not of order $O(1)$ the analysis can be always reduced to the case $I_3, I_4, I_6, I_9$, and the same estimates are used.

For every term involving a boundary layer, one has to equally consider the top and bottom boundary
layers; again, for simplicity, we focus on the boundary layer at the bottom only. In the computations below, $U$ and $U^{bl}$
generically denote remainder terms in the expansion for $u^\ep_{app}$
or its derivatives. The definition of these remainder terms may change from the estimate of one $I_i$ to another $I_j$.

First, we deal with the terms for which estimates hold for any $\gamma$.

\noindent\underline{Term $I_4$.} We start by considering $I_4$, when restricted to the essential set. Using \eqref{eq:def-P}, 
the assumptions on the pressure function and the fact that $\left[|\delta\rho^\veps|\right]_{\mathrm{ess}}$ is uniformly bounded, we can estimate
\begin{align*}
\left|\frac{1}{\veps^2}\int_\Omega\nabla\cdot u^\veps_{app}\,\left[P(\rho^\veps,\rho^\veps_{app})\right]_{\mathrm{ess}}\right|\,&\leq\,\frac{1}{\veps^2}\,\left\|\nabla\cdot u^\veps_{app}\right\|_{L^\infty}\,
\left\|[\de\rho^\veps]_{\mathrm{ess}}\right\|_{L^2}^2 \\
&\leq\,C\,\veps\,\frac{1}{\veps^2}\,\int_\Omega E\left(\rho^\veps,\rho^\veps_{app}\right)\,,
\end{align*}
where we have used also that $\nabla\cdot u^\veps_{app}\,=\,\veps\,\left(\nabla\cdot u_1+\nabla_h\cdot u^{bl}_{1,h}\right)$.\\
Let us consider the integral over the residual set. By \eqref{eq:def-P} again, we have
$[P]_{\mathrm{res}}\,=\,\left[P(\rho^\veps)-P(\rho^\veps_{app})\right]_{\mathrm{res}}-P'(\rho^\veps_{app})[\de\rho^\veps]_{\mathrm{res}}$.
The second term can be easily controlled, in view of the uniform boundedness of $\rho^\veps_{app}$ and the $L^1$ estimate in \eqref{est:d-rho_res}. For the first term,
we use decomposition \eqref{eq:Omega_res}: when $\rho^\veps$ is bounded, the same argument as above applies. On the set $\{\rho^\veps\geq\oline \rho +\sigma\}$, instead, we use
hypothesis \eqref{hyp:p}, the uniform
boundedness of $\rho^\veps_{app}$ and the controls in \eqref{est:rho_res} to get 
\begin{align*}
&\frac{1}{\veps^2}\int_\Omega|\nabla\cdot u^\veps_{app}|\,\left|P(\rho^\veps)-P(\rho^\veps_{app})\right|\,\mathbf 1_{\{\rho^\veps\geq\oline \rho +\sigma\}}\,\leq\,
\frac{C}{\veps}\int_\Omega\left|P(\rho^\veps)-P(\rho^\veps_{app})\right|\,\mathbf{1}_{\{\rho^\veps\geq\oline \rho +\sigma\}} \\
&\qquad\qquad\qquad\qquad\qquad\qquad \leq\,\frac{C}{\veps}\,\left(\left\|\left[\rho^\veps\right]_{\mathrm{res}}\right\|^\g_{L^\g}\,+\,\mc L\left(\Omega_{\mathrm{res}}\right)\right)\;\leq\,C\,\veps\,.
\end{align*}
Putting everything together, we finally infer that
\begin{equation} \label{est:I_4}
\left|I_4\right|\,\leq\,C\,\veps\,+\,C\,\veps\,\frac{1}{\veps^2}\,\int_\Omega E\left(\rho^\veps,\rho^\veps_{app}\right)\,,
\end{equation}
where the last term will be handled by Gr\"onwall's lemma. 

\noindent\underline{Term $I_9$.} The control of $I_9$ is direct, as no $\rho^\veps$ or $\de\rho^\veps$ enter into play. 
We get
\begin{equation} \label{est:I_9}
\left|I_9\right|\,\leq\,\veps\,\left\|S^\ep\right\|_{L^2}\,\left\|\de u^\ep\right\|_{L^2}\,\leq\,C\,\ep\,K_2(t)\,,
\end{equation}where the function $K_2(t)=\|u^\ep(t)\|_{L^2} + \|u^\ep_{app}(t)\|_{L^2}$ belongs to $L^2([0,T))$ for all $T>0.$

\noindent\underline{Terms $I_8$, $I_{10},$ and $I_{11}$.} 
We deal with $I_8$ using Hardy's inequality. 
This gives
\begin{align}
|I_8|\,&=\,\ep\left|\int_{\Omega}\frac{x_3}\ep\, S^{bl}\left(\frac{x_3}{\ep}\right)\,\cdot\frac{\de u^\ep}{x_3}\right| \label{est:I_8} \\ 
&\leq\,C_\de\,\ep\,\,\left\|\zeta\, S^{bl}\right\|_{L^2}^2\,+\,\de\,\ep \,\left\|\partial_3\de u^\ep\right\|_{L^2}^2\, \leq \,C_\de\,\ep^2\,+\,\de\ep\,\left\|\partial_3\de u^\ep\right\|_{L^2}^2\,, \nonumber
\end{align}
for some small $\de>0$, to be chosen later.
The same holds for $I_{10}$ and $I_{11}$: since $\d_3\oline\rho$ is uniformly bounded, we have
\begin{align}
|I_{10}|=&	\,\ep\left|\int_\Omega \frac{x^2_3}{\ep^2}\left(\int_{0}^{1} \,\partial_3\oline\rho(s\,x_3)\,ds\right)(u^{bl}_{0,h,b})^\perp\cdot \frac{\de u_h^\ep}{x_3}\right| \label{est:I_10} \\
\leq &\, C_\de\,\ep \,\left\|\zeta^2\, u^{bl}_{0,h,b}\right\|^2_{L^2}\,+\,\de\ep\,\left\|\partial_3\de u^\ep_h\right\|^2_{L^2} \,\leq \, C_\de\,\ep^2\,+\,\de\ep\left\|\partial_3\de u^\ep_h\right\|^2_{L^2}
\nonumber \\
|I_{11}|=&	\,\ep\left|\int_\Omega \frac{(1-x_3)^2}{\ep^2}\left(\int_{0}^{1} \,\partial_3\oline\rho(1-s(1-x_3))\,ds\right)(u^{bl}_{0,h,t})^\perp\cdot \frac{\de u_h^\ep}{1-x_3}\right| \label{est:I_11}\\
\leq &\,C_\de\ep \left\|\eta^2\, u^{bl}_{0,h,b}\right\|^2_{L^2}\,+\,\de\ep\,\left\|\partial_3\de u^\ep_h\right\|^2_{L^2}
\leq \,C_\de\,\ep^2 \,+\,\de\ep\left\|\partial_3\de u^\ep_h\right\|^2_{L^2}\,. \nonumber
\end{align}
In the estimates above, we have used the fact that the terms
$\left\|\zeta\, S^{bl}\right\|_{L^2}^2,\,\left\|\zeta^2\, u^{bl}_{0,h,b}\right\|^2_{L^2}$ and $\left\|\eta^2\, u^{bl}_{0,h,t}\right\|^2_{L^2}$ are $O(\ep^2).$

We consider now the terms whose bounds must be treated differently if $\gamma\geq 2$ or $3/2\leq\gamma<2$.

\noindent\underline{Term $I_3$.} First of all, observe that $\left\|R^{bl}(x_3/\ep)\right\|^2_{L^2_{x}}=O(\ep)$. 
Thus, we can estimate
\begin{align*}
\frac{1}{\veps}\left|\int_\Omega H''(\rho^\veps_{app})\,\left[\de\rho^\veps\right]_{\mathrm{ess}}\,\left(R^{bl}+\veps\,R^\veps\right)\right|\,&\leq\,C\,\frac{1}{\veps}\,
\left\|R^{bl}+\veps\,R^\veps\right\|_{L^2}\,\left\|\left[\de\rho^\veps\right]_{\mathrm{ess}}\right\|_{L^2} \\
&\leq\,C\,\veps\,+\,\frac{C}{\veps^2}\,\int_\Omega E\left(\rho^\veps,\rho^\veps_{app}\right)\,,
\end{align*}
where we have used also \eqref{est:E_low-b}. As for the residual part, in view of \eqref{est:d-rho_res}, we can argue in exactly the same way if $\g\geq2$. If $3/2\leq\g<2$, instead, we put the
$L^\infty$ norm on the remainder terms and use the $L^1$ bound of \eqref{est:d-rho_res} to get
$$
\frac{1}{\veps}\left|\int_\Omega H''(\rho^\veps_{app})\,\left[\de\rho^\veps\right]_{\mathrm{res}}\,\left(R^{bl}+\veps\,R^\veps\right)\right|\,\leq\,C\,\veps\,.
$$
In any case, in the end we arrive at the bound
\begin{equation} \label{est:I_3}
\left|I_3\right|\,\leq\,C\,\veps\,+\,\frac{C}{\veps^2}\,\int_\Omega E\left(\rho^\veps,\rho^\veps_{app}\right)\,.
\end{equation}

\noindent\underline{Term $I_6$.} Once again, we use the decomposition of $\de\rho^\veps$ into essential and residual parts. For the term involving the essential part, 
thanks to Young's inequality and to the controls \eqref{est:u-L^2}
and \eqref{est:E_low-b}, one has
\begin{align*}
\left|\int_\Omega\left[\de\rho^\veps\right]_{\mathrm{ess}}\,\d_tu^\veps_{app}\cdot\de u^\veps\right|\,&\leq\,\left\|\left[\de\rho^\veps\right]_{\mathrm{ess}}\right\|_{L^2}\,
\left\|\d_tu^\veps_{app}\right\|_{L^\infty}\,\left\|\de u^\veps\right\|_{L^2} \\
&\leq\,C\,\veps^2\,K_1(t)\,+\,\frac{1}{\veps^2}\,E\left(\rho^\veps,\rho^\veps_{app}\right)\,,
\end{align*}
where the function $K_1\,=\,\|u^\veps\|_{L^2}^2\,+\,\|u^\veps_{app}\|_{L^2}^2$ belongs to $L^1\left([0,T)\right)$ for all $T>0$.\\
Next, let us consider the term involving the residual part: when $\g\geq2$, we can argue exactly as above,
in view of \eqref{est:d-rho_res}. If instead 
$3/2\leq\g<2$, we start by writing
$$
\int_\Omega\left[\de\rho^\veps\right]_{\mathrm{res}}\,\d_tu^\veps_{app}\cdot\de u^\veps\,=\,
\int_\Omega\left[\rho^\veps\right]_{\mathrm{res}}\,\d_tu^\veps_{app}\cdot\de u^\veps\,-\,\int_\Omega\left[\rho^\veps_{app}\right]_{\mathrm{res}}\,\d_tu^\veps_{app}\cdot\de u^\veps\,.
$$
For the second term, we use the uniform boundedness of $\rho^\veps_{app}$ and estimate \eqref{est:rho_res} to gather, for some function $K_2\in L^2\left([0,T)\right)$ for all $T>0$, the inequality
$$
\left|\int_\Omega\left[\rho^\veps_{app}\right]_{\mathrm{res}}\,\d_tu^\veps_{app}\cdot\de u^\veps\right|\,\leq\,C\,\left\|\de u^\veps\right\|_{L^2}\,\left(\mc L(\Omega_{\mathrm{res}})\right)^{1/2}\,\leq\,
C\,\veps\,K_2(t)\,.
$$
For the term involving $\left[\rho^\ep\right]_{\rm res}$, we use decomposition \eqref{eq:Omega_res} for the residual set. The integral over the first set can be treated exactly as just done for
$\rho^\veps_{app}$ (because $\rho^\veps$ is uniformly bounded therein). Concerning the integral over the second set, we have
\begin{align*}
\left|\int_\Omega\rho^\veps\,\mathbf 1_{\{\rho^\veps\geq\oline \rho +\sigma\}}\,\d_tu^\veps_{app}\cdot\de u^\veps\right|\,&\leq\,C\,\left\|\sqrt{\rho^\veps}\,\de u^\veps\right\|_{L^2}\,
\left(\int_\Omega\rho^\veps\,\mathbf1_{\{\rho^\veps\geq\oline \rho +\sigma}\}\right)^{\!1/2} \\
&\leq\,C\,\left\|\sqrt{\rho^\veps}\,\de u^\veps\right\|^2_{L^2}\,+\,C\,\veps^2\,,
\end{align*}
since the last integral in the first line can be bounded by the integral over the residual set, for which we can use \eqref{est:rho_res}.\\
Let us introduce the following notation: we set $\delta_{2^-}(\g)\,=\,1$ if $3/2\leq\g<2$, $\delta_{2^-}(\g)\,=\,0$ otherwise. 
In the end, from the previous computations we get
\begin{align}
\left|I_6\right|\,&\leq\,C\,\veps\,\left(\veps\,K_1(t)\,+\,\de_{2^-}(\g)\,K_2(t)+\de_{2^-}(\g)\,\veps\right) \label{est:I_6} \\
&\qquad\qquad\qquad +\,C\left(\frac{1}{\veps^2}\,\int_\Omega E\left(\rho^\veps,\rho^\veps_{app}\right)+
\de_{2^-}(\g)\,\left\|\sqrt{\rho^\veps}\,\de u^\veps\right\|^2_{L^2}\right)\,. \nonumber
\end{align}

\noindent\underline{Terms $I_1$, $I_2$ and $I_5$.}
Terms $I_1$, $I_2$ and $I_5$ have to be combined together, enabling to see a cancellation at the highest order in $\ep$. Such a cancellation is already a key point in \cite{F-N_CPDE}.
After setting $U\,:=\,u^\ep_{app}-(u_{0,h}+u_{0,h}^{bl},0)$, we can write
\begin{align*}
I_1 + I_2+I_5\,=&\,\frac1{\ep^2}\int_{\Omega} \de \rho^\ep \,\nabla G \cdot \de u^\ep\,-\,\frac1{\ep^2}\int_{\Omega}H''(\rho^\ep_{app})\,\partial_3\oline\rho\, \de u^\ep_3\,\de\rho^\ep \\
&-\,\frac1{\ep}\int_{\Omega}H''(\rho^\ep_{app})\,\nabla\rho_{1}\cdot \de u^\ep\,\de\rho^\ep\,-\,
\frac1\ep\int_{\Omega}\de\rho^\ep\,\frac{P'(\oline\rho)}{\oline\rho}\,(\nabla_h^\perp\rho_1)^\perp\cdot\de u^\ep_h \\
&\,-\,\frac1\ep\int_{\Omega}\de\rho^\ep\,(u_{0,h}^{bl})^\perp\cdot\de u^\ep_h\,-\,\int_{\Omega}\de\rho^\ep\,e_3\times U(x_h,x_3,\tfrac{x_3}{\ep},\tfrac{1-x_3}{\ep},t)\cdot\de u^\ep\,.
\end{align*}
Notice that $H''(\oline\rho)=\frac{P'(\oline\rho)}{\oline\rho}$
and $(\nabla_h^\perp\rho_1)^\perp=-\nabla_h\rho_1$. Moreover, from \eqref{eq:bar-rho} we get
\begin{equation*}
\oline\rho\, \nabla G = P'(\oline\rho)\,\nabla \oline\rho\,.
\end{equation*}
Therefore, we find
\begin{align*}
I_1 + I_2+I_5=&-\frac1{\ep^2}\int_{\Omega}\left(H''(\rho^\ep_{app}) - H''(\oline\rho)\right)\partial_3\oline\rho\de u^\ep_3\de\rho^\ep-
\frac1{\ep}\int_{\Omega} H''(\oline\rho)\partial_3\rho_1\de u^\ep_3\de\rho^\ep \\
&-\frac1{\ep}\int_{\Omega}\left(H''(\rho^\ep_{app})-H''(\oline\rho)\right)\nabla\rho_{1}\cdot\de u^\ep\de\rho^\ep-\frac1\ep\int_{\Omega}\de\rho^\ep(u_{0,h}^{bl})^\perp\cdot\de u^\ep_h\\
&-\int_{\Omega}\de\rho^\ep U_h^\perp(x_h,x_3,\tfrac{x_3}{\ep},\tfrac{1-x_3}{\ep},t)\cdot\de u^\ep_h\,=\,J_1+J_2+J_3+J_4+J_5\,.
\end{align*}
Using a Taylor expansion for $h(z)=H''(z)$
with integral remainder, we can write
\begin{align*}
J_1\,+\,J_2=&-\frac1{\ep}\int_{\Omega}\left(h'(\oline\rho)\,\rho_1\,\partial_3\oline\rho\,+h(\oline\rho)\,\partial_3\rho_1\right)\de u^\ep_3\,\de\rho^\ep \, \\
&-\,\int_{\Omega}\rho_1^2\left(\int_{0}^{1}(1-s)h''(\oline\rho +s\ep\rho_1)ds\right)\partial_3\oline\rho \,\de u^\ep_3\,\de\rho^\ep \,\\=
&-\,\int_{\Omega}\rho_1^2\left(\int_{0}^{1}(1-s)h''(\oline\rho +s\ep\rho_1)ds\right)\partial_3\oline\rho\,\de u^\ep_3\,\de\rho^\ep\,,
\end{align*}
where we have used \eqref{e.qindepx3} in the last equality. Since $\rho_1$, $\oline\rho$ and $\partial_3 \oline\rho$ are $L^\infty_{t,x}$, in view of \eqref{hyp:p}
the control of $J_1+J_2$ becomes similar to the one exhibited for $I_6$.
In the same way, after noticing that $\nabla_h\rho_1$ and $\veps^{-1}\left(H''(\rho^\ep_{app})-H''(\oline\rho)\right)$ are uniformly bounded in time and space, 
the control of $J_3$ is obtained. Then, $J_1 + J_2$ and $J_3$ verify estimate \eqref{est:I_6}.
The same can be said about $J_5$, because also $U_h$ belongs to $L^\infty_{t,x}$.\\
Therefore, it remains to deal with $J_4$, for which we rely on Hardy's inequality. More precisely, let us start, as usual, by dealing with the essential part: we have
\begin{align*}
\left|\frac1\ep\int_{\Omega}\left[\de\rho^\ep\right]_{\mathrm{ess}}\,(u_{0,h}^{bl})^\perp\cdot\de u^\ep_h\right|\,&=\,
\left|\int_{\Omega}\left[\de\rho^\ep\right]_{\mathrm{ess}}\,\frac{x_3}{\veps}\,(u_{0,h}^{bl})^\perp\cdot\frac{\de u^\ep_h}{x_3}\right| \\
&\leq\,\left\|\left[\de\rho^\veps\right]_{\mathrm{ess}}\right\|_{L^2}\,\left\|\z\,u^{bl}_{0,h}(t,x_h,\z)\right\|_{L^\infty_{t,x,\z}}\,\left\|\d_3 \de u^\veps_h\right\|_{L^2} \\
&\leq\,\frac{C}{\veps^2}\,\int_\Omega E\left(\rho^\veps,\rho^\veps_{app}\right)\,+\,C\,\veps^2\,\left\|\z\,u^{bl}_{0,h}\right\|^2_{L^\infty_{t,x,\z}}\,\left\|\d_3 \de u^\veps_h\right\|^2_{L^2}\,.
\end{align*}
Notice that, for $\veps$ small enough, the second term can be swallowed by the third term in the left hand side of \eqref{e.entroest}.
As for the control of the residual part, suppose that $\g\geq2$ for a while: in this case, we can argue in the exact same way and obtain, in view of \eqref{est:d-rho_res}, that
\begin{align*}
\left|\frac1\ep\int_{\Omega}\left[\de\rho^\ep\right]_{\mathrm{res}}\,(u_{0,h}^{bl})^\perp\cdot\de u^\ep_h\right|\,&\leq\,
\left\|\left[\de\rho^\veps\right]_{\mathrm{res}}\right\|_{L^2}\,\left\|\z\,u^{bl}_{0,h}\right\|_{L^\infty_{t,x,\z}}\,\left\|\d_3 \de u^\veps_h\right\|_{L^2} \\
&\leq\,\frac{C}{\veps^2}\,\int_\Omega E\left(\rho^\veps,\rho^\veps_{app}\right)\,+\,C\,\veps^2\,\left\|\z\,u^{bl}_{0,h}\right\|^2_{L^\infty_{t,x,\z}}\,\left\|\d_3 \de u^\veps_h\right\|^2_{L^2}\,.
\end{align*}
The case $3/2\leq\g<2$ is slightly more involved. 
The control over $\{0<\rho^\ep\leq \oline \rho - \sigma \}$ does not present any special difficulty, since we have uniform bounds for $\rho_\veps$
(and obviously for $\rho^\veps_{\mathrm app}$) on that set: then,
we can argue as for controlling the essential part. Hence, let us focus on $\{\rho^\ep \geq \oline \rho +\sigma\}$. First of all, using that $\sqrt{a+b}\leq \sqrt{a}+\sqrt{b}$, we notice that
\begin{equation}
\begin{aligned}
& \left|\frac1\ep\int_{\Omega}\left[\de\rho^\ep\right]_{\mathrm{res}}\,(u_{0,h}^{bl})^\perp\cdot\de u^\ep_h\right|\,\leq\,
\frac1\ep\int\sqrt{\de\rho^\ep}\, \mathbf{1}_{\{\rho^\ep \geq\oline \rho +\sigma\}}\;\left|u_{0,h}^{bl}\right|\;\sqrt{\rho_\veps}\,\left|\de u^\ep_h\right|\\
&\qquad\qquad\qquad\qquad\qquad\qquad\quad+\,
\frac1\ep\int\sqrt{\de\rho^\ep}\, \mathbf{1}_{\{\rho^\ep \geq \oline \rho +\sigma\}}\;\left|u_{0,h}^{bl}\right|\;\sqrt{\rho^\veps_{app}}\,\left|\de u^\ep_h\right|\,. 
\end{aligned}
\label{est:J_2-res}
\end{equation}
For the first term in the right hand side of \eqref{est:J_2-res}, we proceed in the following way:
\begin{align*}
\int_{\Omega}\sqrt{\de\rho^\ep}\, \mathbf{1}_{\{\rho^\ep \geq \oline \rho +\sigma\}}\;&\left|u_{0,h}^{bl}\right|\;\sqrt{\rho_\veps}\,\left|\de u^\ep_h\right|\\&\leq\,
\left\|\left[\de\rho^\ep\right]_{\mathrm{res}}\right\|_{L^\g}^{1/2}\,\left\|u_{0,h}^{bl}\right\|_{L^\infty}\,\left\|\sqrt{\rho_\veps}\,\de u^\ep_h\right\|_{L^2}\,\left(\mc L(\Omega_{\mathrm{res}})\right)^{1/q}\,,
\end{align*}
where $1/(2\g)+1/2+1/q=1$. Using \eqref{e.entrocontrol} and \eqref{est:E_low-b}, we deduce that
\begin{align*}
& \frac{1}{\veps}\int_{\Omega}\sqrt{\de\rho^\ep}\, \mathbf{1}_{\{\rho^\ep \geq \oline \rho +\sigma\}}\;\left|u_{0,h}^{bl}\right|\;\sqrt{\rho_\veps}\,\left|\de u^\ep_h\right| \\
&\qquad\qquad  \leq\,\veps^{1/\g+2/q-1}\,\left(\frac{1}{\veps^2}\,E\left(\rho^\veps,\rho^\veps_{app}\right)\right)^{1/(2\g)+1/q}\,\left\|u_{0,h}^{bl}\right\|_{L^\infty}\,
\left\|\sqrt{\rho_\veps}\,\de u^\ep_h\right\|_{L^2} \\
&\qquad\qquad =\,\left(\frac{1}{\veps^2}\,E\left(\rho^\veps,\rho^\veps_{app}\right)\right)^{1/2}\,\left\|u_{0,h}^{bl}\right\|_{L^\infty}\,
\left\|\sqrt{\rho_\veps}\,\de u^\ep_h\right\|_{L^2}\,.
\end{align*}
After applying Young's inequality, this term can be controlled by Gr\"onwall's lemma in the final estimate.
For the last term in \eqref{est:J_2-res}, we argue in the following way:
\begin{align*}
&\int\sqrt{\de\rho^\ep}\mathbf{1}_{\{\rho^\ep \geq \oline \rho +\sigma\}}\left|u_{0,h}^{bl}\right|\sqrt{\rho^\veps_{app}}\left|\de u^\ep_h\right|\,=\,\veps
\int\sqrt{\de\rho^\ep}\mathbf{1}_{\{\rho^\ep \geq \oline \rho +\sigma\}}\left|\frac{x_3}{\ep}\,u_{0,h}^{bl}\right|\sqrt{\rho^\veps_{app}}\,\left|\frac{\de u^\ep_h}{x_3}\right| \\
&\leq\veps\int\left(\sqrt{\rho^\ep}\mathbf{1}_{\{\rho^\ep \geq \oline \rho +\sigma\}}\left|\frac{x_3}{\ep}u_{0,h}^{bl}\right|\sqrt{\rho^\veps_{app}}\left|\frac{\de u^\ep_h}{x_3}\right|+
\rho^\veps_{app} \mathbf{1}_{\{\rho^\ep \geq \oline \rho +\sigma\}}\left|\frac{x_3}{\ep}u_{0,h}^{bl}\right|\left|\frac{1}{x_3}\de u^\ep_h\right|\right) \\
&\leq\,\veps\,\left\|\d_3\de u^\ep_h\right\|_{L^2}\,\left\|\z\,u_{0,h}^{bl}\right\|_{L^\infty}\,\times \\
&\qquad\left(\left\|\rho^\ep_{app}\right\|_{L^\infty}^{1/2}\,\left\|\left[\rho^\ep\right]_{\mathrm{res}}\right\|_{L^\g}^{1/2}\,\left(\mc L(\Omega_{\mathrm{res}})\right)^{1/q}\,+\,
\left\|\rho^\ep_{app}\right\|_{L^\infty}\left(\mc L(\Omega_{\mathrm{res}})\right)^{1/2}\right)\,,
\end{align*}
where $q$ is defined as above. Notice that, in view of \eqref{est:rho_res}, we have
$\left\|\left[\rho^\ep\right]_{\mathrm{res}}\right\|_{L^\g}\,=\,O\left(\veps^{2/\g}\right)$ and $\mc L(\Omega_{\mathrm{res}})\,=\,O\left(\veps^2\right)$. Therefore, we finally find
\begin{align*}
\frac{1}{\veps}\int\sqrt{\de\rho^\ep}\, \mathbf{1}_{\{\rho^\ep \geq\oline \rho +\sigma\}}\;\left|u_{0,h}^{bl}\right|\;\sqrt{\rho^\veps_{app}}\,\left|\de u^\ep_h\right|\,\leq\,C_\de\,\veps\,+\,
\de\,\veps\,\left\|\d_3\de u^\ep_h\right\|_{L^2}^2\,.
\end{align*}
In the end, we deduce the following control:
\begin{equation}\label{est:I_1+I_2+I_5}
\begin{aligned}
&\left|I_1+I_2+I_5\right|\,\leq\,\frac{C}{\veps^2}\,\int_\Omega E\left(\rho^\veps,\rho^\veps_{app}\right)\,+\,C\delta_{2^-}(\gamma)\|\sqrt{\rho^\ep}\de u^\ep_h\|^2_{L^2}\\
&\qquad+
C\ep(\ep K_1(t) + \delta_{2^-}(\gamma)K_2(t) + \delta_{2^-}(\gamma))+(C\ep^2+\delta_{2^-}(\gamma)\de\veps\,)\left\|\d_3\de u^\veps_h\right\|^2_{L^2}\,,
\end{aligned}
\end{equation}
where the last term in the right hand side can be absorbed into the left hand side of the relative entropy inequality \eqref{e.entroest}.

Finally, let us deal with $I_7$.

\noindent\underline{Term $I_7$.} 
We start by considering the following decomposition:
\begin{align*}
&I_7=-\int_{\Omega}\de\rho^\ep u^\ep_{app}\cdot\nabla u^\ep_{app}\cdot\de u^\ep-\frac1\ep\int_{\Omega}\rho^\ep\de u^\ep_3\d_\zeta u_{0,h}^{bl}(\tfrac{x_3}{\veps})\cdot\de u^\ep_h \\
&-\int_{\Omega}\rho^\ep\de u^\ep_3\d_\zeta u_1^{bl}(\tfrac{x_3}{\veps})\cdot\de u^\ep-\int_{\Omega}\rho^\ep\de u^\ep_h\cdot\nabla_h (u_{0,h} + u^{bl}_{0,h})\cdot\de u^\ep_h-
\ep\int_{\Omega}\rho^\ep\de u^\ep\cdot U\cdot\de u^\ep \\
&\quad=\,J_6\,+\,J_7\,+\,J_8\,+\,J_9\,+\,J_{10}\,,
\end{align*}
where $\ep U=  \ep \left(\nabla u_1+\binom{\nabla_h}{0} u_1^{bl} \right)$ is the remainder term in the expansion for $\nabla u^\ep_{app}$.  \\
The first term $J_6$ can be handled as done with $I_6$. Indeed, one has
\begin{equation*}
u^\ep_{app}\cdot\nabla u^\ep_{app}\,=\,\left(u_{0,h}+u_{0,h}^{bl}\right)\cdot\nabla_h (u_{0,h} + u^{bl}_{0,h})\,+\,\left(u_{1,3}+u_{1,3}^{bl}\right)\,\d_\zeta u_{0,h}^{bl}(\tfrac{x_3}\ep)\,+\,h.o.t.\,,
\end{equation*}
where $h.o.t.$ represents higher order terms in $\veps$. Then $u^\ep_{app}\cdot\nabla u^\ep_{app}$
is uniformly bounded in $L^\infty_{t,x}$. Therefore, $J_6$ verifies an inequality similar to \eqref{est:I_6} above.\\
The terms $J_8$, $J_9$ and $J_{10}$ can be simply bounded as follows:
\begin{align*}
\left|J_8\right|\,&\leq\,C\,\left\|\d_\z u^{bl}_{1}\right\|_{L^\infty_{t,x}}\,\left\|\sqrt{\rho^\ep}\,\de u^\ep\right\|_{L^2}^2, \\
\left|J_9\right|\,&\leq\,C\,\left\|\nabla_h (u_{0,h} + u^{bl}_{0,h})\right\|_{L^\infty_{t,x}}\,\left\|\sqrt{\rho^\ep}\,\de u^\ep\right\|_{L^2}^2, \\
\left|J_{10}\right|\,&\leq\,C\,\ep\,\|U\|_{L^\infty_{t,x}}\,\left\|\sqrt{\rho^\ep}\,\de u^\ep\right\|_{L^2}^2\,.
\end{align*}
We remark that these estimates holds for $\gamma\geq 3/2$.\\
We now focus on the remaining term $J_7$, which is the most difficult one to deal with. The difficulties come from the need to gain smallness in $\ep$
(by using Hardy's inequality as above), from the low integrability of the residual part and from the fact that this term is quadratic in $\delta u^\ep$. We first decompose 
\begin{equation} \label{eq:rho-dec}
\rho^\ep=\left[\de\rho^\ep\right]_{\mathrm{ess}}+\left[\de\rho^\ep\right]_{\mathrm{res}} + \rho^\ep_{app}.
\end{equation}
The essential part is easy to bound: owing to the boundedness of $\rho^\veps$ on that set and to an application of Hardy's inequality, 
we get
\begin{align*}
\frac{1}{\veps}\left|\int_{\Omega}\left[\de\rho^\ep\right]_{\mathrm{ess}}\,\de u^\ep_3\,\d_\zeta u_{0,h}^{bl}(\tfrac{x_3}{\veps})\cdot\de u^\ep_h\right|\,&\leq\,C\,\veps\,
\left\|\zeta^2\,\d_\z u_{0,h}^{bl}\right\|_{L^\infty_{t,x}}\,\left\|\frac{\de u^\ep_3}{x_3}\right\|_{L^2}\,\left\|\frac{\de u^\ep_h}{x_3}\right\|_{L^2} \\
&\leq\,C\,\veps\,\left\|\d_3\de u^\ep_3\right\|_{L^2}\,\left\|\partial_3\de u^\ep_h\right\|_{L^2} \\
&\leq\,C\,\ep^{3/2}\,\left\|\partial_3\de u^\ep_h\right\|_{L^2}^2\,+\,C\,\veps^{1/2}\,\left\|\partial_3\de u^\ep_3\right\|_{L^2}^2\,.
\end{align*}
The control of the part involving $\rho^\ep_{app}$ is similar, so let us turn to the residual part. Two different estimates are computed if $\gamma$ is larger or smaller than the critical exponent $2$.
For $\gamma\geq 2$, we write, for $\alpha \in(0,1)$ to be chosen later on,
\begin{equation*}
\delta u_3^\ep=(\delta u_3^\ep)^{1-\alpha}\frac{(\delta u_3^\ep)^{\alpha}}{x_3^\alpha}x_3^\alpha
\end{equation*} 
and then apply Sobolev's and Hardy's inequalities: this yields
\begin{align}\label{e.hardyu3}
\|(\delta u_3^\ep)^{1-\alpha}\|_{L^\frac{6}{1-\alpha}}\leq C\|\nabla\delta u_3^\ep\|_{L^2}^{1-\alpha}\quad\mbox{and}\quad\left\|\frac{(\delta u_3^\ep)^{\alpha}}{x_3^\alpha}\right\|_{L^\frac2{\alpha}}\leq C\|\partial_3\delta u_3^\ep\|_{L^2}^\alpha.
\end{align}
We use the same technique for $\de u^\ep_h$ with $ \beta\in (0,1)$. Then, choosing $\alpha, \beta$ such that
$$\alpha +\beta = \dfrac{1}{2}\,,$$
we have, for all $\de>0$ to be chosen later,
\begin{align*}
&\frac{1}{\veps}\left|\int_{\Omega}\left[\de\rho^\ep\right]_{\mathrm{res}}\,\de u^\ep_3\,\d_\zeta u_{0,h}^{bl}(\tfrac{x_3}{\veps})\cdot\de u^\ep_h\right|\\
=\ &\ep^{\alpha+\beta-1}\left|\int_{\Omega}\left[\de\rho^\ep\right]_{\mathrm{res}}\,(\de u^\ep_3)^{1-\alpha}\frac{(\de u^\ep_3)^{\alpha}}{x_3^\alpha}\,\frac{x_3^{\alpha+\beta}}{\ep^{\alpha+\beta}}\d_\zeta u_{0,h}^{bl}(\tfrac{x_3}{\veps})\cdot(\de u^\ep_h)^{1-\beta}\frac{(\de u^\ep_h)^{\beta}}{x_3^\beta}\,\right|\\
\leq\ &\ep^{-1/2}\|\left[\de\rho^\ep\right]_{\mathrm{res}}\|_{L^2}\|\nabla\delta u_3^\ep\|_{L^2}^{1-\alpha}\|\partial_3\delta u_3^\ep\|_{L^2}^\alpha\,\|\nabla\delta u_h^\ep\|_{L^2}^{1-\beta}\|\partial_3\delta u_h^\ep\|_{L^2}^\beta\,\|\z^{\alpha+\beta}\d_\zeta u_{0,h}^{bl}\|_{L^\infty}\\
\leq\ &C\ep^{1/2}\left(\frac{1}{\ep^2}\int_{\Omega}E(\rho^\ep,\rho^\ep_{app})\right)^{1/2}\|\nabla\delta u_3^\ep\|_{L^2}\|\nabla\delta u_h^\ep\|_{L^2}\\
\leq\ & \frac{C_\de}{\ep^2}\,K_1(t)\int_{\Omega}E(\rho^\ep,\rho^\ep_{app})\,+\,\de\,\ep\,\|\nabla_h\delta u_h^\ep\|^2_{L^2}\,+\,\de\,\ep\,\|\partial_3\delta u_h^\ep\|^2_{L^2}\,,
\end{align*}
where $K_1(t)=\|\nabla u^\ep_3(t)\|^2_{L^2}+\|\nabla u^\ep_{app,3}(t)\|^2_{L^2} $ belongs to $L^1\left([0,T)\right)$ for all $T>0$. In the second inequality we have used the lower bound
$$E(\rho^\ep(x,t),\rho^\ep_{app}(x,t))\geq c \ |\de\rho^\ep(x,t)|^2\,,$$
which comes from \eqref{e.entrocontrol} when $\gamma\geq 2$. \\ 
For $3/2\leq\gamma<2$, we use the same argument as in the case $\gamma\geq2$ for $\delta u_3^\ep$.
The control of $\delta u_h^\ep$, instead, is done via the anisotropic Sobolev embedding given of Lemma \ref{lem.ani}.
Hence, for $\alpha\in[0,1]$ such that 
$2-\frac{3}{\gamma}=\alpha,$ 
H\"older's inequality gives, using \eqref{e.hardyu3} for $\de u_3^\ep$ and \eqref{e.estkappa} for $\de u_h^\ep$ with $\kappa=\ep^{(-\frac12+\frac1\gamma)_+}$,
\begin{equation}
\begin{aligned}
&\frac{1}{\veps}\left|\int_{\Omega}\left[\de\rho^\ep\right]_{\mathrm{res}}\,\de u^\ep_3\,\d_\zeta u_{0,h}^{bl}(\tfrac{x_3}{\veps})\cdot\de u^\ep_h\right|\\
&=\ \ep^{\alpha-1}\left|\int_{\Omega}\left[\de\rho^\ep\right]_{\mathrm{res}}\,(\de u^\ep_3)^{1-\alpha}\frac{(\de u^\ep_3)^{\alpha}}{x_3^\alpha}\,\frac{x_3^{\alpha}}{\ep^{\alpha}}\d_\zeta u_{0,h}^{bl}(\tfrac{x_3}{\veps})\cdot\de u^\ep_h\right|\\
&\leq\ C\ep^{\alpha-1}\|\left[\de\rho^\ep\right]_{\mathrm{res}}\|_{L^\gamma}\|\nabla\delta u_3^\ep\|_{L^2}^{1-\alpha}\|\partial_3\delta u_3^\ep\|_{L^2}^\alpha\\
&\qquad\times\left(\kappa^{-\frac12}\|\nabla_h\de u_h^\ep\|_{L^2}+\kappa\|\partial_3\de u_h^\ep\|_{L^2}\right)\|\z^\alpha\d_\zeta u_{0,h}^{bl}\|_{L^\infty}\\
&\leq\ C\ep^{1-\frac1\gamma}\|\nabla\delta u_3^\ep\|_{L^2}\left(\kappa^{-\frac12}\|\nabla_h\de u_h^\ep\|_{L^2}+\kappa\|\partial_3 \de u_h^\ep\|_{L^2}\right)\\
&\leq\tfrac{\min(\mu,\lambda)}{10}\|\nabla\delta u_3^\ep\|_{L^2}^2+C(\mu,\lambda)\ep^{2-\frac2\gamma}\left(\kappa^{-1}\|\nabla_h\de u_h^\ep\|_{L^2}^2+\kappa^2\|\partial_3\de u_h^\ep\|_{L^2}^2\right)\\
&\leq\tfrac{\min(\mu,\lambda)}{10}\|\nabla\delta u_3^\ep\|_{L^2}^2+C(\mu,\lambda)\ep^{(\frac{5}{2}-\frac3\gamma)_-}\|\nabla_h\de u_h^\ep\|_{L^2}^2+C(\mu,\lambda)\ep^{1_+}\|\partial_3\de u_h^\ep\|_{L^2}^2.
\end{aligned}
\label{ineq32}
\end{equation}
Hence we can swallow the whole right hand side on condition that $\gamma>6/5$ (which is the case, since $\gamma\geq 3/2$) and $\ep$ is sufficiently small.  \\
To put it in a nutshell, we obtain the following bound on $J_7$:
\begin{equation*}
\begin{aligned}
|J_7|\,\leq\,&\frac{C_\delta}{\veps^2}K_1(t)\int_\Omega E\left(\rho^\veps,\rho^\veps_{app}\right) +\, 
(\delta \, \ep + C\ep^\frac32 + \delta_{2^-}(\gamma)C(\mu,\lambda)\ep^{1_+} )\,\left\|\partial_3\de u^\ep_h\right\|_{L^2}^2\,\\&
+\,\left(\delta\, \ep + \delta_{2^-}(\gamma) C(\mu,\lambda)\ep^{(\frac{5}{2}-\frac3\gamma)_-}\right)\|\nabla_h\de u_h^\ep\|_{L^2}^2\\
& +\,\left(\delta_{2^-}(\gamma)\tfrac{\min(\mu,\lambda)}{10} + C \ep^\frac12\right)\|\nabla\delta u_3^\ep\|_{L^2}^2\,.
\end{aligned}
\end{equation*}
Therefore, we finally get the following estimate for $I_7$:
\begin{equation}
\begin{aligned}
|I_7|\,\leq\,&C\veps\Big(\veps\,K_1(t)+\de_{2^-}(\g)K_2(t)+\de_{2^-}(\g)\veps\Big) +\frac{C +C_\delta K_1(t)}{\veps^2}\int_\Omega E\left(\rho^\veps,\rho^\veps_{app}\right) \\&
+\left(\de_{2^-}(\g)C + C_1+C_2\ep\right)\left\|\sqrt{\rho^\ep}\de u^\ep\right\|_{L^2}^2 \\
& + 
(\delta \ep + C\ep^\frac32 + \delta_{2^-}(\gamma)C(\mu,\lambda)\ep^{1_+} )\left\|\partial_3\de u^\ep_h\right\|_{L^2}^2\\&
+\left(\delta\, \ep + \delta_{2^-}(\gamma) C(\mu,\lambda)\ep^{(\frac{5}{2}-\frac3\gamma)_-}\right)\|\nabla_h\de u_h^\ep\|_{L^2}^2\\
& +\left(\delta_{2^-}(\gamma)\tfrac{\min(\mu,\lambda)}{10} + C \ep^\frac12\right)\|\nabla\delta u_3^\ep\|_{L^2}^2\,. 
\end{aligned}
\label{est:I_7}
\end{equation}

\begin{remark}
	The anisotropic Sobolev embedding in Lemma \ref{lem.ani} can be used to provide better estimates only for $\gamma$ small. For instance, in
	\eqref{est:J_2-res}, using Lemma \ref{lem.ani} we get a remainder term of order $\ep^\alpha$, with $0<\alpha<1$ for $ 3/2\leq \gamma<2 $ and $\alpha>1$ only for $\gamma< 12/11$,
	while by using the smallness of the Lebesgue measure of $\Omega_{\mathrm{res}}$ we get a remainder term of order $\ep$ for $ 3/2\leq\gamma<2 $.	\end{remark}

In the end, summing up our estimates, we get from \eqref{e.entroest} the following differential inequality: there exist functions  $C_1(t),\ C_2(t) \in L^1([0,T))$, and
constants $ C_3>0$ and $\ep_0\in(0,1)$, such that, for all $\ep\in(0,\ep_0)$, all $t\in(0,T)$ and all $\de>0$, one has
\begin{equation}\label{e.entroestfinal}
\begin{aligned}
&\frac d{dt}\left(\frac12\int_{\Omega}\rho^\ep|\de u^\ep|^2dx+\frac1{\ep^2}\int_{\Omega}E(\rho^\ep,\rho^\ep_{app})\,dx\right)\\&\qquad+\mu\int_{\Omega}|\nabla_h\de u^\ep|^2 dx
+\ep\int_{\Omega}|\partial_3\de u^\ep|^2 dx+\lambda\int_{\Omega}|\nabla\cdot\de u^\ep|^2dx\\
&\leq\ C_1(t)\left(\int_{\Omega}\rho^\ep|\de u^\ep|^2dx\,+\,
\frac{1}{\ep^2}\int_{\Omega}E(\rho^\ep,\rho^\ep_{app})\,dx\right)+\,\ep\, C_2(t) \\& \qquad+\left(\tfrac{\min(\mu,\lambda)}{10} + C \ep^{(\frac{5}{2}-\frac3\gamma)_-}\right)\|\nabla_h \de u^\ep\|^2_{L^2}\,+\, C_3\left(\delta\,\ep + \ep^{1_+}\right)\|\partial_3 \de u^\ep\|^2_{L^2}\\&\qquad +\,\left(\tfrac{\min(\mu,\lambda)}{10} + C\ep^\frac12\right)\|\partial_3 \de u^\ep_3\|^2_{L^2} .
\end{aligned}
\end{equation}
Let us stress that $C_1(t),\ C_2(t),\ C_3$ and $\ep_0$ do not depend on $\ep$. The quantities these constants depend on have been written explicitly in the computations above;
in particular, $C_1(t)$ and $C_2(t)$ contains the functions $K_1(t)$ and $K_2(t)$. \\
Choosing $\delta$ small enough and using the identity $\partial_3 \delta u_3^\ep= \nabla\cdot \delta u^\ep - \nabla_h \cdot \delta u^\ep_h$, the last three terms in \eqref{e.entroestfinal} can be swallowed in the left hand side. The estimate in Theorem \ref{prop.main} follows from Gr\"onwall's lemma.

\section*{Acknowledgement}

The three authors are partially supported by the project BORDS (ANR-16-CE40-0027-01) operated by the French National Research Agency (ANR).
The second and third authors are partially supported by the project SingFlows (ANR-18-CE40-0027) operated by the French National Research Agency (ANR).
The work of the second author is also partially supported by the LABEX MILYON (ANR-10-LABX-0070) of Universit\'e de Lyon, within the program ``Investissement d'Avenir'' (ANR-11-IDEX-0007).
The first and the third authors acknowledge financial support from 
the IDEX of the University of Bordeaux for the BOLIDE project. The first author is partially supported by the Starting Grant project \textquotedblleft Analysis of moving incompressible fluid interfaces\textquotedblright (H2020-EU.1.1.-639227) operated by the European Research Council (ERC).

\end{document}